\documentclass[10pt,a4paper,oneside]{amsart}
\usepackage[a4paper]{geometry} % Changing page shape
\newcounter{commentcounter}

\usepackage{amsmath} % Lots of maths functionality
\usepackage{amssymb} % Maths symbols
\usepackage{amsthm} % Maths environments: \begin{proof}, etc.
\usepackage{stmaryrd} % [[ brackets
\usepackage[english]{babel} % Language and hyphenation.
\usepackage[font=small,justification=centering]{caption} % More flexibility for captioning figures
\usepackage[nodayofweek]{datetime}
\usepackage[shortlabels]{enumitem} % Change enumeration labelling with \begin{enumerate}[a)] etc.
\usepackage[T1]{fontenc} % For font encoding, to allow accents, copy & paste, inequality signs, etc. to all work nicely.
\usepackage[utf8]{inputenc} % To be loaded after fontenc, also for encoding.
\usepackage{ifthen} % For Dani's \begin{com} environment and \numberedtheorem
\usepackage{mathabx} % Contains \pnest symbol and more. Clashes with accents for \ring
\usepackage{mathtools} % Uses amsmath, fixes quirks and adds functionality.
\usepackage[dvipsnames,table]{xcolor} % Allow links to have colour. Needs to be before hyperref and tikz-cd
\usepackage[pdftex,  colorlinks=true, pagebackref=true]{hyperref} % Makes references and citations into links
    \hypersetup{urlcolor=RoyalBlue, linkcolor=RoyalBlue,  citecolor=black}
\usepackage{setspace} % Allows \onehalfspacing etc. for changing gaps between lines
\usepackage{tikz-cd} % Commutative diagrams
\usepackage{xfrac} % Nicer quotients 
\usepackage[capitalize]{cleveref} % use \Cref{} for instead of X~\ref{} 
\usepackage{multirow} \usepackage{multicol}\usepackage{tabularx}
\renewcommand*{\backref}[1]{}
\renewcommand*{\backrefalt}[4]
{
    \ifcase #1
        No citation in the text.
    \or
        Cited on Page #2.
    \else
        Cited on Pages #2.
    \fi
}

\makeatletter
\def\@tocline#1#2#3#4#5#6#7{\relax
  \ifnum #1>\c@tocdepth % then omit
  \else
    \par \addpenalty\@secpenalty\addvspace{#2}%
    \begingroup \hyphenpenalty\@M
    \@ifempty{#4}{%
      \@tempdima\csname r@tocindent\number#1\endcsname\relax
    }{%
      \@tempdima#4\relax
    }%
    \parindent\z@ \leftskip#3\relax \advance\leftskip\@tempdima\relax
    \rightskip\@pnumwidth plus4em \parfillskip-\@pnumwidth
    #5\leavevmode\hskip-\@tempdima
      \ifcase #1
       \or\or \hskip 1em \or \hskip 2em \else \hskip 3em \fi%
      #6\nobreak\relax
    \dotfill\hbox to\@pnumwidth{\@tocpagenum{#7}}\par% <---- \dotfill -> \hfill
    \nobreak
    \endgroup
  \fi}
\makeatother
\makeatletter
\@namedef{subjclassname@1991}{Mathematical subject classification 1991}
\@namedef{subjclassname@2000}{Mathematical subject classification 2000}
\@namedef{subjclassname@2010}{Mathematical subject classification 2010}
\@namedef{subjclassname@2020}{Mathematical subject classification 2020}
\makeatother

\newtheorem{thm}{Theorem}[section]
\newtheorem{lemma}[thm]{Lemma}
\newtheorem{corollary}[thm]{Corollary}
\newtheorem{prop}[thm]{Proposition}

\AddToHook{env/lemma/begin}{\crefalias{thm}{lemma}}
\AddToHook{env/corollary/begin}{\crefalias{thm}{corollary}}
\AddToHook{env/prop/begin}{\crefalias{thm}{prop}}
\AddToHook{env/conjecture/begin}{\crefalias{thm}{conjecture}}
\AddToHook{env/claim/begin}{\crefalias{thm}{claim}}

\crefname{thm}{Theorem}{Theorems}
\crefname{lemma}{Lemma}{lemmas}
\crefname{prop}{Proposition}{Propositions}
\crefname{conjecture}{Conjecture}{Conjectures}
\crefname{claim}{Claim}{Claims}

\newtheorem{thmx}{Theorem}

\AddToHook{env/corx/begin}{\crefalias{corx}{thmx}}

\crefname{thmx}{Theorem}{Theorems}
\crefname{corx}{Corollary}{Corollaries}

\theoremstyle{definition}
\newtheorem{defn}[thm]{Definition}
\newtheorem{remark}[thm]{Remark}

\newtheorem{question}[thm]{Question}

\newtheorem{notation}[thm]{Notation}
\newtheorem{example}[thm]{Example}

\AddToHook{env/defn/begin}{\crefalias{thm}{defn}}
\AddToHook{env/remark/begin}{\crefalias{thm}{remark}}
\AddToHook{env/remarks/begin}{\crefalias{thm}{remarks}}
\AddToHook{env/quest/begin}{\crefalias{thm}{quest}}
\AddToHook{env/example/begin}{\crefalias{thm}{example}}
\AddToHook{env/notation/begin}{\crefalias{thm}{notation}}

\crefname{defn}{Definition}{Definitions}
\crefname{remark}{Remark}{Remarks}
\crefname{remarks}{Remarks}{Remarks}
\crefname{quest}{Question}{Questions}
\crefname{example}{Example}{Examples}
\crefname{notation}{Notation}{Notations}

\AddToHook{env/examplex/begin}{\crefalias{thmx}{examplex}}
\crefname{examplex}{Example}{Examples}

    \newtheoremstyle{TheoremNum}
        {8.0pt plus 2.0pt minus 4.0pt}{8.0pt plus 2.0pt minus 4.0pt} %%% space between body and thm
        {\itshape} %%% Thm body font
        {-0.15cm} %%% Indent amount (empty = no indent)
        {\bfseries} %%% Thm head font
        {.} %%% Punctuation after thm head
        { }  %%% Space after thm head
        {\thmname{#1}\thmnote{ \bfseries #3}}%%% Thm head spec
    \theoremstyle{TheoremNum}
    \newtheorem{duplicate}{}

\newcommand*{\claimproofname}{My proof}

\DeclareMathOperator{\Aut}{\mathrm{Aut}}
\DeclareMathOperator{\Out}{\mathrm{Out}}

\DeclareMathOperator{\im}{\mathrm{im}}

\DeclareMathOperator{\lk}{lk}

\DeclareMathOperator{\sym}{Sym}
\DeclareMathOperator{\alt}{Alt}

\DeclareMathOperator{\Stab}{Stab}

\newcommand{\cala}{{\mathcal{A}}}
\newcommand{\calb}{{\mathcal{B}}}
\newcommand{\calc}{{\mathcal{C}}}
\newcommand{\cald}{{\mathcal{D}}}
\newcommand{\cale}{{\mathcal{E}}}
\newcommand{\calf}{{\mathcal{F}}}
\newcommand{\calg}{{\mathcal{G}}}
\newcommand{\calh}{{\mathcal{H}}}
\newcommand{\cali}{{\mathcal{I}}}

\newcommand{\caln}{{\mathcal{N}}}

\newcommand{\calx}{{\mathcal{X}}}

\newcommand{\ab}{{\rm ab}}

\newcommand{\GL}{\mathrm{GL}}
\newcommand{\PGL}{\mathrm{PGL}}

\newcommand{\SL}{\mathrm{SL}}
\newcommand{\PSL}{\mathrm{PSL}}

\newcommand{\PGU}{\mathrm{PGU}}

\newcommand{\SU}{\mathrm{SU}}
\newcommand{\PU}{\mathrm{PU}}
\newcommand{\PSU}{\mathrm{PSU}}
\newcommand{\Sp}{\mathrm{Sp}}
\newcommand{\PSp}{\mathrm{PSp}}
\newcommand{\POmega}{\mathrm{P}\Omega}

\newcommand{\SO}{\mathrm{SO}}
\newcommand{\PSO}{\mathrm{PSO}}
\newcommand{\PGO}{\mathrm{PGO}}
\newcommand{\AGL}{\mathrm{AGL}}

\newcommand{\He}{\mathrm{He}}
\newcommand{\GO}{\mathrm{GO}}

\DeclareMathOperator{\soc}{Soc}

\newcommand{\rightQ}[2]{\left.\raisebox{.2em}{$#1$}\middle/\raisebox{-.2em}{$#2$}\right.}

\DeclareMathOperator{\id}{id}

\newcommand{\st}{\mathrm{st}}
\newcommand{\EE}{\mathbb{E}}
\DeclareMathOperator{\cd}{\mathrm{cd}}

\def\Z{\mathbb{Z}}

\newcommand{\QQ}{\mathbb{Q}}
\newcommand{\FF}{\mathbb{F}}

\usepackage{tikz}
\usetikzlibrary{arrows,quotes}
\tikzstyle{blackNode}=[fill=black, draw=black, shape=circle]

\renewcommand{\bar}{\overline}

\title[Finite quotients of Artin groups]{Spherical and affine Artin groups, finite quotients,\\ and the isomorphism problem}

\author{Sam Hughes}
\address[S. Hughes (current)]{Max-Planck-Institut für Mathematik, Vivatsgasse 7, 53111 Bonn, Germany}
\address[S.~Hughes (former)]{Rheinische Friedrich-Wilhelms-Universit\"at Bonn, Mathematical Institute, Endenicher Allee 60, 53115 Bonn, Germany}
\email{sam.hughes.maths@gmail.com}\email{hughess@mpim-bonn.mpg.de}
\urladdr{https://samhughesmaths.github.io}

\author{Thomas Ng}
\address[T.~Ng]{Department of Mathematics and Statistics, Haverford College, Haverford, PA. USA}
\email{thomas.ng.math@gmail.com}
\urladdr{https://sites.google.com/site/thomasng192/}

\author{Kaitlin Ragosta}
\address[K.~Ragosta]{Euskal Herriko Unibertsitatea}
\email{kaitlin.ragosta@ehu.eus}
\urladdr{https://sites.google.com/view/kaitlinragostahughes/}

\author{Nancy Scherich}
\address[N.~Scherich]{Elon University}
\email{nscherich@elon.edu}
\urladdr{https://nancyscherich.com/}

\author{Yvon Verberne}
\address[Y.~Verberne]{Department of Mathematics, Middlesex College (MC) 255, Western University, London, Ontario, Canada}
\email{yverber@uwo.ca}
\urladdr{https://sites.google.com/view/yvonverberne}

\date{23rd October 2026}

\begin{document}
\begin{abstract}
We compute the smallest non-abelian quotients of spherical and affine Artin groups. Usually, these are isomorphic to the smallest non-abelian quotients of the corresponding Coxeter groups. As applications, we deduce that two affine Artin groups are isomorphic if and only if their Coxeter diagrams are isomorphic and that irreducible spherical Artin groups are distinguished from each other by their finite quotient groups.
\end{abstract}
\maketitle

\section{Introduction}\label{sec:intro}
Artin introduced braid groups in \cite{Artin1947}, and later these were generalised by Tits \cite{Tits66} to the class of Artin groups.  For any Coxeter diagram $S$, see \Cref{S:Artin-basics} for a definition, Tits calls an Artin group $A_S$ the extended Coxeter group of the Coxeter group $W_S$ and equips $A_S$ with a natural epimorphism to $W_S$.  The systematic study of Artin groups began with the works \cite{Brieskorn1971, Deligne1972, BrieskornSaito72, Brieskorn1973} and \cite{Salvetti1987} with much research focusing on the \emph{spherical Artin groups}: the Artin groups $A_S$ where $W_S$ is a finite group.  The Coxeter diagrams of finite Coxeter groups were classified by Coxeter \cite{Coxeter1935} and are given in \Cref{fig:FiniteCoxeterGroups}.  For notational clarity, the classical $n$-strand braid group is denoted by $\cala_{n-1}$ to be distinguished from the Artin group $\calb_n$, the dihedral Artin groups are denoted $\cali_2(n)$, and the dihedral Coxeter groups are denoted $Dh_n$ to distinguish it from the Artin group $\cald_n$.  
We often denote the canonical Coxeter quotient of an Artin group $G$ by $W(G)$. 

Much is known about the spherical Artin groups.  The $K(\pi,1)$ conjecture holds \cite{CharneyDavis95,BradyWatt02}, the groups are linear (hence residually finite \cite{Malcev1940}); see \cite{Kramer1912,Bigelow01} for the braid group and \cite{CohenWales02,Digne03} for the other groups; and they are classified up to isomorphism \cite{Paris2004}.  Moreover, the geometry of these groups has been extensively studied but remains mysterious \cite{Tits66, Deligne1972, Charney95, CumplidoGebhardtGonzalesMenesesWiest17, HuangOsajda21, CalvezWiest24, Ragosta25}.

The \emph{affine Artin groups}, i.e., Artin groups where the corresponding Coxeter group is affine, are much more mysterious.  The Coxeter diagrams of affine Coxeter groups were classified by Coxeter \cite{Coxeter34} and are given in \Cref{fig:affineCoxeterGraphs}.  The word problem \cite{McCamondSulway2017} and the $K(\pi,1)$ conjecture \cite{PaoliniSalvetti21} for these groups was resolved recently.  It appears to be unknown if every affine Artin group is linear or even residually finite.   Our first main result classifies the affine Artin groups up to isomorphism.

\begin{duplicate}[\Cref{thmx:iso}]
        The isomorphism problem is solvable amongst all spherical and affine Artin groups: if $\Gamma$ and $\Lambda$ denote spherical or affine Coxeter diagrams, then $A_\Gamma\cong A_\Lambda$ if and only if $\Gamma\cong \Lambda$.
\end{duplicate}

The earliest reference we could find explicitly asking about the isomorphism problem for general Artin groups is \cite[Problems 26-28]{Charney_probs}, although it is alluded to in \cite[Problem~2.1]{BradyMcCammondMuhlerrNeumann2002}. The isomorphism problem is understood when $\Gamma$ is either right-angled or large-type and $\Lambda$ is any Coxeter diagram \cite{Baudisch, Droms87, Vaskou23, MartinVaskou24}, and it is solved when $\Gamma$ and $\Lambda$ are both spherical Coxeter diagrams \cite{Paris2004}. 

% \yv{Should we cite Charney's problem paper (\url{https://people.brandeis.edu/~charney/papers/Artin_probs.pdf}), Muhlherr "The isomorphism problem for Coxeter groups" (\url{https://arxiv.org/pdf/math/0506572}), or Cohen "Coxeter groups and three related topics" (\url{https://ir.cwi.nl/pub/2303/2303D.pdf}); this could highlight that the problem has been open for a long time and provide some additional context to the reader}

\subsection{Finite quotients}
The proof of \Cref{thmx:iso} requires some work.  First we compute the smallest non-cyclic quotient of an irreducible spherical Artin group (\Cref{thmx:spherical}) and then apply this to compute the smallest non-cyclic quotient of an irreducible affine Artin group (\Cref{thmx:affine}).  These results and some cohomological input then completes the proof of \Cref{thmx:iso}.

\begin{duplicate}[\Cref{thmx:spherical}]
    Let $G$ be an irreducible spherical Artin group. The following conclusions hold:
    \begin{enumerate}
        \item If $G$ is isomorphic to $\cala_{n-1}$, $\calb_n$, or $\cald_n$, for $n\geq 5$, then the smallest non-abelian quotient of $G$ is $\sym(n)$ and it is unique up to automorphisms of the image. ($\cala_{n-1}$ is due to Kolay \cite{Kolay23})
        \item If $G$ is isomorphic to one of $\cala_2$ $\cala_3$, $\calb_3$, $\calb_4$, $\cald_4$, $\calf_4$, or $\cali_2(m)$ with $4|m$, then the smallest non-abelian quotient of $G$ is $\sym(3)$.
        \item If $G$ is isomorphic to $\cali_2(m)$ with $4\nmid m$ then the smallest non-abelian quotient of $G$ is the dihedral group $Dh_p$, where $p$ is the smallest odd prime divisor of $m$, and it is unique up to automorphisms of the image.
        \item If $G$ is isomorphic to $\calh_3$, then the smallest non-abelian quotient of $G$ is $\alt(5)$.  Up to automorphisms of the image, there are two such homomorphisms.
        \item If $G$ is isomorphic to $\calh_4$, $\cale_6$, $\cale_7$, or $\cale_8$, then the smallest non-abelian quotient of $G$ is isomorphic to $W(G)/Z(W(G))$ and is unique up to automorphisms of the image.
    \end{enumerate}
\end{duplicate}

\begin{figure}[htb]
	\begin{center}
	\captionsetup{justification=centering}
		\begin{tikzpicture}[scale=0.8]
			\draw[fill=black]  (0,0) circle (2pt);
			\draw[fill=black]  (1,0) circle (2pt);
            \draw[fill=black]  (2,0) circle (2pt);
			\draw (0,0)--(1,0);
            \draw[dashed] (1,0)--(2,0);
            \draw[fill=black] (3,0) circle (2pt);
            \draw (2,0)--(3,0);
			\node at (-0.9,0) {$\underset{n\geq 1}{\cala_n}$};
            \node at (0, -0.4) {$s_n$};
            \node at (1, -0.4) {$s_{n-1}$};
            \node at (2, -0.4) {$s_{2}$};
            \node at (3, -0.4) {$s_{1}$};
            \draw[fill=black]  (0,-1.5) circle (2pt);
			\draw[fill=black]  (1,-1.5) circle (2pt);
            \draw[fill=black]  (2,-1.5) circle (2pt);
			\draw (0,-1.5)--(1,-1.5);
            \draw[dashed] (1,-1.5)--(2,-1.5);
            \draw[fill=black] (3,-1.5) circle (2pt);
            \draw (2,-1.5)--(3,-1.5);
            \node at (2.5, -1.3) {$4$};
            \node at (-0.9,-1.5) {$\underset{n\geq 3}{\calb_n}$};
            \node at (0, -1.9) {$s_n$};
            \node at (1, -1.9) {$s_{n-1}$};
            \node at (2, -1.9) {$s_{2}$};
            \node at (3, -1.9) {$s_{1}$};

            \draw[fill=black]  (0,-3.5) circle (2pt);
			\draw[fill=black]  (1,-3.5) circle (2pt);
            \draw[fill=black]  (2,-3.5) circle (2pt);
			\draw (0,-3.5)--(1,-3.5);
            \draw[dashed] (1,-3.5)--(2,-3.5);
            \draw[fill=black] (3,-2.8) circle (2pt);
            \draw[fill=black] (3,-4.2) circle (2pt);
            \draw (2,-3.5)--(3, -2.8);
            \draw (2,-3.5)--(3, -4.2);
            \node at (-0.9,-3.5) {$\underset{n\geq 4}{\cald_n}$};
            \node at (0, -3.9) {$s_n$};
            \node at (1, -3.9) {$s_{n-1}$};
            \node at (2, -3.9) {$s_{3}$};
            \node at (3.4, -2.8) {$s_{2}$};
            \node at (3.4, -4.2) {$s_{1}$};

            \draw (0,-5.5)--(4,-5.5);
            \draw (2,-4.5)--(2,-5.5);
            \draw[fill=black] (0,-5.5) circle (2pt); 
            \draw[fill=black] (1,-5.5) circle (2pt);
            \draw[fill=black] (2,-5.5) circle (2pt);
            \draw[fill=black] (3,-5.5) circle (2pt);
            \draw[fill=black] (4,-5.5) circle (2pt);
            \draw[fill=black] (2,-4.5) circle (2pt);
            \node at (-0.5,-5.5) {$\cale_6$};
            \node at (0, -5.9) {$s_1$};
            \node at (1, -5.9) {$s_{3}$};
            \node at (2, -5.9) {$s_{4}$};
            \node at (3, -5.9) {$s_{5}$};
            \node at (4, -5.9) {$s_{6}$};
            \node at (1.6, -4.5) {$s_{2}$};

            \draw (0,-7.5)--(5,-7.5);
            \draw (2,-6.5)--(2,-7.5);
            \draw[fill=black] (0,-7.5) circle (2pt); 
            \draw[fill=black] (1,-7.5) circle (2pt);
            \draw[fill=black] (2,-7.5) circle (2pt);
            \draw[fill=black] (3,-7.5) circle (2pt);
            \draw[fill=black] (4,-7.5) circle (2pt);
            \draw[fill=black] (5,-7.5) circle (2pt);
            \draw[fill=black] (2,-6.5) circle (2pt);
            \node at (-0.5,-7.5) {$\cale_7$};
            \node at (0, -7.9) {$s_1$};
            \node at (1, -7.9) {$s_{3}$};
            \node at (2, -7.9) {$s_{4}$};
            \node at (3, -7.9) {$s_{5}$};
            \node at (4, -7.9) {$s_{6}$};
            \node at (5, -7.9) {$s_{7}$};
            \node at (1.6, -6.5) {$s_{2}$};

            \draw[fill=black] (8,0) circle (2pt);
            \draw[fill=black] (9,0) circle (2pt);
            \draw[fill=black] (10,0) circle (2pt);
            \draw[fill=black] (11,0) circle (2pt);
            \draw (8,0)--(11,0);
            \node at (9.5, 0.2) {$4$};
            \node at (7.5,0) {$\calf_4$};
            \node at (8, -0.4) {$s_{4}$};
            \node at (9, -0.4) {$s_{3}$};
            \node at (10, -0.4) {$s_{2}$};
            \node at (11, -0.4) {$s_{1}$};

            %\draw[fill=black] (8,-1.5) circle (2pt);
            %\draw[fill=black] (9,-1.5) circle (2pt);
            %\draw (8,-1.5)--(9,-1.5);
            %\node at (8.5,-1.3)  {$6$};
            %\node at (7.5,-1.5) {$G_2$};

            \draw[fill=black] (8,-1.5) circle (2pt);
            \draw[fill=black] (9,-1.5) circle (2pt);
            \draw[fill=black] (10,-1.5) circle (2pt);
            \draw (8,-1.5)--(10,-1.5);
            \node at (9.5, -1.3){$5$};
            \node at (7.5,-1.5) {$\calh_3$};
            \node at (10.1, -1.8) {$s_1$};
            \node at (9, -1.8) {$s_2$};
            \node at (8, -1.8) {$s_3$};

            \draw[fill=black] (8,-3) circle (2pt);
            \draw[fill=black] (9,-3) circle (2pt);
            \draw[fill=black] (10,-3) circle (2pt);
            \draw[fill=black] (11,-3) circle (2pt);
            \draw (8,-3)--(11, -3);
            \node at (10.5, -2.8){$5$};
            \node at (7.5,-3) {$\calh_4$};
            \node at (11.1, -3.3) {$s_1$};
            \node at (10.1, -3.3) {$s_2$};
            \node at (9, -3.3) {$s_3$};
            \node at (8, -3.3) {$s_4$};

            \draw[fill=black] (8,-4.5) circle (2pt);
            \draw[fill=black] (9,-4.5) circle (2pt);
            \draw (8,-4.5)--(9,-4.5);
            \node at (8.5, -4.3){$m$};
            \node at (7.3,-4.5) {$\underset{m\geq 4}{\cali_2(m)}$};
            \node at (8, -4.9) {$s_{2}$};
            \node at (9, -4.9) {$s_{1}$};
            %\node at (7.3, -7){ $m\geq3$};

            \draw (8,-7)--(14,-7);
            \draw (10,-6)--(10,-7);
            \draw[fill=black] (8,-7) circle (2pt); 
            \draw[fill=black] (9,-7) circle (2pt);
            \draw[fill=black] (10,-7) circle (2pt);
            \draw[fill=black] (11,-7) circle (2pt);
            \draw[fill=black] (12,-7) circle (2pt);
            \draw[fill=black] (13,-7) circle (2pt);
            \draw[fill=black] (14,-7) circle (2pt);
            \draw[fill=black] (10,-6) circle (2pt);
            \node at (7.5,-7) {$\cale_8$};
            \node at (8.1,-7.4) {$s_1$};
            \node at (9,-7.4) {$s_3$};
            \node at (9.6,-6) {$s_2$};
            \node at (10,-7.4) {$s_4$};
            \node at (11,-7.4) {$s_5$};
            \node at (12,-7.4) {$s_6$};
            \node at (13,-7.4) {$s_7$};
            \node at (14,-7.4) {$s_8$};
				
		\end{tikzpicture}
	\caption{Irreducible Coxeter diagrams of spherical type $X_n$, where $X_n$ has $n$ vertices.}
    \label{fig:FiniteCoxeterGroups}
	\end{center}
\end{figure}

Whilst \Cref{thmx:spherical} is a key tool in proving Theorems~\ref{thmx:iso} and \ref{thmx:affine}, the result has broader appeal. The representation theory of braid groups has been extensively studied from different perspectives; see for instance \cite{JONES, Formanek1996, AdemCohenCohen2003, Marin2013, PalmerSoulie2025}, or \cite{Margalit2019} for a survey. One important avenue has been to use local representations of the braid group to give  topological models for quantum computation \cite{TQC, Marin2013}; see \cite{DelaneyRowellWang2016} for a survey.  Even the simplest representations---those factoring through finite quotient groups---have been intensively studied \cite{Coxeter1959, BirmanWajnryb1986, Wajnryb1991, EtingofRowellWitherspoon2008, LarsenRowell2008, Caplinger2023, KordekLiPartin2021, ChudnovskyKordekLiPartin20, SV2, BloomquistPatztS2024}.  This culminated with Kolay's \cite{Kolay23} positive resolution of a longstanding question of Margalit asking if the symmetric group $\sym(n)$ is the smallest non-abelian quotient of the braid group $\cala_{n-1}$ for $n\geq 5$.  Similar results were earlier established for $\Aut(F_N)$ \cite{BaumeisterKielakPierro2019}, mapping class groups \cite{KielakPierro2020}, and surface braid groups \cite{Tan2024}.
%We use the classification of spherical Artin groups listed in \Cref{fig:FiniteCoxeterGroups} and the classification of affine Artin groups listed in \Cref{fig:affineCoxeterGraphs}. 

The techniques used in proving \Cref{thmx:spherical} are similar in spirit to \cite{BaumeisterKielakPierro2019, KielakPierro2020}, but the implementation and specific tricks we require are rather different.   We also highlight that many of our results in the body of the paper give much more structural information about quotients of spherical Artin groups than the theorem above (we refer the reader to \Cref{sec:proof_strat} for more information).

\begin{duplicate}[\Cref{thmx:affine}] 
        Let $G$ be an irreducible affine Artin group.  The following conclusions hold:
    \begin{enumerate}
        \item If $G$ is isomorphic to $\widetilde\cala_{n-1}$, $\widetilde\calb_n$, $\widetilde\calc_n$ or $\widetilde\cald_n$, for $n\geq 5$, then the smallest non-abelian quotient of $G$ is $\sym(n)$.
        \item If $G$ is isomorphic to one of $\widetilde\cala_1$, $\widetilde\cala_2$ $\widetilde\cala_3$, $\widetilde\calb_2$, $\widetilde\calb_3$, $\widetilde\calb_4$, $\widetilde\cald_4$, $\widetilde\calf_4$, or $\widetilde\calg_2$, then the smallest non-abelian quotient of $G$ is $\sym(3)$.
        \item If $G$ is isomorphic to $\widetilde\cale_n$ for $n=6,7,8$ then the smallest non-abelian quotient of $G$ is isomorphic to $W(\cale_n)/Z(W(\cale_n))$.
    \end{enumerate}
\end{duplicate}

In most of the cases of Theorems~\ref{thmx:spherical} and \ref{thmx:affine}, we see that the smallest non-abelian quotient of $A_\Gamma$ is that of the Coxeter group $W_\Gamma$.  This fails in general, however. For example, it fails in the case of $\cali_2(4m)$ with $m$ not divisible by $3$. 

\begin{question}
    For which Artin groups $A_\Gamma$ is the smallest non-abelian quotient of $A_\Gamma$ isomorphic to the smallest non-abelian quotient of the Coxeter group $W_\Gamma$?
\end{question}

\begin{figure}[h]
	\begin{center}
	\captionsetup{justification=centering}
		\begin{tikzpicture}[scale=0.8, transform shape]
            % A_1
			\draw[fill=black]  (0,0) circle (2pt);
			\draw[fill=black]  (1,0) circle (2pt);
			\draw (0,0)--(1,0);
			\node at (0.5,0.2){$\infty$};
			\node at (-0.9, 0) {$\widetilde A_1$};

            % A_n
            \draw[fill=black] (0,-1.5) circle (2pt);
            \draw[fill=black] (1,-1.5) circle (2pt);
            \draw (0,-1.5)--(1,-1.5);
            \draw[fill=black] (2,-1.5) circle (2pt);
            \draw[fill=black] (3,-1.5) circle (2pt);
            \draw[dashed] (1,-1.5)--(2, -1.5);
            \draw (2,-1.5)--(3,-1.5);
            \draw[fill=black] (1.5, -0.8) circle (2pt);
            \draw (0,-1.5)--(1.5, -0.8);
            \draw (3,-1.5)--(1.5, -0.8);       
            \node at (-0.9, -1.5) {$\underset{n\geq 2}{\widetilde A_n}$};	

            % B_n
            \draw[fill=black] (0,-3) circle (2pt);
            \draw[fill=black] (1,-3) circle (2pt);
            \node at (0.5,-2.8) {$4$};
            \draw (0,-3)--(1,-3);
            \draw (1,-3)--(2,-3);
            \draw[dashed] (2,-3)--(3,-3);
            \draw[fill=black] (3,-3) circle (2pt);
            \draw[fill=black] (2,-3) circle (2pt);
            \draw[fill=black] (4, -2.3) circle (2pt);
            \draw[fill=black] (4, -3.7) circle (2pt);
            \draw (3, -3)--(4,-2.3);
            \draw (3,-3)--(4, -3.7);           
            \node at (-0.9, -3) {$\underset{n\geq 3}{\widetilde B_n}$};	

            %C_n
            \draw[fill=black] (0,-4.5) circle (2pt);
            \draw[fill=black] (1,-4.5) circle (2pt);
            \draw (0,-4.5)--(1, -4.5);
            \node at (0.5, -4.3) {$4$};
            \draw[fill=black] (2, -4.5) circle (2pt);
            \draw (1, -4.5)--(2,-4.5);
            \draw[dashed] (2,-4.5)--(3,-4.5);
            \draw[fill=black] (3,-4.5) circle (2pt);
            \draw[fill=black] (4,-4.5) circle (2pt);
            \draw (3,-4.5)--(4,-4.5);
            \node at (3.5, -4.3) {$4$};
            \node at (-0.9, -4.5) {$\underset{n\geq 2}{\widetilde C_n}$};

            %D_n
            \draw[fill=black] (0,-5.3) circle (2pt);
            \draw[fill=black] (0,-6.7) circle (2pt);
            \draw[fill=black] (1,-6) circle (2pt);
            \draw (0,-5.3)--(1,-6);
            \draw (0,-6.7)--(1,-6);
            \draw[fill=black] (2,-6) circle (2pt);
            \draw (1,-6)--(2,-6);
            \draw[dashed] (2,-6)--(3,-6);
            \draw[fill=black] (3,-6) circle (2pt);
            \draw[fill=black] (4, -5.3) circle (2pt);
            \draw[fill=black] (4, -6.7) circle (2pt);
            \draw (3,-6)--(4, -5.3);
            \draw (3,-6)--(4, -6.7);       
            \node at (-0.9, -6) {$\underset{n\geq 4}{\widetilde D_n}$};

            %E_6
            \draw[fill=black] (6,0) circle (2pt);
            \draw[fill=black] (7,0) circle (2pt);
            \draw[fill=black] (8,0) circle (2pt);
            \draw[fill=black] (9,0) circle (2pt);
            \draw[fill=black] (10,0) circle (2pt);
            \draw[fill=black] (8,1) circle (2pt);
            \draw[fill=black] (8,2) circle (2pt);
            \draw (6,0)--(10,0);
            \draw (8,0)--(8,2);
            \node at (5.5,0) {$\widetilde E_6$};

            %E_7
            \draw[fill=black] (6, -1.5) circle (2pt);
            \draw[fill=black] (7, -1.5) circle (2pt);
            \draw[fill=black] (8, -1.5) circle (2pt);
            \draw[fill=black] (9, -1.5) circle (2pt);
            \draw[fill=black] (10, -1.5) circle (2pt);
            \draw[fill=black] (11, -1.5) circle (2pt);
            \draw[fill=black] (12, -1.5) circle (2pt);
            \draw[fill=black] (9, -0.5) circle (2pt);
            \draw (6,-1.5)--(12, -1.5);
            \draw (9, -1.5)--(9, -0.5);
            \node at (5.5,-1.5) {$\widetilde E_7$};

            %E_8
            \draw[fill=black] (6, -3) circle (2pt);
            \draw[fill=black] (7, -3) circle (2pt);
            \draw[fill=black] (8, -3) circle (2pt);
            \draw[fill=black] (9, -3) circle (2pt);
            \draw[fill=black] (10, -3) circle (2pt);
            \draw[fill=black] (11, -3) circle (2pt);
            \draw[fill=black] (12, -3) circle (2pt);
            \draw[fill=black] (13, -3) circle (2pt);
            \draw[fill=black] (8, -2) circle (2pt);
            \draw (6,-3)--(13, -3);
            \draw (8, -3)--(8, -2);
            \node at (5.5,-3) {$\widetilde E_8$};

            %F_4
            \draw[fill=black] (6, -4.5) circle (2pt);
            \draw[fill=black] (7, -4.5) circle (2pt);
            \draw[fill=black] (8, -4.5) circle (2pt);
            \draw[fill=black] (9, -4.5) circle (2pt);
            \draw[fill=black] (10, -4.5) circle (2pt);
            \draw (6, -4.5)--(10, -4.5);
            \node at (8.5, -4.3) {$4$};
            \node at (5.5,-4.5) {$\widetilde F_4$};

            %G_2
            \draw[fill=black] (6, -6) circle (2pt);
            \draw[fill=black] (7, -6) circle (2pt);
            \draw[fill=black] (8, -6) circle (2pt);
            \draw (6,-6)--(8, -6);
            \node at (7.5, -5.8) {$6$};
            \node at (5.5,-6) {$\widetilde G_2$};   
		\end{tikzpicture}
	\caption{Irreducible Coxeter diagram of affine type $\widetilde X_n$, where $\widetilde X_n$ has $n+1$ vertices.
 }\label{fig:affineCoxeterGraphs}
	\end{center}
\end{figure}

\subsection{Profinite rigidity}
Profinite rigidity, the study of groups via their finite quotients, has attracted a lot of interest in recent years.  See, for example, the ICM lecture notes of Alan Reid \cite{Reid2018} or the book of Wilkes \cite{Wilkes2024book}.  One of the main goals is to distinguish finitely generated residually finite groups up to isomorphism by their finite quotients.  There are only a few examples where such a result has been achieved: some hyperbolic $3$-manifold groups \cite{BridsonMcReynoldsReidSpitler2020}, some hyperbolic triangle groups \cite{BridsonMcReynoldsReidSpitler2021}, crystallographic groups in dimension at most 4 \cite{PiwekPopvicWilkes2021}, affine Coxeter groups \cite{CorsonHughesMoellerVarghese2025,PaoliniSklinos2024}, and certain lamplighter groups \cite{Blachar2025,NikolovWykowski2026}.  There are also many negative results, see for example \cite{BridsonGrunewald2004,Pyber2004,Bridson2016,Bridson2024,CorsonHughesMoellerVarghese2025}.  The relative problem of determining certain groups up to isomorphism within a fixed class of groups is much more widely studied, with most attention being paid to geometrically meaningful groups such as 3-manifold groups \cite{BridsonWilton2017,WiltonZalesskii2017,Wilkes2017,WiltonZalesskii2019,Liu2023,Xu2025,Xu2026}, free-by-cyclic groups \cite{HughesKudlinska2025, BridsonPiwek2025, AndrewHillenAlonzoPfaff2025, Nyberg-Brodda2026fbyz}, Coxeter groups \cite{CorsonHughesMoellerVarghese2025, CorsonHughesMoellerVarghese2026, HughesMoellerVarghese2025}, and right-angled Artin groups \cite{KrophollerWilkes2016,CorsonHughesMoellerVarghese2026,Liu26}.  While both braids group \cite{Nyberg-Brodda2026braid} and certain right-angled Artin groups \cite{CorsonHughesMoellerVarghese2026} fail to be profinitely rigid amongst all finitely generated residually finite groups, they may be profinitely rigid among all Artin groups.  

Beyond the intrinsic appeal of profinite rigidity theorems, the profinite completions of Artin groups, and specifically of braid groups, are closely related to Grothendieck's vision of anabelian geometry \cite{Grothendieck1997} and Drinfeld's construction of the group Grothendieck--Teichm\"uller group $\widehat{GT}$ \cite{Drinfeld1991}.  The latter group is intimately related with the Galois group $\mathrm{Gal}(\bar\QQ/\QQ)$ and conjecturally isomorphic to it.  In \cite[Theorem~B]{MinamideNakamura2022}, it is shown for $n\geq 4$ that $\Aut(\widehat\cala_n/Z(\widehat\cala_n))$ is isomorphic to $\widehat{GT}$.  For more information and recent developments we refer the reader to the surveys \cite{Schneps1997,Lochak2012}.  We remark that in \cite[Question~3]{Marin2012} it was asked whether the braid groups of complex reflection groups, a class which includes spherical Artin groups, can be distinguished by their profinite completions.  Our final result makes some new progress in this direction.  In the sequel we will denote by $\widehat G$ the profinite completion of a group $G$.

\begin{duplicate}[\Cref{thmx:profinite}]
    Let $G$ and $H$ be irreducible spherical Artin groups.  If $\widehat G\cong \widehat H$, then $G\cong H$.
\end{duplicate}
% \smallskip

We would like to recognize the concurrent and independent work \emph{Smallest quotients and profinite rigidity of irreducible spherical type Artin groups} by Gavazzi--Haladijan--Paris \cite{GavazziHalajianParis2026} who proved versions of \Cref{thmx:spherical} and  \Cref{thmx:profinite}. Our approaches for \Cref{thmx:spherical} in the cases of $\calb_n,$ $\cald_n,$ $\calf_4$ and $\cali_2(m)$ are similar.  Our approaches to \Cref{thmx:spherical} for the cases of $\cale_6,\cale_7,\cale_8, \calh_3$ and $\calh_4$ are very different.  Our approaches for \Cref{thmx:profinite} share the same broad strokes but the specific implementations are different.

%We would like to recognize the independent and concurrent work \emph{Smallest quotients and profinite rigidity of irreducible spherical type Artin groups} by Gavazzi--Haladijan--Paris who proved versions of \Cref{thmx:spherical} and  \Cref{thmx:profinite}. Our approaches for \Cref{thmx:spherical} for the cases of $\calb_n, \cald_n, \calf_4$ and $\cali_2(n)$ are similar, but our approaches to \Cref{thmx:spherical} for the cases of $\cale_6,\cale_7,\cale_8, \calh_3$ and $\calh_4$ are very different.  Our approaches for \Cref{thmx:profinite} share common broad strokes but the implementations are different.

\subsection{Proof strategy of Theorem~\ref{thmx:spherical}}\label{sec:proof_strat}
We first describe the proof strategy of \Cref{thmx:spherical} in the infinite families $\calb_n$, $\cald_n$ and $\cali_2(n)$.  The case of $\calb_n$ and $\cald_n$ are applications of Kolay's Theorem \cite{Kolay23} and some general facts about spherical Artin groups.  The remaining case of $\cali_2(n)$ involves consideration of the centre and the Orbit-Stabiliser Theorem.

The proof strategy of \Cref{thmx:spherical} in the case of $\cale_n$ for $n=6,7,8$ is very different.  
The first key observation is that the commutator subgroup of $\cale'_n$ is perfect \cite{Zinde1975} so every finite quotient is perfect-by-abelian.  Using this and facts about the socle (see \Cref{S:socle}) of a finite group we constrain the structure of possible finite quotient groups that can occur (\Cref{lem:min_quotient_En}).  Combining this structure description with the classification of finite simple groups of order less than $|W(\cale_n)/Z(\cale_n)|$ gives a list of candidate groups to rule out.  In the case of $\cale_8$, there are around 200 simple groups, the majority of which are isomorphic to $\PSL_2(q)$ for some natural number $q$.

The second key observation is that we can analyse the structure of totally symmetric sets in $\cale_n$ and $\cale_n'$; see \Cref{lem:En_comm_tss} and \Cref{lem:E7_comm_TSS}.  This analysis allows us to exclude all quotients with socle isomorphic to a product of $\PSL_2(q)$ for $\cale_7'$, $\cale_7$, and $\cale_8$.  We also introduce a new inductive method for excluding quotients isomorphic to subgroups of $\GL_n(q)$ using work of De Franceschi--Liebeck--O'Brien \cite{DeFranceschiLiebeckOBrien2025}; see \Cref{lem:induct_on_Anq}.

The third key observation is that in a finite quotient there must be elements of certain orders; see Lemmas \ref{lem:E6_elemOrders} and \Cref{lem:E7_elemOrders}.  The proof of this uses work of Soroko \cite{Soroko2021}, who identified generators of these torsion elements in $\cale_n/Z(\cale)$.  Our proof is essentially a direct computation but we introduce some new notation for working with equations in Artin groups which greatly simplifies the exposition.  

With these reductions in hand, we are left with a tractable list of finite groups.  Our general strategy is then to look at possible images for the generator $s_n$ of $\cale_n$ and  inductively study their centralisers.  In $\cale_6$, this amounts to the theorem of Kolay \cite{Kolay23} since $\cala_4\cong\langle s_1,s_3, s_2,s_4\rangle \leqslant \cale_6$ centralises $s_6$.   In $\cale_8$, we apply our results for $\cale_6$, since $\cale_6 \cong \langle s_1 , s_2, s_3, s_4, s_5 , s_6 \rangle$ centralizes $s_8$. %In each of the possible quotient groups we show  the centraliser must have incompatible structural properties with a quotient of $\cale_n$.

The verification for $\calh_3$ is by direct computation. The strategy for $\calh_4$ is similar to $\cale_n$ but considerably easier.  The remaining cases all admit $\sym(3)$ as a quotient, so there is nothing to check.

To prove \Cref{thmx:iso}, we begin by showing that it suffices to consider on irreducible affine Artin groups. 
Using \Cref{thmx:affine} and the first Betti numbers of irreducible affine Artin groups, we distinguish all irreducible affine Artin groups up to isomorphism.

% \smallskip

\subsection{Paper outline}\label{sec:paper_outline}
In \Cref{sec:prelims}, we recall some notions and notation from finite group theory, prove \Cref{lem:min_quotient_En} about the shapes of quotients, recall some generalities about totally symmetric sets, and state some results about parabolic subgroups in Artin groups.

In \Cref{sec:non_excpetional}, we handle the cases of $\calb_n$, $\cald_n$, and $\cali_2(n)$.

In \Cref{sec:Artin_props}, we recall a number of properties of spherical Artin groups.  We construct totally symmetric sets in $\cale_n$ and $\cale_n'$ for all $n$; see \Cref{lem:En_comm_tss}.  We prove a general result about collapsing sets in Artin groups; see \Cref{lem:artin-gen-collapsing}.  We also state the element order lemmas mentioned in \Cref{sec:proof_strat}; the computations which comprise the majority of their proofs are provided in \Cref{apx:computations}.

In \Cref{sec:f4} we handle the cases of $\calf_4$, $\calh_3$, and  $\calh_4$.  
In \Cref{sec:e6} we handle the case of $\cale_6$.  
In \Cref{sec:e7} we handle the case of $\cale_7$.  
In \Cref{sec:e8} we handle the case of $\cale_8$.  
In \Cref{sec:spherical} we prove \Cref{thmx:spherical}.
In \Cref{sec:affine} we prove \Cref{thmx:affine}.
In \Cref{sec:iso} we prove \Cref{thmx:iso}
In \Cref{sec:profinite} we prove \Cref{thmx:profinite}.

\tableofcontents

\subsection*{Acknowledgements}
SH thanks Nick Gill for helpful correspondence, Dawid Kielak for helpful conversations, and Alexandre Martin for patiently listening to him explain his ideas during the early stages of this project. KR thanks Carolyn Abbott for helpful conversations and mentorship during the early stages of this project.
%Versions of \Cref{thmx:spherical} and \Cref{thmx:profinite} were independently proven via different methods in work of Gavazzi--Haladijan--Paris titled \emph{Smallest quotients and profinite rigidity of irreducible spherical type Artin groups} posted concurrently with this article.

\smallskip

\footnotesize{SH was supported by the European Research Council (ERC)
under the European Union’s Horizon 2020 research and innovation programme (Grant
agreement no. 850930), by a Humboldt Research Fellowship at Universit\"at Bonn, and by the Deutsche Forschungsgemeinschaft (DFG, German Research Foundation) under Germany's Excellence Strategy - EXC-2047/1 - 390685813. 
TN was partially supported by ISF grant 660/20, an AMS--Simons travel grant, and a Zuckerman Fellowship at the Technion.
KR was supported by NSF grant DMS-2139752, NSF grant DMS-2106906, Basque Government grant IT1913-26, and PID2025-171325NB-I00 grant funded by MICIU/AEI/10.13039/501100011033 and FEDER/EU. 
NS was partially supported by the National Science Foundation grant DMS-2532699.
YV was partially supported by an NSERC-PDF Fellowship and by an NSERC Discovery grant RGPIN 05587.}
\normalsize

\subsection*{AI declaration} The authors did not use AI in the preparation of this article.

%%%%%%%%%%%%%%%%%%%%%%%%%%%%%%%%%%%%%%%%%%%%%%%%%%%%%%%%%%
\section{Preliminaries}\label{sec:prelims}

\subsection{Finite groups}\label{subsec:finite-groups}
We first establish some standard notation from finite group theory.  

\begin{notation}\label{StandardFiniteGroupsNotation}
A finite cyclic group of order $k$ is denoted by $k$. A direct product of $l$ copies of the cyclic group of order $k$ is denoted by $k^l$. 

A short exact sequence 
\[\begin{tikzcd}
    1 \ar[r] & N \ar[r] & G \ar[r] & Q \ar[r] & 1
\end{tikzcd}\]
of finite groups with unspecified extension will be denoted $N.Q$. If the extension is split, we write $N:Q$. If $N$ is abelian and the extension is central, we write $N\cdot Q$.  The wreath product $N\wr Q$ with head $Q$ and base $N$ is the semidirect product $\left(\prod_{Q} N\right):Q$ where $Q$ acts by permuting the copies of $N$.

Specific finite groups that will show up repeatedly throughout this article include $\sym(n)$ and $\alt(n)$, the symmetric and alternating groups on $n$ objects respectively, and $Dh_n$, the dihedral group of order $2n$.
\end{notation}

\subsubsection{Groups of Lie type}
Let $p$ be a prime, let $q=p^k$, and let $\FF$ be an algebraically closed field of characteristic $p$.  The simple algebraic groups over $\FF$ are classified by the Dynkin diagrams $\cala_n$ for $n\geq 1$, $\calb_n$ for $n\geq 3$, $\calc_n$ for $n\geq 2$, $\cald_n$ for $n\geq 4$, $\cale_6$, $\cale_7$, $\cale_8$, $\calf_4$, and $\calg_2$.  To each Dynkin diagram, we associate a finite number of simple algebraic groups; two such groups associated to the same diagram are called \emph{versions}. All versions become isomorphic upon taking their quotients by their respective centres. There is a version (the \emph{simply-connected} one) which maps onto every other version with a finite central kernel and there is a version (the \emph{adjoint} one) which is a
quotient of every other version with a finite central kernel.

Any \emph{finite group of Lie type} $G$ is obtained as the fixed point set of a Steinberg endomorphism of a connected simple algebraic group $\mathbf G$ defined over an algebraically closed field of characteristic $p$.  The groups of Lie type fall into two families: the types $\mathtt A_n$, $^2\mathtt A_n$, $\mathtt B_n$, $\mathtt C_n$, $\mathtt D_n$, $^2\mathtt D_n$ are the \emph{classical types}; the types $^2\mathtt B_2$, $^3\mathtt D_4$, $\mathtt E_6$, $^2\mathtt E_6$, $\mathtt E_7$. $\mathtt E_8$, $\mathtt F_4$, $^2\mathtt F_4$, $\mathtt G_2$, and $^2\mathtt G_2$ are the \emph{exceptional types}.  Note that groups of type $^2\mathtt B_2$ and $^2\mathtt F_4$ are only defined over fields of order $2^{2m+1}$  and groups of type $^2\mathtt G_2$ are defined over fields of order $3^{2m+1}$.

Each type of group has a finite number of \emph{versions}, the smallest one is again called the \emph{adjoint}.  The adjoint version is a simple group with the following exceptions:
\begin{align*}
    \mathtt A_1(2)\cong \sym(3), && \mathtt A_1(3) \cong \alt(4), && \mathtt C_2(2) \cong \sym(6), && ^2\mathtt A_2(2)\cong 3^2: Q_8 \\
    \mathtt G_2(2) \cong\; ^2\mathtt A_2(3):2, && ^2\mathtt B_2(2)\cong 5:4, && ^2\mathtt G_2(3)\cong \mathtt A_1(8):3, && \text{and }^2\mathtt F_4(2),
\end{align*}
where $Q_8$ is the quaternion group, and the commutator subgroup of $^2\mathtt F_4(2)$, which we denote by $^2\mathtt F_4'(2)$ is a simple group known as the \emph{Tits group}.  As is often the convention, we will treat $^2\mathtt F_4'(2)$ as an exceptional group of Lie type.  We also record a number of exceptional isomorphisms between the groups: $\mathtt A_1(4)\cong \mathtt A_1(5) \cong \alt(5)$, $\mathtt A_1(7)\cong \mathtt A_2(2)$, $\mathtt A_1(9)\cong \alt(6)$, $\mathtt A_3(2)\cong \alt(8)$, $^2\mathtt A_3(2) \cong \mathtt C_2(3)$,
and $\mathtt B_n(2^\ell) \cong \mathtt C_n(2^\ell)$ for all $\ell\geq 1$ and $n\geq 3$.

The adjoint classical groups admit the following descriptions: $\mathtt A_n(q)\cong\PSL_{n+1}(q)$, $\mathtt B_n(q)\cong \POmega_{2n+1}(q)$, $\mathtt C_n(q)\cong \PSp_{2n}(q)$, $\mathtt D_n(q)\cong \POmega^+_{2n}(q)$, $^2\mathtt A_n(q^2)\cong \PSU_{n+1}(q)$, and $^2\mathtt D_n(q^2)\cong \POmega_{2n}^-(q)$.  When working with a specific group, we will tend to use the right-hand-side notation, e.g, $\PSL_{n+1}(q)$.  However, when referring to a family of groups, we will tend to use the Chevalley notation, e.g., $\mathtt A_n(q)$. 

\subsubsection{The classification of finite simple groups}

\begin{thm}
Every finite simple group is, up to isomorphism, one of the following groups: 
\begin{itemize}
    \item a cyclic group of prime order;
    \item an alternating group of degree at least 5;
    \item a finite simple group of Lie type;
    \item one the 26 sporadic groups.
\end{itemize} 
Note here we view the Tits group $^2\mathtt F_4'(2)$ as a group of Lie type.
\end{thm}

Of the 26 sporadic groups, the relevant ones for our purposes will be the Mathieu groups $M_{11}$, $M_{12}$, $M_{22}$, $M_{23}$, and $M_{24}$, the Janko groups $J_1$, $J_2$, and $J_3$, and the Higman--Sims group $HS$.  Constructions of the sporadic groups as well as all of the relevant facts we will need are contained in the `ATLAS of finite simple groups' \cite{ATLAS}.  From here on we will often refer to \cite{ATLAS} as \emph{the ATLAS}.

\subsubsection{The socle}\label{S:socle}
A key tool in our studies will be an important characteristic subgroup of a group $G$.

The \emph{socle}, denoted $\soc(G)$, of a group $G$ is the subgroup generated by all minimal normal subgroups of $G$.  
We briefly recall a number of elementary facts, see \cite[Theorem~4.3A]{DixonMortimer1996}.  If the group $G$ is finite, then the socle is isomorphic to the direct product of the minimal normal subgroups of $G$.  Moreover, the socle of a finite group $G$ is isomorphic to a direct product of finite simple groups and a finite abelian group.

\subsection{The shape of a smallest quotient of perfect-by-cyclic group}
The following lemma describes the structure of a smallest non-cyclic quotient of a perfect-by-cyclic group.  This will be of central importance in our study of the exceptional Artin groups. We begin with a definition.

\begin{defn}
    A group $G$ is \emph{almost simple} if there is a non-abelian simple group $S$ such that $S \leq A \leq \Aut(S)$.
\end{defn}

\begin{lemma}\label{lem:min_quotient_En}
    If $G$ is a group such that $G^\ab$ is cyclic and $[G,G]$ is perfect, then a smallest non-cyclic quotient $Q$ of $G$ fits into an exact sequence $1\to S^n\to Q \to k \to 1$ where $S$ is a finite simple group and $k$ is cyclic.  Moreover, $\soc(Q)=S^n$ and
    \begin{enumerate}
    	\item if $n=1$, then $Q$ is almost simple;%\yv{Should we define the term ``almost simple"}
    	\item if $n>1$, then $k$ is equal to the order of an element of $\Out(S^n)$ and $Q$ is a subgroup of $S^n\wr k$.
    \end{enumerate}
\end{lemma}
\begin{proof}
    Any finite quotient $Q$ of $G$ fits into an exact sequence $1\to N\to Q \to C_k \to 1$ with $N$ perfect.  
    The result follows if we can show that $\soc(Q)=N$ and that $N$ is a product of simple groups.  If $\soc(Q)$ is not equal to $N$, then there is a normal subgroup $L\triangleleft Q$ such that the quotient $Q/L$ is non-abelian, contradicting minimality of $Q$.  Thus, we may suppose $\soc(N)=Q$.  
    The socle of any finite group is the product of finitely many simple groups and an abelian group.  Indeed, a minimal normal subgroup is characteristically simple and so since $Q$ is finite any minimal normal subgroup is a direct product of isomorphic simple groups.  
    
    If $Q$ had a non-trivial abelian normal subgroup $A$, we would be able to take a further smaller quotient $Q/A$.  Thus $\soc(Q)=\prod_{i=1}^\ell S_i$ with each $S_i$ simple.  
    Next, we show that each of the $S_i$ are isomorphic.  Indeed, if they were not, then we would have a proper subset $I$ of $[\ell]$ such that $L=\prod_{i\in I} S_i$ is a normal subgroup of $Q$ with $Q/L$ non-abelian.  This contradicts minimality of $Q$.  Thus, $S_i\cong S_j\cong S$ for all $i,j$ and $\soc(Q)=S^n$ as required.  
    
    It follows immediately that if $n=1$, then $Q$ is almost simple.  Otherwise, $Q$ would admit a further non-abelian quotient with kernel isomorphic to the kernel of map $k\to \Out(S^n)$.  If $n>1$, then by the Kaloujinine--Krasner Universal Embedding Theorem \cite{KrasnerKaloujnine1951}, every extension of $S^n$ and $k$ is a subgroup of the wreath product $S^n\wr k$.  Clearly $n|k$; otherwise, $k$ would fix some copies of $S$.  It remains to show that $k$ is equal to the order of an element of $\Out(S^n)$.   Suppose not, then the kernel of the homomorphism $k\to \Out(S^n)$ must be trivial.  Indeed, if it were not, then $Q$ would admit non-trivial central elements and hence a smaller non-abelian quotient.  Hence, $k$ is equal to the order of an element of $\Out(S^n)$.
\end{proof}

\subsection{Bounding the of size non-cyclic quotients using totally symmetric sets.}\label{sec:TTS}
In this section, we introduce one of the tools we will use to bound the size of non-cyclic quotients.
\begin{defn}
    For a group $G$, a subset $T = \{ t_1, t_2, \dots \} \subseteq G$ is called a \emph{totally symmetric set} (TSS) if for any permutation $\sigma \in \sym(T)$ there exists $g_\sigma \in G$ such that $g_\sigma t_i g_\sigma^{-1} = t_{\sigma(i)}$.
\end{defn}
 
This notion was first introduced by Kordek and Margalit \cite{KordekMargalit22} who worked under an additional pairwise commuting hypothesis.
This definition was later generalized by Caplinger--Salter to the version which appears in this paper; we refer to totally symmetric sets consisting of pairwise commuting elements as \emph{commuting totally symmetric sets}.
One can immediately observe that homomorphic images of totally symmetric sets remain totally symmetric. 

A key property of TSS's is the following ``collision implies collapse'' result. A version of this result was first proven by Kordek--Margalit with the additional pairwise commuting hypothesis {\cite[Lemma~2.1]{KordekMargalit22}}, and the version which appears in this paper was proven by Caplinger \cite[Lemma~4]{Caplinger2023}:

\begin{lemma}
\label{lem:TSS_collapse}
    Let $G$ be a group and $T \subseteq G$ be any totally symmetric set.  For any homomorphism $f: G \to H$, either 
    $f$ restricts to an injection on $T$ or $|f(T)| = 1$.  
\end{lemma}

This ``collision implies collapse'' property is weaker than the property of being a totally symmetric set, and we will sometimes require only that a given set satisfies the former.

\begin{defn}\label{D:collapsingSet}
    Let $G$ be a group.  A subset $X$ is a \emph{collapsing set} if for any epimorphism $\varphi:G \to Q$, either $\varphi$ restricts to an injection on $X$ or $|\varphi(X)| = 1$.
\end{defn}

The following is a convenient rephrasing of \cite[Lemma~8]{Kolay23}(see also \cite[Lemma~2.2]{KordekMargalit22}) in light of \Cref{lem:TSS_collapse} as the standard generators of braid groups are a subset of the Birman-Ko-Lee generators.
\begin{lemma}\label{L:braidGenCollapse}
    The standard generators of braid groups $\cala_n$ for $n \geq 4$ form a collapsing set. 
\end{lemma}

\subsection{Artin groups}\label{S:Artin-basics}

One obtains a presentation for an Artin group $A_{\Gamma}$ from a defining graph $\Gamma$ with edge labels in $\mathbb{N}_{\geq 3} \cup \{\infty\}$ as follows. For each vertex $v_i$ of $\Gamma$, let $s_i$ be a generator. For each pair $i$ and $j$, $m_{ij}$ is the label of the edge connecting vertices $v_i$ and $v_j$ in $\Gamma$, if it exists. If $v_i$ and $v_j$ are connected via an unlabelled edge, $m_{ij} = 3$. If $v_i$ and $v_j$ are not adjacent, $m_{ij} = 2$. The group is then given via the relations below.
\begin{align*}
        A_{\Gamma} =  \langle s_1,...s_n | \LaTeXunderbrace{s_i s_j s_i...}_{m_{ij}} = \LaTeXunderbrace{s_j s_i s_j ...}_{m_{ij}} \rangle
\end{align*}
   Relations of the type in the presentation above will often be referred to as \emph{Artin relations} and the defining graph $\Gamma$ is called a \emph{Coxeter diagram}. We will use the notation $[x,y]_\ell$ to denote the alternating word $xyxy \dots $ of length $\ell$, that is
    \[ [x,y]_\ell :=\LaTeXunderbrace{x y x...}_{\ell}.\]
    So, an Artin relation can be rewritten as $[x,y]_\ell=[y,x]_\ell$ where $m_{i,j}=m_{j,i}=\ell$.
    Notice when $\ell$ is even, the Artin relation becomes
    \[ [x,y]_\ell=(xy)^\frac{\ell}{2}.\]
    
    In general, Artin groups may have generators which do not admit any relation, but no such groups will arise in this paper. An Artin group is called \emph{reducible} if it can be decomposed as the direct product of two Artin groups where neither factor is trivial.  The irreducible spherical Artin groups are exactly those Artin groups whose defining graphs appear in Figure \ref{fig:FiniteCoxeterGroups} \cite{Coxeter34}.

    Throughout the paper, we will use the labellings of Artin generators given in Figure \ref{fig:FiniteCoxeterGroups} unless otherwise specified.

\subsubsection{Parabolic subgroups control finite quotients}
\label{S:parabolics}
In addition to totally symmetric sets, parity of the length of Artin relations furnish additional means of bounding finite quotients.
We record the following elementary facts:

A first fact is that pairs of generators with an odd relation are conjugate.
The next fact, recorded in the following lemma, is a straightforward application of Bezout's identity.
\begin{lemma}
\label{L:gcdRelation}
    Let $x$ and $y$ be elements of any group $G$. If $x$ and $y$ satisfy Artin relations of length $m$ and $n$ then they satisfy the Artin relation of length $\gcd(m,n)$.
\end{lemma}
\begin{proof}
    By Bezout's indentity there are integers $a$ and $b$ such that $k = \gcd(m,n) = am + bn$. 
    Up to reordering, assume that $a > 0 > b$.
    It is easy to see that Artin relations of a fixed length imply those for all multiples and that:
    \[
    [x,y]_{bn}^{-1}[x,y]_{am} = \begin{cases}
        [x,y]_k &   \text{if } bm \text{ is odd}\\
        [y,x]_k &   \text{if } bm \text{ is even}
    \end{cases}
    \qquad 
    [y,x]_{bn}^{-1}[y,x]_{am} = \begin{cases}
        [y,x]_k &   \text{if } bm \text{ is odd}\\
        [x,y]_k &   \text{if } bm \text{ is even.}
    \end{cases}
    \]
    An immediate consequence is that $[x,y]_k = [y,x]_k$.
\end{proof}

\begin{lemma}\label{L:divisbleEdge}
    Suppose that $A_\Gamma$ is an Artin group and $x,y$ are a pair of standard generators whose corresponding vertices in the defining graph are joined by an edge labelled $m$.  If $d$ is any divisor of $m$ then there exists a quotient $A_\Gamma \to A_{\bar{\Gamma}}$ where $\bar{\Gamma}$ is obtained by replacing the edge labelled by $m$ with an edge labelled by $d$.
\end{lemma}
\begin{proof}
    This can be seen immediately from the presentation. By adding the relation $[x,y]_d$ the prior relation $[x,y]_m$ becomes redundant using the same reasoning in \Cref{L:gcdRelation}.
\end{proof}

Recall that if $\Gamma$ is the defining graph for an Artin group $A_\Gamma$ and $\Lambda$ is a proper subgraph of $\Gamma$, then the inclusion $\Lambda \subset \Gamma$ induces an embedding of $A_\Lambda$ whose image is called a \emph{standard parabolic subgroup} of $A_\Gamma$.  
The third fact allows us to show that in many cases where the standard parabolic subgroup of an Artin group is a braid group, one can find a retraction map from an Artin group to this braid group.

\begin{lemma}
\label{lem:parabolic-retraction}
    Let $A_\Gamma$ be an Artin group and suppose that the standard parabolic subgroup corresponding to $\Lambda \subset \Gamma$ is a braid group. 
    If every $x \in \Gamma \smallsetminus \Lambda$ has even labelled relations with all generators in $\Lambda$, then there is a retraction $A_\Gamma \twoheadrightarrow A_\Lambda$.
\end{lemma}
\begin{proof}
    It is a general fact that given a group $G = \langle S \mid R \rangle$ and another group $H$, a map $\varphi$ which assigns an element of $H$ to each generator $s \in S$ extends to a homomorphism $\varphi: G \rightarrow H$ if and only if $\varphi(r) = 1$ for all $r \in R$, where $\varphi(r)$ is the natural extension $\varphi(s_i)\cdots\varphi(s_n)$ when $r = s_i \cdots s_n$.
    
    Set $\varphi(s) = s$ if $s \in V(\Lambda)$ and $\varphi(s) = 1$ if $s \in V(\Gamma) - V(\Lambda)$. Taking $R$ to be the standard Artin relations on $A_{\Gamma}$, it is clear that any $r \in R$ satisfies the above condition if the two Artin generators which appear in it are either both in $\Lambda$ or both in $\Gamma - \Lambda$. Any Artin relation involving $s \in \Lambda$ and $x \in \Gamma - \Lambda$ must be even, so the relation $r$ is of the form $(sx)^{m/2} (s^{-1}x^{-1})^{m/2}$. Applying $\varphi$, we obtain precisely $\varphi(s)^{m/2} \varphi(s)^{-m/2} = 1$. Thus the map $\varphi$ extends to a homomorphism $A_{\Gamma} \rightarrow A_{\Lambda}$ in which every generator of $A_{\Lambda} \leq A_{\Gamma}$ is fixed.
    
    %(Sketch) Even relations can be replaced with commutation relations in a quotient. 
    %Commutating relations are encoded by disjointness.
    %Thus, we can assume the Coxeter diagram is disconnected.  Note that $\Lambda$ is preserved under these adjustments. 
    %Connected components of Coxeter diagrams are factors in a direct product decomposition.  
    %The retraction is obtained by projecting onto the factor corresponding to $A_\Lambda$.
\end{proof}

\section{Smallest quotients of non-exceptional Artin groups}\label{sec:non_excpetional}

In this subsection, we will determine the smallest quotients of the non-exceptional Artin groups $\cala_n, \cali_2(n), \calb_n$ and $\cald_n$, in this order.

\subsection{Smallest quotient of $\cala_n$}\label{sec:A_n_results}
Since $\cala_n \cong Br_{n+1}$, this case is resolved by Kolay. 
\begin{thm}~{\cite[Claim~7 and Theorem~1]{Kolay23}}\label{kolay}
    The smallest non-abelian quotient of $\cala_2 \cong Br_{3}$ is $\sym(3)$.
    
     For $n\geq 4$, the smallest non-abelian quotient of $\cala_n \cong Br_{n+1}$ is $\sym(n+1)$.  Moreover, the quotient map is unique up to an automorphism of $\sym(n+1)$.
\end{thm}

\begin{remark}\label{rk:a3-exceptional}
    The group $\cala_3 \cong Br_{4}$ admits a quotient to $\sym(3)$ because its associated Coxeter group is $\sym(4)$, which has an exceptional homomorphism to $\sym(3)$.
\end{remark}

%\begin{thm}~{\cite[Claim~7 and Theorem~1]{Kolay23}}\label{kolay}
%    For $n = 3$, the smallest non-abelian quotient is $Sym(3)$.
    
%    For $n\geq 4$ the smallest non-abelian quotient of $\cala_n$ is $\sym(n+1)$.  Moreover the quotient map is unique up to an automorphism of $\sym(n+1)$.
%\end{thm}

When $n = 4$, one can also classify the homomorphisms from $\cala_4$ to $\sym(n+2)$. This result is used in later sections to examine the possible images of $\cala_4$ parabolic subgroups in quotients of other Artin groups.

%\thomas{Added \cref{L:braidGenCollapse} above that we can cite instead of the later result that mentions $D_4$ and $E_6$}

\begin{lemma}\label{lem:A4toSym6}
    Any homomorphism $\cala_4\to\sym(6)$ must either have cyclic image or be equal to a copy of $\sym(5)$.  In the latter case, the homomorphism factors through the canonical Coxeter quotient, $W(\cala_4)$.
\end{lemma}
\begin{proof}
 We claim the image of any homomorphism $\alpha:\cala_4\to\sym(6)$ must either be cyclic or be equal to a copy of $\sym(5)$.  If the generator $\alpha(s_4)$ is trivial, then all of $\alpha(\cala_4)$ is trivial because all Artin generators are conjugate. If $\alpha(s_4)$ has order 2, then all Artin generators must have order 2 for the same reason. In this case, $\alpha$ factors through the canonical Coxeter quotient $W(\cala_4)$, so $\alpha(\cala_4)$ must be a subgroup of $W(\cala_4) = \sym(5)$. The image cannot be a proper subgroup by \cref{kolay}.
 
 Suppose the order of $\alpha(s_4)$ is larger than 2, and let $\cala_2$ denote the standard parabolic subgroup $\langle s_1 , s_2 \rangle$ which commutes with $s_4$. We examine the possible images of $\alpha(s_4)$. A non-trivial permutation in $\sym(6)$ with order not equal to 2 has one of the following 7 cycle types: $2^13^1$, $2^14^1$, $3^1$, $3^2$, $4^1$, $5^1$, $6^1$. The centraliser of such an element is isomorphic to either $6$, $2\times 4$, $3\times\sym(3)$, $3\wr2$, $2\times4$, $5$, and $6$ respectively, written in the finite groups notation given at the beginning of Section \ref{subsec:finite-groups}.  

 If the image $\alpha(\cala_2)$ is cyclic, then $\alpha(\cala_4)$ must also be cyclic by \cref{L:braidGenCollapse}, so it remains only to consider the possibility that the centraliser of $\alpha(s_4)$ is $\sym(3)$ or a subgroup of $3\wr 2$. If $\cala_2$ maps to $\sym(3)$, then this factors through the canonical Coxeter quotient $\cala_2\to\sym(3)$. This forces all Artin generators to have order 2, so in fact, $\alpha$ itself factors through the canonical Coxeter quotient $W(\cala_4)$.
 
 The group $3\wr2$ is isomorphic $3\times\sym(3)$. If the image $\alpha(\cala_2)$ is abelian, then $\alpha(s_1)$ and $\alpha(s_2)$ commute in the quotient, but they also satisfy a braid relation, so $\alpha(s_1) = \alpha(s_2)$ by \cref{L:gcdRelation}. The entire image $\alpha(\cala_4)$ is then cyclic by \cref{L:braidGenCollapse}. 
 
 If $\alpha(\cala_2)$ is a non-abelian subgroup of $3\times\sym(3)$, then in fact $\alpha(\cala_2) \cong \sym(3)$. Indeed, the previous paragraph shows that $\alpha(s_1)$ and $\alpha(s_2)$ must be in the same factor of the direct product if $\alpha(\cala_2)$ is non-cyclic, and these elements generate $\alpha(\cala_2)$. As before, $\alpha(\cala_2)$ must factor though the canonical Coxeter quotient by \cref{kolay}, implying that $\alpha(\cala_4)$ also factors through $W(\cala_4)$. 
 
 %The the only possible non-cyclic images for a copy of $\cala_2$ centralising an Artin generator of $\cala_4$ are $\sym(3)$ or a subgroup of $3\wr 2$.  
 %\yv{Why does it suffice to only consider ``possible non-cyclic images for a copy of $\cala_2$ centralising an Artin generator of $\cala_4$"? There's something here that I'm missing.}
%But $3\wr2$ is isomorphic $3\times\sym(3)$ and so any non-cyclic image of $\cala_2$ must surject onto $\sym(3)$.  
\end{proof}

We will bootstrap off the control of finite quotients obtained in \Cref{kolay} to study quotients of other Artin groups using the following results.

\begin{lemma}\label{L:braidCollapse}
    Let $n\geq 2$. If $\cala_n \to C$ is any map to an abelian group, then the generators collapse.
\end{lemma}
\begin{proof}
    This is immediate from \Cref{L:braidGenCollapse} and \Cref{L:gcdRelation} since $\gcd(2,3) =1$.
\end{proof}

\begin{prop}\label{prop:braidParabolic}
    Let $A_\Gamma$ be an Artin group with $A_\Lambda$ an (irreducible) standard parabolic subgroup isomorphic to $\cala_n = Br_{n+1}$ for $n \geq 4$.\\
    If $\varphi: A_\Gamma \to Q$ is any quotient onto a finite group whose restriction to $A_\Lambda$ is noncyclic, then 
    \[
    |Q| \geq n!,
    \]
    Moreover if there is a generator $t \in \Gamma$ such that $\varphi(t) \notin \varphi(A_\Lambda)$ then the inequality is strict.
    %Previous statement was not quite right.  
    % Suppose that $G = \langle S \cup T \rangle$ where $|S| \geq 5$.  
    % If $i: Br_n \hookrightarrow G$ is given by $\sigma_i \mapsto s_i$ for each $s_i \in S$.  If $\varphi:G \twoheadrightarrow Q$ is a non-abelian quotient whose restriction to $Br_n$ then either 
    % \[
    % |Q| \geq n!,
    % \]
    % or $\varphi \circ i(Br_n) = \langle\bar{s}\rangle$ is cyclic and 
    % there exists $t \in T \smallsetminus S$ such that $\varphi(t)$ does not commute with $\bar{s}$.
\end{prop}
\begin{proof}
    Consider the restriction of $\varphi$ to the parabolic braid subgroup.  By \cite{Kolay23}, the order $|\varphi(H)| \geq n!$, so $|Q| \geq |\varphi(H)|$.
    The moreover statement follows by the Orbit-Stabiliser Theorem applied to cosets of $\varphi(H)$ by powers of $\varphi(t)$. 
\end{proof}

% ------------------------------------------
\subsection{Smallest quotient of $\cali_2(n)$}

We begin this subsection by proving that when $n$ is divisible by $4$, then $\cali_2(n)$ has an exceptional quotient that does not factor through the associated Coxeter group, including the case $\cali_2(4) = \calb_2$.

\begin{lemma}
    \label{L:I24exotic}
    Suppose that $n$ is divisible by $4$.  The smallest nonabelian quotient of $\cali_2(n)$ is $\sym(3)$.
\end{lemma}
\begin{proof}
    Since $4$ divides $n$, by \Cref{L:divisbleEdge}, there exists a quotient $\cali_2(n) \to \cali_2(4)$.
    Consider the presentation $I_2(4) = \langle x,y \mid (xy)^2 = (yx)^2 \rangle$. 
    The assignment $x \mapsto (123)$ and $y \mapsto (23)$ defines the desired homomorphism.  To see this, it suffices to check that the relation is satisfied.  
    Observe that $xy \mapsto (12)$ and $yx \mapsto (13)$ both map to involutions. 
    Thus, the relation $(xy)^2 = (yx)^2$ is also satisfied in $\sym(3)$.

    The smallest nonabelian finite group is $Sym(3)$, so it must be the smallest nonabelian quotient of $\cali_2(n)$.
\end{proof}
Before considering values of $n$ not divisible by 4, we note the following simple fact which will be used to see how the centre of $\cali_2(n)$ behaves under quotients.

\begin{lemma}
    \label{L:central-by-cyclic}
    A central extension of a cyclic group is abelian.
\end{lemma}
\begin{proof}
    Suppose that 
    \[
    1 \to A \to G \to C \to 1,
    \]
    is a short exact sequence where $A$ is central in $G$ and $C$ is cyclic.
    Lift a generator of $C$ to $z \in G$. 
    Every pair of elements $g,h \in G$ can thus be rewritten as $g = az^k$ and $h = bz^j$ where $a,b \in A$.  
    Since $A$ is central we see that $gh = abz^{j+k} = hg$.
    \end{proof}
    
    It was proved in \cite{BrieskornSaito72} that if $x,y$ are the stardard Artin generators of $\cali_2(n)$, then the centre is generated by either $[x,y]_n = (xy)^{\frac n2}$ when $n$ is even or $[x,y]_n^2 = ((xy)^{\frac{n-1}{2}}x)^2$ when $n$ is odd.
    In either case, these elements have a nontrivial root involving the element $xy$, which gives rise to convenient alternate presentations.

\begin{lemma}
    \label{L:I2-centralPresentation}
    The Artin group $\cali_2(n)$ admits one of the following presentations:
    \[
    \langle x, r \mid x^{-1}r^mx = r^m \rangle   \quad \text{ for } n = 2m, \qquad \langle r, s \mid s^2 = r^n  \rangle \quad \text{ for } n = 2m+1. 
    % \begin{cases*}
    %     \langle x, r \mid x^{-1}r^mx = r^m \rangle   & n = 2m\\
    %     \langle r, s \mid s^2 = r^n  \rangle    & n = 2m+1 
    % \end{cases*}
    \] 
\end{lemma}
\begin{proof}
    Suppose first that $n = 2m$ is even. 
    Let $r = xy$, implying $y = x^{-1}r$.
    The following rewriting of the Artin relation gives the desired presentation:
    \begin{align*}
        [x,y]_{2m}  &=  [y,x]_{2m}\\
        (xy)^m      &=  (yx)^m = y(xy)^{m-1}x\\
        r^m         &=  x^{-1}r^{m} x.
    \end{align*}

    Now suppose $n = 2m+1$ is odd. Let $r = xy$ and let $s = (xy)^m$.
    The following rewriting of the Artin relation gives the desired presentation:
    \begin{align*}
        [x,y]_{2m+1}    &=  [y,x]_{2m+1}\\
        (xy)^m x        &=  yx(yx)^{m-1}y = y(xy)^m \\
        r^m x           &=  x^{-1}r^{m+1}\\
        (r^m x)(r^m x)  &=  r^m r^{m+1}\\
        s^2             &=  r^n. \qedhere
    \end{align*}
\end{proof}
 
The presentations in \Cref{L:I2-centralPresentation} makes it clear to see that the quotient of $I_2(n)$ by its centre is virtually free which is also proved in \cite[Theorem~6.2]{McCammond10}.  
Indeed, 
\[
\rightQ{\cali_2(n)}{Z(\cali_2(n))} = \begin{cases*}
    \langle x, r \mid r^m = 1 \rangle &\text{ for even } n,\\
    \langle r,s \mid s^2 = r^m = 1 \rangle &\text{ for odd } n.
\end{cases*}
\]

As we will see, similar analysis can be used to control the size of finite quotients of $\cali_2(n)$.

\begin{thm}
    \label{thm:I2}
    Let $n>0$.  If $n$ is not divisible by $4$ then the smallest nonabelian quotient of $\cali_2(n)$ is the dihedral group $Dh_p$ where $p$ is the smallest odd prime divisor of $n$.
    If $4|n$, then the smallest nonabelian quotient of $\cali_2(n)$ is $\sym(3)$.
\end{thm}
\begin{proof}
    The case when $4$ divides $n$ is proved in \Cref{L:I24exotic}.
    
    Suppose that $\varphi: \cali_2(n) \twoheadrightarrow Q$ is any nonabelian finite quotient.
    Let $p$ be the smallest odd prime divisor of $n$ (hence $p>2)$.
    Since $n$ is not divisible by $4$, we will consider the case where there exists an odd $m \geq 3$ such that $n = 2m$ and the case where $n = 2m+1$ in turn.
    In either case, first we show that $|Q|\geq 2p$.

    %   EVEN CASE
    Suppose there exists an odd $m \geq 3$ such that $n = 2m$, that is, $n \cong 2 \pmod{4}$. 
    By Lemma \ref{L:I2-centralPresentation}, we have the following presentation:
    \[ \cali_2(2m) = \langle x,r \mid x r^m = r^m x \rangle. \]
    Let $\bar{r} = \varphi(r)$ and $\bar{x} = \varphi(x)$.   
    The presentation implies that $Q$ is a central extension of $\bar{Q} := \rightQ{Q}{\langle \bar{r}^m \rangle}$. 
    By \Cref{L:central-by-cyclic}, $\bar{Q}$ cannot be cyclic, implying that neither $\bar{r}$ nor $\bar{x}$ are trivial. 
    Consider the proper subgroup $H = \langle \bar{r} \rangle$ of $\bar{Q}$.  
    By construction, $\bar{r}^m = 1_{\bar{Q}}$, so $|H| $ divides $m$.
    In particular, $|H| \geq p$ because $m$ is odd.   
    Since $\bar{x} \notin H$, the cosets $H$ and $\bar{x}H$ are disjoint. 
    By the Orbit-Stabiliser Theorem, $|Q| \geq |\bar{Q}| \geq |H \sqcup \bar{x}H| \geq 2p$.

    %   ODD CASE
    We now suppose $n = 2m+1$ for $m \geq 1$.
    By \Cref{L:I2-centralPresentation} we have the following presentation
    \[
    \cali_2(2m+1) = \langle r, s \mid s^2 = r^{2m+1} \rangle.
    \]
    Let $\bar{r}  = \varphi(r)$ and $\bar{s} = \varphi(s)$. 
    The presentation implies that $Q$ is a central extension of $\bar{Q} = \rightQ{Q}{\langle \bar{r}^{2m+1} \rangle}$.
    By \Cref{L:central-by-cyclic}, $\bar{Q}$ cannot be cyclic implying that neither $\bar{r} \neq 1_{\bar{Q}} \neq \bar{s}$.
    Consider the proper subgroup $H = \langle \bar{r} \rangle$.
    By construction, $\bar{r}^{2m+1} = 1_{\bar{Q}}$, so $|H|$ divides $2m+1$. 
    In particular, $H$ and $\bar{s}H$ are disjoint cosets where $|H| = |\bar{s}H| \geq p$.
    By the Orbit-Stabiliser Theorem, $|Q| \geq |\bar{Q}| \geq |H \sqcup \bar{s}H| \geq 2p$.
    
    %   UNNIQUENESS OF SMALLEST QUOTIENT
    We conclude by noting that each of the previous cases are constructive.  
    Indeed, if $Q$ is smallest among all nonabelian finite quotients, then the subgroup $H$ must have order $|H| = p$, otherwise the natural dihedral quotient obtained from the composition $\cali_2(n) \twoheadrightarrow \cali_2(p) \twoheadrightarrow Dh_p$ would be smaller. 
    Similarly, the element $\bar{x}$ when $n=2m$ or $\bar{s}$ when $n=2m+1$ is necessarily 2-torsion. 
    Thus, we may conclude that there is a unique quotient $\cali_2(2m)$  to a group of order $2p$ up to an automorphism, and there is a unique quotient $\cali_2(2m+1)$  to a group of order $2p$ up to an automorphism.
\end{proof}

% \thomas{merged the statements of even and odd cases (commented out here)}
% \begin{prop}
%     \label{P:I2evenA}
%     If $n \cong 2 \pmod{4}$, then the smallest nonabelian quotient of $I_2(n)$ is $D_p$ where $p$ is the smallest odd prime divisor of $n$.
% \end{prop}
% \begin{proof}
%     Let $m \geq 3$ be odd such that $n = 2m$.  Suppose that $\varphi: \cali_2(2m) \twoheadrightarrow Q$ is any nonabelian finite quotient.
% \end{proof}
    
% The case for $n$ odd is handled similarly using the other presentation.  

% \begin{prop}
%     \label{P:I2odd}
% If $n$ is odd then the smallest nonabelian quotient of $I_2(n)$ is $D_p$ where $p$ is the smallest prime divisor of $n$.
% \end{prop}
% \begin{proof}
%     Let $m \geq 1$ so that $n = 2m+1$. Suppose that $\varphi: I_2(2m+1) \twoheadrightarrow Q$ is any nonabelian finite quotient.

%     As before, the smallest quotient is achieved when $|H| = p$ and when $Q = \bar{Q} = D_p$.
% \end{proof}

%\begin{proof}
 %   Case \ref{item:I2div4} is proved in \Cref{L:I24exotic} and Case \ref{item:I2generalCase} is proved in \Cref{P:I2generalCase}.
%\end{proof}

% ------------------------------------------
\subsection{Smallest quotient of $\calb_n$}

Note that $\calb_n$ is generated by a canonical maximal parabolic braid subgroup $P(\Lambda) \cong \cala_{n-1}$ and one additional element $s_1$ having even Artin relations with every generator of $\Lambda$.
By \Cref{L:divisbleEdge}, there exists a retraction onto $P(\Lambda)$ coming from the first two maps in the following chain of quotients
\[
\calb_n \to \Z \times \cala_{n-1} \to \cala_{n-1}.
\]
We will see that this is often the source of smallest quotient for this family.

Before we establish bounds on finite quotients we observe the following facts.

\begin{lemma}\label{L:CnExotic}
    If $n = 2,3,4$, then the smallest nonabelian quotient of $\calb_n$ is $\sym(3)$. %and factors through the Coxeter group.  
\end{lemma}
\begin{proof}
    We handle each value of $n$ separately. 
    When $n = 2$, $\calb_2 = \cali_2(4)$ which has unique smallest quotient $\sym(3)$ from \Cref{L:I24exotic}.
    When $n=3$, the Coxeter group $2 \wr \sym(3)$ has a natural map to $\sym(3)$.  
    When $n=4$, the parabolic subgroup $P(\Lambda) \cong \cala_3$ admits an exotic quotients to $\sym(3)$ as mentioned in \cref{rk:a3-exceptional}.
\end{proof}

\begin{lemma}\label{L:cyclicAbelian}
    Let $n \geq 3$.  For any quotient map $\varphi: \calb_n \to Q$, if $\varphi(P(\Lambda))$ is cyclic then $Q$ is abelian. 
\end{lemma}
\begin{proof}
    By \Cref{L:braidCollapse}, each of the generators $s_2, \dots, s_n$ of $P(\Lambda)$ are mapped to a single element in $Q$, call it $\bar{t}$.  Hence, $Q$ is generated by $\bar{t}$ and $\bar{s} = \varphi(s_1)$. 
    Since $n \geq 3$, we have relation $[\bar{s}, \bar{t}] = 1$.
    Thus, $Q$ is abelian.
\end{proof}

\begin{thm}\label{thm:Cn}
    Let $n \geq 5$.  If $\varphi:\calb_n \to Q$ is any quotient to a nonabelian finite group, then one of the following must hold:
    \begin{enumerate}
        \item The quotient map factors through the Coxeter group and $|Q| \geq n!$.  
        \item The quotient map does not factor through the Coxeter group and $|Q| > n!$.
    \end{enumerate}
\end{thm}
\begin{proof}
    By \Cref{L:cyclicAbelian} and \Cref{prop:braidParabolic}, the image of the parabolic subgroup cannot be cyclic and has order $|Q| \geq n!$ with equality only if $\varphi$ factors through the Coxeter quotient on the $\cala_{n-1}$ subgroup.
    
    If $\varphi(s_1) \not\in \varphi(\cala_{n-1}$, then $|Q| > n!$ by \cref{prop:braidParabolic}. If $\varphi (s_1) \in \varphi(\cala_{n-1}) \cong \sym(n)$, then $\varphi(s_1)$ must be the trivial element because its centralizer contains the $\sym(n-1)$ subgroup generated by $\{\varphi(s_3) , \cdots , \varphi(s_n)\}$. In particular, if $|Q| = n!$, then $\varphi$ factors through the quotient $\calb_n \rightarrow \cala_{n-1}$ described at the beginning of the section and the unique quotient $\cala_{n-1} \rightarrow \sym(n)$.
    % Hence, suppose $\varphi$ does not factor through the Coxeter group. 
    % In particular, at least one generator does not map to an element of order 2.  
    % If such a generator is in $P(\Lambda)$ then $|\bar{P(\Lambda)}| > n!$.   
    % Otherwise, $\bar{P(\Lambda)} \cong Sym(n)$ and $\bar{s}:= \varphi(s_1) \notin \bar{P(\Lambda)}$.  Moreover, $|\varphi(s_1)| \geq 3$. By the Orbit-Stabiliser theorem $|Q| \geq |\bar{P(\Lambda)} \sqcup \bar{s}\bar{P(\Lambda)} \sqcup \bar{s}^2 \bar{P(\Lambda)}| = 3n^! > n!$.
\end{proof}

\subsection{Smallest quotient of $\cald_n$}
Recall that $\cald_n$ is only defined for $n\geq 4$ and that the Coxeter group for $\cald_n$ is the finite group $2^{n-1}.S_n$.

The Coxeter group $W(\cald_4)\cong 2^3. S_4$ is isomorphic to $2^{1+4} : S_3$, so $\cald_4$ has $\sym(3)$ as a quotient. 
% For $n \leq 3$, $\cald_n \cong Br_{n+1}$. 

%  OLD BOUND THAT IS NOT STRONG ENOUGH TO PROVE WHAT WE WANT
 % \begin{thm}
 % Let $n\geq 6$. 
  
 %  $\eta (\cald_n) \geq \lfloor \frac{n-1}{2} \rfloor !$
 % \end{thm}

\begin{thm}\label{thm:Dn}
    Let $n \geq 5$. If $\varphi: \cald_n \rightarrow Q$ is a nonabelian quotient map, then either $|Q| > n!$ or $\varphi$ factors through the natural quotient map from $\cald_n$ to its Coxeter group and $|Q| \geq n!$.
\end{thm}
\begin{proof}
Notice that $\cald_n$ has two standard parabolic subgroups $P(\Gamma \smallsetminus \{s_1\})$ and $P(\Gamma \smallsetminus \{s_2\})$ both isomorphic to $\cala_{n-1} \cong Br_{n}$. Applying Proposition \ref{prop:braidParabolic}, either $|Q| \geq n!$ or both braid subgroups have cyclic image in $Q$. 
The cyclic case is impossible because each of the parabolic subgroups share the generator $s_n$, so by \Cref{L:braidCollapse} this would force $Q$ to be cyclic contradicting it being nonabelian. 

Suppose now that $\varphi$ does not factor through the Coxeter group.
All generators are conjugate in $\cald_n$, so if $\varphi$ does not factor through the Coxeter group on $\cald_n$, then it also does not factor through the Coxeter group when restricted to either of them. Thus $|Q| \geq |\varphi(\Gamma\smallsetminus \{s_i\})| > n!$ for at least one value of $i = 1,2$.
\end{proof}

% The requirement $n \geq 5$ allows us to apply results of Kolay for $Br_{n}$ that hold when $n = 3$ or $n \geq 5$. 

The lower bound in \Cref{thm:Dn} is realized by composing the quotient map which identifies $x_n$ and $x_{n-1}$ and the natural quotient map from $\cala_{n-1}$ to $S_n$. One can check that the first map is indeed a homomorphism via the same method used in the proof of Lemma \ref{lem:parabolic-retraction}.

%\section{Some lemmas about spherical Artin groups}
% \begin{lemma}
%     Let $S$ be a finite simple group.  If $Q$ is a group of shape $S^n.k$ with no proper non-cyclic quotients, then any $\alt(m)$ subgroup with $m\geq 4$ is contained in a copy of $S$.
% \end{lemma}

% \begin{lemma}
%     Let $S$ be a finite simple group.  If $Q$ is a group of shape $S^n.k$ with no proper non-cyclic quotients, then any totally symmetric subset of $Q$ of cardinality at least $3$ is contained in a copy of $S$.
% \end{lemma}
\section{Properties of Artin groups}\label{sec:Artin_props}

This section is devoted to structural properties of spherical Artin groups which will be used in later sections. We begin by recalling the following standard facts.
%\katie{These appear to be standard facts; it seems like many papers state them without citation. I'll keep looking to see if I can find one, but it might just be fine to leave them}

\begin{lemma}
\label{L:rankAbelianization}
    The abelianization of a spherical Artin group $A_{\Gamma}$ is free abelian with rank is equal to the number of connected components of the graph $\Gamma^{odd}$ obtained by removing all even-labelled edges from $\Gamma$.
\end{lemma}

%\sam{I've commented this out as we never used it}
% In the case where $\Gamma^{odd}$ is connected, \Cref{L:rankAbelianization} yields the following concise description of the commutator subgroup of $A_{\Gamma}$.
% \begin{lemma}
% \label{L:abelianizationExplicit}
%     If $A_{\Gamma}^{ab} \cong \mathbb{Z}$, the map which sends each Artin generator to $1 \in \mathbb{Z}$ defines a homomorphism $\ell$ called the \emph{length homomorphism}, and the commutator subgroup of $A_{\Gamma}$ is precisely $\ker(\ell)$.
% \end{lemma}

\begin{thm}~\cite{Zinde1975}~\cite{Muholland2006}~\cite{Orekov2012}\label{thm:comm-perfect}
    If $G$ is isomorphic to one of $\cale_6$, $\cale_7$, $\cale_8$, $\calh_3$, $\calh_4$, or $\cala_n$ with $n \geq 4$, then the commutator subgroup $G'$ is perfect. 
\end{thm}
We make the following observation, which will be used throughout.
\begin{lemma}\label{lem:a4-solvable-cyclic}
    If $G$ is isomorphic to one of $\cale_6$, $\cale_7$, $\cale_8$, $\calh_3$, $\calh_4$, or $\cala_n$ with $n \geq 4$, every homomorphism $G\to H$ where $H$ is a solvable group must have cyclic image.
\end{lemma}
\begin{proof}
    By \Cref{L:rankAbelianization} and \Cref{thm:comm-perfect}, the abelianization $G^{ab}$ is cyclic, and the commutator subgroup of $G$ is perfect.
\end{proof}

As observed in \Cref{S:parabolics} two Artin generators are conjugate if they are connected by an odd-labelled edge path in $\Gamma$. This allows us to make the following observation.

\begin{remark}
     If $\Gamma$ is connected and odd labelled, then all Artin generators have the same order in any quotient $Q$ of $A_{\Gamma}$. In particular, if the image of any Artin generator is trivial, $Q$ is trivial.
\end{remark}

\subsection{Totally symmetric sets in $\cale_n$}

In this section, we construct totally symmetric sets in Artin groups of type $\cale_n$ for all $n$. The results in this section are stronger than what is needed to determine the smallest quotients of $\cale_6$, $\cale_7$, and $\cale_8$, but they are interesting in their own right and may prove useful for readers interested in finite quotients of $\cale_n$ for $n >8$. Throughout, the generators of $\cale_n$ are labelled following the convention established for $\cale_6$, $\cale_7$, and $\cale_8$ in \cref{fig:FiniteCoxeterGroups}.

\begin{lemma}\label{lem:En_comm_tss}
    Let $n \geq 6$. In an Artin group of type $\cale_n$ where $n$ is even, the sets 
    \[\{s_1 , s_4, s_{4+2}, \dots, s_{n}\}\quad \text{and} \quad \{s_2, s_3, s_{3+2}, \dots ,s_{n-1}\}\]
    are commuting totally symmetric sets of size $\frac{1}{2}n$. In an Artin group of type $\cale_n$ where $n$ is odd, the sets 
    \[\{s_1 , s_4, s_{4+2}, \dots, s_{n-1}\}\quad \text{and}\quad \{s_2, s_3, s_{3+2}, \dots, s_{n-2}\}\]
    are commuting totally symmetric sets of size $\frac{1}{2}(n-1)$.
\end{lemma}
\begin{proof}
    It is immediate that all elements of the given sets pairwise commute. The sets $\{s_1 , s_4, s_{4+2}, \dots, s_{n}\}$ for even $n$ and $\{s_1 , s_4, s_{4+2}, \dots, s_{n-1}\}$ for odd $n$ are contained entirely in the standard parabolic subgroup $\langle s_1 , s_3 , s_4 , s_5 , \dots, s_n \rangle$. This subgroup is isomorphic to an Artin group of type $\cala_{n-1}$. It was shown in \cite{ChudnovskyKordekLiPartin20,KordekMargalit22} that sets of this type in braid groups are commutative totally symmetric sets, and the elements of $A_{n-1}$ which induce the desired permutations still induce the desired permutations when included into $\cale_n$. 

    The sets $\{s_2, s_3, s_{3+2}, \dots, s_{n-1}\}$ for even $n$ and $\{s_2, s_3, s_{3+2}, \dots ,s_{n-2}\}$ for odd $n$ are not contained in a braid-type parabolic subgroup, so a slightly different strategy is required. First, observe that we can apply the analogous fact for braid groups to realize any permutation supported in $\{s_{3+2}, \dots ,s_{n-1}\}$ or $\{s_{3+2}, \dots, s_{n-2}\}$ as conjugation \emph{by an element of the parabolic subgroup} $\langle s_{3+2}, \dots , s_n \rangle$. This subgroup commutes with both $s_2$ and $s_3$, so this element induces the desired permutation on the entire proposed set. Since the symmetric group is generated by the adjacent transpositions, it now suffices to show that we can realize the transpositions $(s_2, s_5)$ and $(s_2, s_3)$.

    We begin with the transposition $(s_2 , s_5)$. Notice that there is a standard parabolic subgroup of type $D_5$, namely $\langle s_1, s_2 , s_3, s_4, s_5 \rangle$, which commutes with $\{s_7, \dots , s_n\}$. Let $\Delta_{D_5}^L$ denote the inclusion into $A_{\Gamma}$ of the Garside element of the Artin group $D_5$ with the obvious generating set. Conjugation by $\Delta_{D_5}^L$ permutes $s_2$ and $s_5$ while fixing all other Artin generators in the $D_5$ subgroup, including $s_3$, and while fixing anything which commutes with the subgroup, including $\{s_{3+4}, \dots \}$. 

    A similar argument holds for any pair of Artin generators which can be realized as the ``prong'' vertices of a parabolic subgroup of type $D_{2k+1}$ for any $k$. In particular, vertices $s_2$ and $s_3$ form the ``prongs'' of the type $D_{n-1}$ subgroup $\langle s_2, s_3 , s_4, \dots, s_n\rangle$ when $n$ is even, or of the type $D_{n-2}$ subgroup $\langle s_2, s_3 , s_4, \dots, s_{n-1}\rangle$ when $n$ is odd. In either case, this subgroup contains all of the generators in the proposed totally symmetric set, so conjugation by the inclusion of the Garside element permutes $s_2$ and $s_3$ while leaving the remainder of the set fixed. In particular, this element commutes with the entire braid-type parabolic subgroup $\langle s_{3+2} , \dots \rangle$.
\end{proof}

In the $\cale_7$ case, we also make use of totally symmetric sets in the commutator subgroup $\cale_7'$. The following was proven by Zinde \cite{Zinde1975}, see also \cite{Muholland2006,Orekov2012}.

\begin{thm}
 \label{thm:commutator-type-E-presentation}
   For $n=6,7,$ or $8$ the commutator subgroup $\cale_n'$ of the
   Artin group $\cale_n$ is a finitely presented perfect group.  $\cale_n'$
   is the group generated by
    \begin{eqnarray*}
     p_0=a_{3}a_{1}^{-1}, \quad
     p_1=a_{1}a_{3}a_{1}^{-2}, \quad
     q_{\ell}=a_{\ell}a_{1}^{-1} \ \text{for $\ell=2,4,\dots,n$} , \quad
     b=a_{3}a_{1}^{-1}a_{4}a_{3}^{-1},
   \end{eqnarray*}
   with defining relations
    \begin{gather*}
      b=p_0q_4p_{0}^{-1}, \quad
      p_0bp_{0}^{-1}=b^{2}q_4^{-1}b,\\
      p_1q_4p_{1}^{-1}=q_4^{-1}b, \quad
      p_1bp_{1}^{-1}=(q_4^{-1}b)^{3}q_4^{-2}b,\\
      p_{0}q_{j}=q_{j}p_{1}, \quad p_1q_j=q_{j}p_{0}^{-1}p_{1} \quad (5\le
      j \le n \text{ or } j=2), \\
      q_{i}q_{i+1}q_{i}=q_{i+1}q_{i}q_{i+1} \quad (4\le i \le n-1), \\
      q_{2}q_{4}q_{2}=q_{4}q_{2}q_{4},  \\
      q_{i}q_{j}=q_{j}q_{i} \quad (4\le i <j \le n, |i-j|\ge 2), \\
      q_{i}q_{2}=q_{2}q_{i}\quad (5\le i \le n).
   \end{gather*}
   
%   \begin{gather*}
%      b=p_0q_3p_{0}^{-1}, \quad
%      p_0bp_{0}^{-1}=b^{2}q_3^{-1}b,\\
%      p_1q_3p_{1}^{-1}=q_3^{-1}b, \quad
%      p_1bp_{1}^{-1}=(q_3^{-1}b)^{3}q_3^{-2}b,\\
%      p_{0}q_{j}=q_{j}p_{1}, \quad p_1q_j=q_{j}p_{0}^{-1}p_{1} \quad (4\le
%      j \le n), \\
%      q_{i}q_{i+1}q_{i}=q_{i+1}q_{i}q_{i+1} \quad (3\le i \le n-2), \\
%      q_{n}q_{3}q_{n}=q_{3}q_{n}q_{3},  \\
%      q_{i}q_{j}=q_{j}q_{i} \quad (3\le i <j \le n-1, |i-j|\ge 2), \\
%      q_{i}q_{n}=q_{n}q_{i}\quad (4\le i \le n-1).
%   \end{gather*}
\end{thm}

We will also need the following lemma.
\begin{lemma}\label{lem:En_conjugate_commutators}
    Let $A_{\Gamma}$ be an Artin group with cyclic abelianization.  If two elements $g,h\in A_{\Gamma}'$ are conjugate in $A_{\Gamma}$ by an element $x$ and $[g,s_i]=1$ for some $i$, then $g$ and $h$ are conjugate in $A_{\Gamma}'$ by $xs_i^\ell$ where $\ell$ is -1 times the image of $x$ in $A_{\Gamma}^\ab\cong \Z$.
\end{lemma}

\begin{proof}
    The map $A_\Gamma\to A_\Gamma^\ab\cong \Z$ maps each Artin generator $s_j$ to $1\in \Z$. Thus, $xs_i^\ell\in A_{\Gamma}'$ as it has exponent sum $\ell-\ell=0$.  Now, $xs_i^k g s_i^{-k}x^{-1}=xgx^{-1}=h$.
\end{proof}

\begin{lemma}\label{lem:E7_comm_TSS}
    Let $n\geq 6$. The set $\{s_2 s_1^{-1} , s_5 s_1^{-1} , s_{5+2}s_1^{-1}, \dots \}$ forms a commutative totally symmetric set of size $\frac{n-2}{2}$ in $\cale'_n$ when $n$ is even, and it forms a commutative totally symmetric set of size $\frac{n-1}{2}$ in $\cale'_n$ if $n$ is odd.
\end{lemma}
\begin{proof}
    All elements in this set commute and the set has the prescribed size. Observe that $\cale_n$ has only odd labelled edges, and recall that this implies an element $g \in \cale_n$ is in the commutator subgroup exactly when $g$ is in the kernel of the word length map to $\mathbb{Z}$, i.e., when there is a word in the Artin generators representing $g$ such that the sum of the powers of all Artin generators which appear in the word is 0. From this, we deduce that each element in the set $\{s_2 s_1^{-1}, s_5 s_1^{-1}, s_{5+2} s_1^{-1}, \dots \}$ is contained in $\cale'_n$. 

    Notice that $\{s_2 , s_5, s_{5+1}, \dots \}$ is contained in a braid-type standard parabolic subgroup $\langle s_2 , s_4, s_5, s_6, \dots \rangle$, so there is an element \emph{of the braid subgroup} that induces all possible permutations of this collection \cite{ChudnovskyKordekLiPartin20,KordekMargalit22}. Choose some permutation $\sigma$, and let the corresponding conjugating element be $g$. The entire parabolic subgroup commutes with the Artin generator $s_1$, so by \cref{lem:En_conjugate_commutators}, the element $g s_1^l$ for some $l$ induces the desired permutation in $\cale'_n$. 
\end{proof}

We can build slightly larger totally symmetric sets in $\cale_n'$ if we do not require commutativity.
\begin{lemma}
    Let $n\geq 6$. When $n$ is even, the set $$\{s_2 s_4^{-1}, s_3 s_4^{-1}, s_5 s_4^{-1}, s_{5+2}s_4^{-1}, \dots, s_{n-1}s_4^{-1}\}$$ forms a (non-commutative) totally symmetric set of size $\frac{n}{2}$ in $\cale'_n$. When $n$ is odd, the set $$\{s_2 s_4^{-1}, s_3 s_4^{-1}, s_5 s_4^{-1}, s_{5+2}s_4^{-1}, \dots, s_{n-2}s_4^{-1}\}$$ forms a (non-commutative) totally symmetric set of size $\frac{n-2}{2}$ in $\cale'_n$.
\end{lemma}
\begin{proof}
    It is clear that the sets have the correct size. When $n \geq 7$, we see that every element of the proposed set commutes with the Artin generator $s_7$, so by Lemma \ref{lem:En_conjugate_commutators}, it suffices to construct appropriate conjugating elements in $\cale_7$. As in the previous cases, elements of the form $s_7 s_4^{-1} = s_7$, $s_9 s_4^{-1} = s_9$, etc. can be permuted via conjugation by an element in the braid-type standard parabolic subgroup $\langle s_7 , s_8 , \dots , s_n \rangle$. Such an element commutes with each of $s_2 s_4^{-1}, s_3 s_4^{-1},\text{ and } s_5 s_4^{-1}$, so elements which realize transpositions of this subset within the braid-type subgroup realize the same transpositions on the entire set. Because the symmetric group is generated by adjacent transpositions, it then suffices to realize the transpositions $(s_2 s_4^{-1} , s_3 s_4^{-1})$ and $(s_2 s_4^{-1} , s_5 s_4^{-1})$ by conjugation.

    Conjugation by the inclusion of the Garside element of the $D_5$ Artin group $\langle s_1 , s_2 , s_3 , s_4 , s_5 \rangle$ fixes $s_3, s_4, s_7, s_{7+2}, \dots,$ while permuting $s_2$ and $s_5$, so it induces the transposition $(s_2 s_4^{-1} , s_5 s_4^{-1})$. Similarly, conjugation by the inclusion of the Garside element of either $\langle s_2 , s_3 , s_4 , \dots, s_n \rangle$ or $\langle s_2 , s_3 , s_4 , \dots, s_{n-1} \rangle$ (whichever has an odd number of generators) fixes $ s_4, s_7, s_{7+2}, \dots,$ while permuting $s_2$ and $s_3$, so it induces the transposition $(s_2 s_4^{-1} , s_3 s_4^{-1})$.

    %%%%As in the previous case, the bottom 

    %For the transposition $(s_3 s_4^{-1}, s_5 s_4^{-1})$, we apply precisely the same argument as in the proof of \cref{lem:En_comm_tss}, first conjugating by the inclusion of the Garside element of the braid subgroup $\langle s_7 , s_8 , \cdots s_n \rangle$ to move the elements of our set outside the star of the unique standard parabolic $\cale_6$ subgroup, we conjugate by the inclusion of the Garside element of this $\cale_6$ subgroup to permute $s_3$ and $s_5$ while pointwise fixing $\{s_2, s_5, s_8, s_9, \cdots s_n\}$, and then we conjugate again in the braid group to move the other generators back. 

    Now we consider the case $\cale_6$. Here, there is no Artin generator which commutes with all of the elements in the proposed set, so we must show that the desired conjugations can be realized in $\cale_6'$ more directly. Let $\Delta_{D_4}$ denote the inclusion into $\cale_6$ of the Garside element of $\langle s_3 , s_4 , s_2 , s_5 \rangle$. The Garside element of $D_4$ is central in $D_4$, so in particular, this element commutes with each element in $\{s_3 s_4^{-1}, s_2 s_4^{-1} , s_5 s_4^{-1}\}$. By \cite{BrieskornSaito72}, $\ell (\Delta_{D_5}^L) = \ell (\Delta_{D_5}^R) = 20$ and $\ell(\Delta_{D_4}) = 12$.
    %$\ell(\Delta_{E_6}) = 36$,

    %To realize the transposition $(s_3 s_4^{-1}, s_5 s_4^{-1})$ in $\cale_6'$, we can then conjugate by the element $\Delta_{E_6} \Delta_{D_4}^{-3}$, which has word length 0 and induces exactly the same permutation on all of $s_3$, $s_2$, $s_4$, and $s_5$ as $\Delta_{E_6}$. 
    To realize the transpositions $(s_2 s_4^{-1}, s_3 s_4^{-1})$ and $(s_2 s_4^{-1}, s_5 s_4^{-1})$ in $\cale'_6$, we highlight that the square of the Garside element is always central in a spherical Artin group. The desired conjugacy relations in $\cale_6$ are also realized by $(\Delta_{D_5}^R)^{3}$ and $(\Delta_{D_5}^L)^3$, each of which has length 60. The transpositions can then be realized in $\cale_6'$ by $(\Delta_{D_5}^R)^3 (\Delta_{D_4})^{-5}$ and $(\Delta_{D_5}^L)^3 (\Delta_{D_4})^{-5}$.
\end{proof}

%K: throughout the commented subsection, one % denotes things that we commented in the interest of time and will bring back later, and two %% denotes things that were already commented.
\subsection{Collapsing sets in non-braid Artin groups}

Recall that collapsing sets were defined in \Cref{D:collapsingSet}.  Thus far, we have only used that the standard generators of braid groups on 5 or more strands form a collapsing set.
In this section we will exhibit collapsing sets in more general Artin groups that will be particularly suited to studying $\calh_4$ and $\cale_n$.

We will use the following graph theoretic tools.
\begin{defn}
    The \emph{link} of vertex $s_i$ in graph $\Gamma$, denoted $\lk(s_i)$, is  the collection of vertices which are adjacent to $s_i$ in $\Gamma$. (If a graph has no 3-cycles,  this is equivalent to the standard definition.) The \emph{star} $\st(s_i)$ denotes $\lk(s_i) \cup \{v\}$, and for a subgraph $\Lambda$ of $\Gamma$, $\st(\Lambda)$ is the induced subgraph with vertex set equal to the union of $\st(s_i)$ for each $s_i \in \Lambda$.
\end{defn}

\begin{lemma}\label{lem:artin-gen-collapsing}
    Let $A_{\Gamma}$ be an irreducible Artin group satisfying the following conditions: 
    \begin{enumerate}
        \item $|V(\Gamma)| \geq 4$,
        \item $\Gamma$ has no cycles of length less than 5,
        \item All edge labels in $\Gamma$ are odd, and
        \item No two valence $1$ vertices of $\Gamma$ are adjacent to the same vertex.
    \end{enumerate}
    Then the Artin generators form a collapsing set.
\end{lemma}
\begin{proof}
Throughout the proof, we conflate an Artin generator $s_i$ with the corresponding vertex in the defining $\Gamma$ for ease of notation. 
    We first prove that all of the Artin groups of interest have the following property.

    \smallskip
    
    \noindent\textbf{Claim:} \emph{For any two Artin generators $s_i$ and $s_j$ which are not adjacent, there is at least one Artin generator $s_k$ such that $s_k \in \lk(s_i) \cup \lk(s_j)$ but $s_k \notin \lk(s_i) \cap \lk(s_j)$, where $\lk(s_i)$ is considered in $\Gamma$.}
    %\yv{We're technically abusing notation here letting $s_j$ denote both a vertex and a generator; I imagine this is standard in Artin groups, but I tried to make it more clear.}\katie{I added a sentence to clarify it as well}

    \smallskip
    
    \noindent\textbf{Proof of claim:} Choose a pair $s_i$ and $s_j$ which are not adjacent in $\Gamma$. First, suppose $s_i$ and $s_j$ each have valence 1. The unique vertex $v$ adjacent to $s_i$ is disjoint from $s_j$ by assumption, so we are done. %Otherwise, notice that $v$ must have valence at least 3 because $|V(\Gamma)| \geq 4$. In particular, there is some vertex $w$ adjacent to $v$ such that $\{s_i , s_j , v , w\}$ span a $D_4$ subgroup. By assumption on $\Gamma$, this subgraph must be contained in a $\cale_6$ subgraph. This is a contradiction because in a $\cale_6$ subgroup, the valence 1 vertices are not adjacent to the valence 3 vertex. %any pair of valence 1 vertices in $\cale_6$ are at least distance 3 from one another in the defining graph for $\cale_6$, and their inclusion into $\Gamma$ cannot introduce any shortcuts
    %\yv{Do you mean relations here?} %because $\Gamma$ is a tree.
    %\katie{I meant literally a path in the tree, but I don't think that was the best explanation. I think it's much clearer to replace the $D_4$ subgraph condition with the equivalent condition ``there is no vertex of valence at least 3 which is adjacent to two valence 1 vertices" so I've done that}
    
    Now suppose at least one of $s_i$ and $s_j$ has valence at least two. Without loss of generality, say the link of $s_j$ contains vertices $s_{k}$ and $s_{k'}$. The vertices $s_k$ and $s_{k'}$ cannot both be adjacent to $s_i$ because there would then be a 4-cycle in $\Gamma$ of the form $s_i \rightarrow s_k \rightarrow s_j \rightarrow s_{k'}$, so at least one of $s_k$ and $s_{k'}$ is adjacent to $s_j$ but disjoint from $s_i$. \hfill$\blackdiamond$

    \smallskip
    
    We can now justify the following stronger claim.

    \smallskip

    \noindent\textbf{Claim:} \emph{Let $\varphi: A_{\Gamma} \rightarrow G$ be a group homomorphism. If $\varphi(s_i) = \varphi(s_j)$ for any pair of Artin generators, then every $s_k \in \st(s_i) \cup \st(s_j)$ has $\varphi(s_k) = \varphi(s_i) = \varphi(s_j)$, where $\st(s_i)$ is considered in $\Gamma$.}

    \smallskip
    
    \noindent\textbf{Proof of claim:} %Since $\Gamma$ is connected and has at least four vertices, $\text{Lk}(s_i) \cup \text{Lk}(s_j) \neq \varnothing$.
    If $s_i$ and $s_j$ are adjacent, then every vertex of $\lk(s_j) - \{s_i\}$ is disjoint from $s_i$ and vice versa since $\Gamma$ has no 3-cycles. In particular, for any vertex $s_k$ of $\lk(s_i)$ or $\lk(s_j)$, the image $\varphi(s_k)$ satisfies Artin relations of length both 2 and an odd integer with $\varphi(s_i) = \varphi(s_j)$ by the assumption that edge labels are odd. By Lemma \ref{L:gcdRelation}, this implies that $\varphi(s_k) = \varphi(s_i) = \varphi(s_j)$. 

    Now suppose $s_i$ and $s_j$ are not adjacent. For any vertex $s_k$ of $\lk(s_j)$ which is disjoint from $s_i$ (or vice versa), $s_k$ satisfies an odd length Artin relation with $s_j$. It also satisfies a commuting relation with $s_i$, so by Lemma \ref{L:gcdRelation}, this implies that $\varphi(s_k) = \varphi(s_i) = \varphi(s_j)$. Thus it remains only to consider the case of $s_k \in \lk(s_i) \cap \lk(s_j)$. By the first claim, there is some $s_{k'} \in \lk(s_i)$ or $\lk(s_j)$ which is disjoint from $s_j$ or $s_i$ respectively, so Lemma \ref{L:gcdRelation} shows that $\varphi(s_{k'}) = \varphi(s_i) = \varphi(s_j)$. Suppose without loss of generality that $s_{k'}$ is in $\lk(s_j)$. The vertices $s_k$ and $s_{k'}$ cannot be adjacent because they are both contained in the link of $s_j$ and $\Gamma$ has no 3-cycles. This implies that $s_k$ satisfies an Artin relation of odd length with $s_j$ and of length 2 with $s_{k'}$. We conclude by Lemma \ref{L:gcdRelation} applied to $\varphi(s_k)$ and $\varphi(s_i) = \varphi(s_j) = \varphi(s_{k'})$. \hfill$\blackdiamond$

    Finally, notice that since $\Gamma$ is connected, either $\Lambda = \Gamma$ or $\Lambda \subsetneq \st(\Lambda)$ for any induced subgraph $\Lambda$ of $\Gamma$. To complete the proof, suppose that $\varphi(s_i) = \varphi(s_j)$ for some $s_i$ and $s_j$, and let $\Lambda$ be the induced subgraph whose vertex set is the vertex set of $\st(s_i) \cup \st(s_j)$. By the second claim, every vertex in $\Lambda$ has the same image under $\varphi$. We can now reapply the second claim to show that all vertices in $\st(\lambda)$ also have this image. Because this is a strictly increasing process and $\Gamma$ is finite, we obtain that all vertices in $\Gamma$ have the same image after finitely many repetitions.
\end{proof}
\begin{corollary}
    The Artin generators form a collapsing set in $\calh_4$ and in $\cale_n$ for all $n$.
\end{corollary}

While the above proof requires at least 4 vertices, the desired statement is also true in $\calh_3$. 

\begin{lemma}
    The Artin generators in $\calh_3$ form a collapsing set.
\end{lemma}
\begin{proof}
    Observe that for any pair of Artin generators $s_i$ and $s_j$, the third generator $s_k$ satisfies coprime length Artin relations $m_{ik}$ and $m_{jk}$. The result then follows immediately from \cref{L:gcdRelation}.
\end{proof}

We conclude this subsection with some applications which will be useful in the remainder of the paper.

%We now show that restrictions of non-cyclic homomorphisms to standard parabolic subgroups must, in many cases, also be non-cyclic. We will make frequent use this, for example, to show that the image of $\cale_6$ is not cyclic under non-cyclic homomorphisms of $\cale_8$, allowing us to apply the bounds on quotient size we obtain for $\cale_6$ to the $\cale_8$ case. 
\begin{lemma}\label{lem:E8_non_cyclic_E6}
    Let $A_{\Gamma}$ be an odd-labelled Artin group where the Artin generators form a collapsing set, and let $A_X$ be an irreducible standard parabolic subgroup with $|X| > 1$. Let $\alpha \colon A_{\Gamma} \rightarrow Q$ be a homomorphism such that the restriction of $\alpha$ to $A_X$ has cyclic image. The entire homomorphism $\alpha$ has cyclic image.
\end{lemma}
\begin{proof}
    Since $\Gamma$ is odd-labelled, the Artin generators $\{s_i\} = X$ are all conjugate. The image $\alpha(A_X)$ is abelian, so every such generator is identified under $\alpha$. By assumption, the Artin generators of $A_{\Gamma}$ form a collapsing set, so in fact all Artin generators coincide and the image of $A_\Gamma$ is cyclic.%In particular, there is one generator $s_i$ in this set such that $\alpha(s_i)$ generates $\alpha(A_X)$. Let $\overline{s}$ denote this generator for $\alpha(A_X)$. 
    %Since the Artin generators $\{s_i\} = X$ form a generating set for the standard parabolic subgroup $A_X$, each of their images generate the cyclic subgroup $\alpha(A_X)$. In particular, there is one generator $s_i$ in this set such that $\alpha(s_i)$ generates $\alpha(A_X)$. Let $\overline{s}$ denote this generator for $\alpha(A_X)$. 

    %We claim that every Artin generator in the set $X$ has image $\overline{s}$. To see this, it suffices to recall that because $X$ is irreducible and all edge labels are odd, the generators in this set are pairwise conjugate \emph{by elements of $A_X$}. Let $\alpha(s_j) = \overline{s}^k$ for some $j \neq i$, let $g$ be an element of $A_X$ which conjugates $s_i$ to $s_j$, and let $m$ be an integer such that $\alpha(g) = \overline{s}^m$. In the quotient, the relation $gs_ig^{-1}=s_j$ becomes $s^m s s^{-m} = s^k$, which implies either $k = 1$ or the entire cyclic subgroup $\alpha(A_X)$ is trivial. The images of the generators coincide in either case.
    %Since, by assumption, all the Artin generators form a collapsing set, then all Artin generators coincide and the image of $A_\Gamma$ is cyclic.
    %The statement then follows from \cref{lem:artin-gen-collapsing}.
\end{proof}

%Lemma~\ref{lem:E8_non_cyclic_E6} implies that we may bound the smallest nonabelian quotient of certain Artin groups, $A_\Gamma$, using the smallest nonabelian quotient of $A_\Lambda$.

%\begin{corollary}\label{cor:parabolic_braid}
%    Let $A_\Gamma$ be an Artin group satisfying the conditions of \cref{lem:E8_non_cyclic_E6}, and suppose that the standard parabolic subgroup corresponding to $\Lambda \subset \Gamma$ is a braid group. 
%    The smallest nonabelian quotient of $A_\Lambda$ lower bounds the order of the smallest nonabelian quotient of $A_\Gamma$
    % $\eta(A_\Gamma) \geq \eta(A_\Lambda)$.
%\end{corollary}
%\nancy{I think this corollary gets used alot but not nec referenced. We should check to make sure we reference it when used.}
%\katie{Thomas: I think you did something with this in the other section? If so, this instance of it can probably be removed.}
%\thomas{The conditions here are curated to $E_8$, so this can stay here.}
%\katie{this wasn't referenced anywhere, so I removed it. I think all references to it were actually to the version which appears earlier}

Lastly, we make the following observation.
\begin{lemma}\label{lem:centraliser-proper-subgp}
    Let $A_{\Gamma}$ be an odd-labelled Artin group where the Artin generators form a collapsing set, and let $q : A_{\Gamma} \rightarrow Q$ be a non-cyclic quotient. For each generator $s_i$, the image of the standard parabolic subgroup generated by $\Gamma - \text{st}(s_i)$ is a proper subgroup of $Q$.
\end{lemma}
\begin{proof}
    Suppose that the image of the standard parabolic subgroup generated by $\Gamma - \text{st}(s_i)$ is all of $Q$. This parabolic subgroup commutes with $s_i$ in $A_{\Gamma}$, so we must have that $q(s_i) \in Z(Q)$. If the Artin generating set is collapsing, then the defining graph $\Gamma$ is necessarily connected, so there is some $s_k \in \text{lk}(s_i)$. Since $\Gamma$ is odd-labelled, the generators $s_k$ and $s_i$ satisfy an odd-length Artin relation. In the quotient, however, we then have that $q(s_k)$ and $q(s_i)$ satisfy both a commuting relation and an Artin relation of odd length. By \cref{L:gcdRelation}, the group $Q$ is cyclic.
\end{proof}

\subsection{Element orders}\label{subsec:element-orders}
In this section, we will show that every non-cyclic finite quotient of an Artin group of type $\cale_6$, $\cale_7$, or $\calh_4$ must have torsion elements of certain orders. 

Before we begin, we give an overview of the proof strategy. Elements of a spherical Artin group $A_{\Gamma}$ whose images in $A_{\Gamma}/Z(A_{\Gamma})$ are torsion of various orders were constructed in \cite{Soroko2021} for Artin groups of type $\cale_6$, $\cale_7$, and $\calh_4$. The precise elements and their orders will be stated later in the section. We aim to show that these torsion elements survive in every non-cyclic finite quotient of $A_{\Gamma}$, and that suitable powers yield torsion elements of the same order.  

To do this, we will consider an element which has order $k$ in $A_{\Gamma}/Z(A_{\Gamma})$, say $\epsilon$ of order $10$. In any quotient of $A_{\Gamma}$ where the generator of $Z(A_{\Gamma})$ has finite order $q$ for some $q$, the order of $\epsilon$ must divide $10q$. Thus to show that any non-cyclic finite quotient of $A_{\Gamma}$ has an element of order $10$, it suffices to show that $A_{\Gamma}/ \langle \epsilon^2 \rangle$ and $A_{\Gamma} / \langle \epsilon^5 \rangle$ are cyclic.%its order must be divisible by $\text{lcm}(2,5) = 10$. This will in turn imply that any non-cyclic finite quotient of $A_{\Gamma}$ has an element of order 10 by taking the image of a suitable power of $\epsilon$. 

To show that the relevant quotients are cyclic, we will show that in any quotient of $A_{\Gamma}$ where the required powers of $\epsilon$ are trivial, the required powers being $\epsilon^2$ and $\epsilon^5$ in the example, a relation is imposed which forces two Artin generators to have the same image. We will then conclude by \cref{lem:E8_non_cyclic_E6}.

Primarily, we will do this by computing conjugates of a particular Artin generator by powers of $\epsilon$ and reducing the resulting word via the Artin relations. This process is necessarily fairly computation-heavy. As such, we provide the computations in an appendix. The appendix presents these computations in a compact notation which allows an interested reader to reconstruct each step.

 %Second, we recall a result of Paris, which says that for any two pairs of Artin generators $(s_{i}, s_j )$ and $(s_k , s_l)$ such that each pair satisfies a braid relation and all four generators are contained in a standard parabolic braid subgroup, there is a single element $g$ of $A_{\Gamma}$ such that $g s_i g^{-1} = s_k$ and $g s_j g^{-1} = s_l$. In particular, if any two generators $s_i$ and $s_j$

\begin{lemma}\label{lem:E6_elemOrders}
    Let $G = \cale_6 / Z(\cale_6)$. The following conclusions hold:
    \begin{enumerate}
        \item The image of $\epsilon_{12} = s_4 s_2 s_3 s_1 s_5 s_6$ has order 12.
        \item The image of $\epsilon_9 = s_4 s_2 s_5 s_4 s_3 s_1 s_5 s_6$ has order 9.
        \item The image of $\epsilon_8 = s_4 s_3 s_1 s_5 s_4 s_2 s_3 s_6 s_5$ has order 8.
        \item The groups $\cale_6 / \langle \langle \epsilon_{12} \rangle \rangle$, $\cale_6 / \langle \langle \epsilon_{9} \rangle \rangle$, $\cale_6 / \langle \langle \epsilon_{8} \rangle \rangle$, $\cale_6 / \langle \langle \epsilon_{12}^{4} \rangle \rangle = \cale_6 / \langle \langle \epsilon_{9}^{3} \rangle \rangle$, $\cale_6 / \langle \langle \epsilon_{12}^{3} \rangle \rangle = \cale_6 / \langle \langle \epsilon_{8}^{2} \rangle \rangle$, $\cale_6 / \langle \langle \epsilon_{8}^{4} \rangle \rangle$, $\cale_6 / \langle \langle \epsilon_{12}^{2} \rangle \rangle$, and $\cale_6 / \langle \langle \epsilon_{12}^{6} \rangle \rangle$ are all cyclic.
    \end{enumerate}
    In particular, any finite non-cyclic quotient of $\cale_6$ has elements of order $8$, $9$, and $12$.
\end{lemma}

\begin{proof}
    The first three conclusions and the equalities in the fourth follow from Theorem 7 of \cite{Soroko2021}. Computations which demonstrate that the quotients $\cale_6 / \langle \langle \epsilon_{12}^{6} \rangle \rangle$, $\cale_6 / \langle \langle \epsilon_{9}^{3} \rangle \rangle$, and $\cale_6 / \langle \langle \epsilon_{8}^{4} \rangle \rangle$ are all cyclic are given in the appendix sections \ref{apx:E6-order12}, \ref{ap:E6-order9}, and \ref{ap:E6-order8}. The other quotients must also be cyclic, as they are further quotients of these three.
\end{proof}

\begin{lemma}\label{lem:E7_elemOrders}
    Let $G=\cale_7/Z(\cale_7)$.  The following conclusions hold:
    \begin{enumerate}
        \item the image of $\epsilon_7=s_4s_2 s_7s_6s_5  s_4s_2  s_3s_1$ in $G$ has order $7$;
        \item the image of $\epsilon_9=s_4s_2 s_3s_1  s_5s_6s_7$ in $G$ has order $9$;
        \item the groups $\cale_7/\langle\langle\epsilon_7\rangle\rangle$, $\cale_7/\langle\langle\epsilon_9\rangle\rangle$, and $\cale_7/\langle\langle\epsilon_9^3\rangle\rangle$ are cyclic.
    \end{enumerate}
    In particular, any finite non-cyclic quotient of $\cale_7$ contains elements of order $7$ and $9$.
\end{lemma}
\begin{proof}
    The first and second parts are proved in \cite[Theorem~7]{Soroko2021}.  Computations which demonstrate that $\cale_7/\langle\langle\epsilon_7\rangle\rangle$ and $\cale_7/\langle\langle\epsilon_9^3\rangle\rangle$ are cyclic are provided in appendix sections \ref{apx:E7-order7} and \ref{apx:E7-order9}. The quotient $\cale_7/\langle\langle\epsilon_9\rangle\rangle$ must also be cyclic, because it is a further quotient.
\end{proof}

\begin{lemma}\label{lem:H4_elem_orders}
    Let $G = \calh_4 / Z(\calh_4)$. The following conclusions hold:
    \begin{enumerate}
        \item The image of $\epsilon_{10} = s_1 s_2 s_1 s_2 s_3 s_4 $ has order $10$.
        \item The image of $\epsilon_{15} = s_1 s_2 s_3 s_4$ has order $15$.
        \item The groups $\calh_4/\langle\langle\epsilon_{10}\rangle\rangle$, $\calh_4/\langle\langle\epsilon_{10}^2\rangle\rangle$, $\calh_4/\langle\langle\epsilon_{10}^5\rangle\rangle$, $\calh_4/\langle\langle\epsilon_{15}\rangle\rangle$, $\calh_4/\langle\langle\epsilon_{15}^3\rangle\rangle$, and $\calh_4/\langle\langle\epsilon_{15}^5\rangle\rangle$ are all cyclic.
    \end{enumerate}
    In particular, any finite non-cyclic quotient of $\calh_4$ contains elements of order 10 and 15.
\end{lemma}
\begin{proof}
    The first and second parts again follow from \cite[Theorem~7]{Soroko2021}. Computations which demonstrate that $\calh_4/\langle\langle\epsilon_{10}^2\rangle\rangle$ and $\calh_4/\langle\langle\epsilon_{10}^5\rangle\rangle$ are cyclic are provided in appendix section \ref{apx:H4-order10}. Computations which demonstrate that $\calh_4/\langle\langle\epsilon_{15}^3\rangle\rangle$ and $\calh_4/\langle\langle\epsilon_{15}^5\rangle\rangle$ are cyclic are provided in appendix section \ref{apx:H4-order15}. The other two quotients must also be cyclic because each is a further quotients of one of these four.
\end{proof}

\section{The Artin groups \texorpdfstring{$\calf_4$}{F4}, \texorpdfstring{$\calh_3$}{H3}, \texorpdfstring{$\calh_4$}{H4}}\label{sec:f4}

In this section, we determine the smallest nonabelian quotient for the groups $\calf_4, \calh_3$ and $\calh_4$.

\subsection{The Artin group \texorpdfstring{$\calf_4$}{F4}}

\begin{thm}\label{thm:F4quotient}
    The smallest non-abelian quotient of $\calf_4$ is $\sym(3)$. 
\end{thm}
\begin{proof}
    The group $\sym(3)$ is the smallest non-abelian group, so it suffices to show that $\sym(3)$ is a quotient of $\calf_4$. The Coxeter group of $\calf_4$ is the group $2^{1+4} : (\sym(3) \times \sym(3))$ \cite{Wilson09}. One obtains a quotient $q: \calf_4 \rightarrow \sym(3)$ by composing the canonical projective homomorphism $\calf_4 \to (\sym(3) \times \sym(3))$ with a projection map to one of the coordinates of $(\sym(3) \times \sym(3))$.
\end{proof}

\begin{remark}
    The smallest non-solvable quotient of $\calf_4$ is $\alt(5)$. As in the theorem, it suffices to verify that $\alt(5)$ is a quotient. To obtain a quotient map, one first retracts the Artin group $\calf_4$ onto one of its $\cala_2$ standard parabolic subgroups. The group $\cala_2$ is isomorphic to $\Tilde{SL}_2(\mathbb{Z})$. One can then map onto $SL_2(\mathbb{Z})$ and further onto $\PSL_2(5) \cong \alt(5)$.
\end{remark}

\subsection{The Artin group \texorpdfstring{$\calh_3$}{H3}}

We construct two epimorphisms $\calh_3\to \alt(5)$.  The first factors through the Coxeter group $W(\calh_3)$ and sends a generator to an involution.  The second sends generators to $5$-cycles and can be given as follows
\[
\psi\colon \calh_3 \twoheadrightarrow \alt(5) \quad \text{by} \quad s_1 \mapsto (1\ 2\ 3\ 4\ 5) \quad s_2\mapsto (1\ 4\ 5\ 2\ 3) \quad s_3\mapsto (1\ 5\ 4\ 3\ 2).
\]
One easily checks that $\psi(s_1s_2s_1s_2s_1)=\psi(s_2s_1s_2s_1s_2)$, that $\psi(s_1s_3)=\psi(s_3s_1)$, and that $\psi(s_2s_3s_2)=\psi(s_2s_3s_2)$.
Note that the homomorphisms are identified by $\Aut(\alt(5))$ since the generators of $\calh_3$ are sent to elements of different orders.

% \begin{lemma}
%     The is no surjective homomorphism $G=\calh_3/\langle\langle s_1^3\rangle \rangle\to \alt(5)$.
% \end{lemma}
% \begin{proof}
%     A long tedious and direct computation (that we omit for brevity) shows that the relations in $\calh_3$ are not satisfied by any triple of $3$-cycles in $\alt(5)$. \yv{I think we need to add in a few more details than this.}
% \end{proof}

\begin{thm}\label{thm:H3}
    If $G$ is isomorphic to $\calh_3$, then the smallest non-abelian quotient of $G$ is $\alt(5)$.  Up to automorphisms of the image, there are two such homomorphisms.
\end{thm}
\begin{proof}
    Every non-cyclic quotient of $\calh_3$ is non-solvable because $\calh_3'$ is perfect.  We have already exhibited two distinct quotients up to automorphisms of the image. The only other possible element order for an element in $\alt(5)$ is $3$.  A long tedious and direct computation (that we omit for brevity) shows that the relations in $\calh_3$ are not satisfied by any triple of $3$-cycles in $\alt(5)$.
\end{proof}

%%%%%%%%%%%%%%%%%%%%%%%%%%%%%%%%%%%%%%%%%%%%%%%%%%%%%%%%%%
\subsection{The Artin group \texorpdfstring{$\calh_4$}{H4}}\label{sec:h4}

In this section we will compute the smallest non-abelian quotient of $\calh_4$.  Throughout we will use \Cref{StandardFiniteGroupsNotation}. Recall that $\calh_4$ has cyclic abelianization, and its commutator subgroup is perfect by \cref{thm:comm-perfect}, so we are in the setting of Lemma \ref{lem:min_quotient_En}.

\begin{table}[h]
    \centering
    \begin{tabular}{|c|p{4.5cm}|p{2cm}|l|c|c|}
       \hline
       Family & Group & Order & $\Out(G)$ & 10s & 15s \\
       \hline
       \hline
       \multirow{3}{*}{Alternating} & $\alt(5)$  & 60 & $2$ & n & n\\
                                    & $\alt(6)$  & 360 & $2^2$ & n & n \\
                                    & $\alt(7)$  & 2,520 & $2$ & n & n \\
       \hline
      \multirow{7}{*}{$\texttt A_1$} &  {$\PSL_2(7)$} & {168} & 2 & n & n \\
                                    & $\PSL_2(8)$ & 504 & 3 & n & n \\
                                    & $\PSL_2(11)$ & 660 & 2 & n & n \\
                                    & $\PSL_2(13)$ & 1,092 & 2 & n & n \\
                                    & $\PSL_2(17)$ & 2,488 & 2 & n & n \\
                                    & $\PSL_2(19)$ & 3,420 & 2 & y & n \\
                                    & $\PSL_2(16)$ & 4,080 & 4 & n & y\\
       \hline
       \multirow{1}{*}{$\texttt A_2$} & $\PSL_3(3)$ & 5,616   & $2$ & n & n \\
       \hline
       \multirow{1}{*}{$^2\texttt A_2$} & $\SU_3(3)$  & 6,048   & $2$ & n & n \\
        \hline
    \end{tabular}
    \caption{Simple groups with order at most 7200.}
    \label{tab:H4simple}
\end{table}

% \begin{lemma}\label{lem:H4_7200}
%     A smallest non-cyclic quotient of $\calh_4$ has order 7200.
% \end{lemma}
% \begin{proof}
%     Element orders.
% \end{proof}

\begin{lemma}\label{lem:alt5sq.2}
    There is exactly one isomorphism class of groups $\alt(5)^2.2$ with the property that every proper quotient is cyclic.  The group is isomorphic to $\alt(5)\wr2$.
\end{lemma}
\begin{proof}
    We have that $\Out(\alt(5))^2\cong 2\wr2$.  Write an element as $(a,b;c)$ where $(a,b)\in2^2$ and $c\in 2$.  We have that $(1,0;0)$ and $(0,1;0)$ correspond to the non-trivial element in $\Out(\alt(5))\cong 2$ for each factor.  We call these outer classes \emph{symmetric}.  The non-trivial central element $(1,1;0)$ is a symmetric outer automorphism on each factor.  We call this outer class \emph{central}.  The element $(0,0;1)$ is the outer automorphism that permutes the two factors of $\alt(5)$.  We call this outer automorphism \emph{wreathing}.  The final involution in $2\wr2$ is the element $(1,1;1)$.  We call this outer class \emph{diagonal}.
    
    Since $Z(\alt(5)^2))=1$, an equivalence class of group extensions
    \[\begin{tikzcd}
        1 \ar[r] & \alt(5)^2 \ar[r] & Q \ar[r] & 2 \ar[r] & 1
    \end{tikzcd}\]
    is given by a homomorphism $\phi\colon 2\to \Out(\alt(5)^2)$, see \cite[Corollary IV.6.8]{Brown1982}.  Note that different equivalence classes of group extensions may give rise to isomorphic groups.  The group $2\wr2$ contains $5$ elements of order $2$.   Let $\phi\colon 2\to\Out(\alt(5)^2)$ and let $Q_\phi$ denote the resulting extension group.  We analyse each possible extension in turn.  We call $\phi$ trivial, symmetric, central, wreathing, or diagonal if the non-trivial element of $2$ is mapped to an element with the corresponding property.

    Suppose that $\phi$ is trivial.  In this case, $Q_\phi\cong \alt(5)^2\times 2$. This clearly has a non-cyclic proper quotient, so we may rule it out.

    Suppose that $\phi$ is symmetric.  Clearly the resulting extension group is isomorphic to $\alt(5)\times\sym(5)$.  This evidently has a non-cyclic proper quotient so we may rule it out.

    Suppose that $\phi$ is central.  We claim that either tautological $\alt(5)$ subgroup is normal.  For i=1,2 write $(\sigma_i,\tau_i;x_i)\in Q_\phi$ with $(\sigma_i,\tau_i)\in\alt(5)^2$ and $x_i\in 2$ acts diagonally.  The group multiplication is then given by a semidirect product
    \[(\sigma_1,\tau_1;x_1)(\sigma_2,\tau_2;x_2)=(\sigma_1x_1(\sigma_2),\tau_1x_1(\tau_2);x_1x_2)\]
    and an inverse of $(\sigma,\tau;x)$ is given by $(x(\sigma^{-1},x(\tau^{-1}),x)$.
    For $(\mu,1;0)\in \alt(5)\times\{1\}$ we have
    \begin{align*}
    (\sigma,\tau;x)(\mu,1;0)(\sigma,\tau;x)^{-1}&=(\sigma x(\mu),\tau,x)(\sigma^{-1},x(\tau^{-1}),x)\\
    &=(\sigma x(\mu)x(\sigma^{-1}),\tau x(x(\tau^{-1})),xx)\\
    &=(\sigma x(\mu)x(\sigma^{-1}),1;0)
    \end{align*}
    which is contained in $\alt(5)\times\{1\}$.  A similar argument also shows that $\{1\}\times\alt(5)$ is normal.  In either case, the quotient by a normal $\alt(5)$ subgroup has shape $\alt(5).2$ and is therefore non-abelian.

    The two remaining cases are when $\phi$ is wreathing or diagonal.  Clearly in either case the normal closure of any $\alt(5)$ subgroup is the subgroup $\alt(5)^2$.  Thus, one readily checks that every proper quotient of $Q_\phi$ is cyclic.

    Let $Q_1\cong \alt(5)\wr 2$ and let $Q_2$ denote the group given by $Q_\phi$ where $\phi$ is diagonal.  It remains to show that $Q_1\cong Q_2$.  

    We have that $Q_2$ is a subgroup of $\Aut(\alt(5)^2)\cong \sym(5)\wr 2$.  In particular we may view $Q_2$ as a subgroup of $\sym(10)$.  Indeed, we may generate $Q_2$ by the two copies of $\alt(5)$, one on $\{1,\dots,5\}$ and the other on $\{6,\dots,9,0\}$ and the element 
    \[x=(1\ 2)(6\ 7)(1\ 6)(2\ 7)(3\ 8)(4\ 9)(5\ 0)=(1\ 7)(2\ 6)(3\ 8)(4\ 9)(5\ 0).\]
    The latter element is $\sym(5)\wr 2$ conjugate by $y=(1\ 2)$ to $z=(1\ 6)(2\ 7)(3\ 8)(4\ 9)(5\ 0)$.  The element $(1\ 2)$ normalises $\alt(5)^2$, so the group $Q_2$ is conjugate by $y$ to $\langle \alt(5)^2,z\rangle$.  This later group is exactly $\alt(5)\wr 2<\sym(10)$ as required.
\end{proof}

% \begin{lemma}
%     \begin{enumerate}
%         \item $\cali_2(5)/\langle\langle s_1^3\rangle \rangle \cong 3\times \SL_2(5)$;
%         \item the smallest non-abelian quotient of $\cali_2(5)/\langle\langle s_1^5\rangle \rangle$ is $\alt(5)$.
%     \end{enumerate}
% \end{lemma}

\begin{thm}\label{thm:H4}
    The smallest non-cyclic quotient of $\calh_4$ is $W(\calh_4)/Z(\calh_4)$, and it is unique up to automorphisms of the image.
\end{thm}
\begin{proof}
     By Lemma \ref{lem:min_quotient_En}, either $Q$ is almost simple, or it is a subgroup of $\Aut(S^n)$ for some simple group $S$ and $n >1$. For all $S$ and $n$ except $\alt(5)^2$, $\Aut(S^n) > |W(\calh_4)|$, so $Q$ is either almost simple or a subgroup of $\Aut(\alt(5)^2)$. We first apply element order considerations to rule out almost simple groups of order at most $7200 = |W(\calh_4)|$.  
     
     The simple groups $Q$ with order at most 3600 are $\alt(5)$, $\PSL_2(7)$, $\alt(6)$, $\PSL_2(8)$, $\PSL_2(11)$, $\PSL_2(13)$, $\PSL_2(17)$, $\alt(7)$, $\PSL_2(19)$.  One easily verifies via the ATLAS that $\Aut(Q)$ does not contain an element of order 15 for any of these.  The simple groups $Q$ with order between 3600 and 7200 are $\PSL_2(16)$, $\PSL_3(3)$, $\PSU_3(3)$, and $\PSL_2(23)$.  In each of these cases, after consulting the ATLAS, $Q$ contains has no element of order $10$.  Note we need not check $\Aut(Q)$ for these, since $2|Q|>7200$.
    
    We are now left with groups of the shape $\alt(5)^2.2$.  By \Cref{lem:alt5sq.2}, there is exactly one group of this shape which does not admit a proper non-cyclic quotient, namely $\alt(5)\wr2$, so necessarily this group must be isomorphic $W=W(\calh_4)/Z(\calh_4)$ as required.  One easily checks the quotient is unique up to automorphisms of the image using MAGMA, see \Cref{Code:6.2}.
\end{proof}

% \begin{table}[]
%     \centering
%     \begin{tabular}{|c|c|c|c|c|c|}
%     \hline
%     Name & Order & Class size & Centraliser shape & Centraliser size\\ 
%     \hline 
%  1   & 1  & 1   & $\alt(5)\wr 2$ & 7200  \\
%  $2_A$   & 2  & 30  & $2^2\times \alt(5)$ & 240\\
%  $2_B$   & 2  & 60  & $2\times \alt(5)$   & 120 \\
%  $2C$   & 2  & 225 & $2^2\wr 2$  & 8 \\
% $3_A$   & 3  & 40  &  $\GL_2(4)$ &  180  \\
%  $3_B$   & 3  & 400 & $3\times \sym(3)$  &36  \\
%  $4$   & 4  & 900 & $2\times 4$  \\
% $5_A$ & 5  & 24     & $5\times\alt(5)$   \\
% $5_A'$  & 5  & 24  & $5\times\alt(5)$     \\
%  $5_B$  & 5  & 144 & $5\times D_5$    \\
%  $5_B'$  & 5  & 144 & $5\times D_5$  \\
%  $5_C$   & 5  & 288 & $5^2$    \\
% $6_A$   & 6  & 600 & $2\times 6$    \\
%  $6_B$   & 6  & 1200&  $6$   \\
%  $10_A$ & 10 & 360 & $2\times 10$  \\
% $10_A'$ & 10 & 360 & $2\times 10$ \\
% $10_B$ & 10 & 720 & 10  \\
% $10_B'$ & 10 & 720 & 10  \\
% $15$ & 15 & 480 & 15  \\
%  $15'$ & 15 & 480 & 15 \\ 
%  \hline
%     \end{tabular}
%     \caption{Conjugacy classes in $\alt(5)\wr2$}
%     \label{tab:ccls_in_a5wr2}
% \end{table}

%%%%%%%%%%%%%%%%%%%%%%%%%%%%%%%%%%%%%%%%%%%%%%%%%%%%%%%%%%
\section{The Artin group \texorpdfstring{$\cale_6$}{E6}}\label{sec:e6}
In this section, we will compute the smallest non-abelian quotient of $\cale_6$.  Throughout, we will use \Cref{StandardFiniteGroupsNotation}.

\subsection{The groups}
Our proof will use the classification of finite simple groups. We recall the finite simple groups smaller than $W(\cale_6)$ in the table below.
\begin{table}[h]
    \centering
    \begin{tabular}{|c|p{4.5cm}|p{2cm}|p{3cm}|}
       \hline
       Family & Group & Order & $\Out(G)$ \\
       \hline
       \hline
       \multirow{5}{*}{Alternating} & $\alt(5)$  & 60 & $2$ \\
                                    & $\alt(6)$  & 360 & $2^2$ \\
                                    & $\alt(7)$  & 2,520 & $2$ \\
                                    & $\alt(8)$  & 20,160 & $2$ \\
                                    %& $\alt(9)$  & 181,440 & $2$ \\
       \hline
      {$\texttt A_1$} &  {$\PSL_2(q)$  $q\in\{$7, 8, 11, 13, 17, 19, 16, 23, 25, 27, 29, 31, 32, 37, 41, 43$\}$} & {$\frac{1}{d}q(q^2-1)$} where $d=\gcd(2,q-1)$ & $d\times k$ where $d=\gcd(2,q-1)$ and $q=p^k$ for $p$ a prime  \\
       \hline
       \multirow{2}{*}{$\texttt A_2$} & $\PSL_3(3)$ & 5,616   & $2$ \\
                                      & $\PSL_3(4)$ & 20,160  &  $2\times\sym(3)$ \\
                                      %& $\PSL_3(5)$ & 372,000 & $2$ \\
       \hline
       \multirow{1}{*}{$^2\texttt A_2$} & $\SU_3(3)$  & 6,048   & $2$ \\
                                        %& $\SU_3(4)$  & 62,400  & $4$ \\ 
                                        %& $\PSU_3(5)$ & 126,000 & $\sym(3)$ \\
       \hline
       \multirow{1}{*}{$\texttt C_2$} & $\PSp_4(3) \cong \SU_4(2)$ & 25,920 & $2$  \\
                                    %& $\Sp_4(4)$  & 979,200 & 4  \\
       \hline
       \multirow{1}{*}{Exceptional} & $^2\texttt B_2(8)$ & 29,120 & $3$ \\
       \hline
       \multirow{1}{*}{Sporadic} & $M_{11}$ & 7,920 & 1 \\
                                 %& $M_{12}$ & 95,040 & $2$ \\ 
                                 %& $M_{22}$ & 443,520 & $2$ \\
                                 %& $J_1$    & 175,560 & 1 \\
                                 %& $J_2$    & 604,800 & 2 \\
    \hline
    \end{tabular}
    \caption{Simple groups with order at most 51,840.}
    \label{tab:E6simple}
\end{table}

\begin{lemma}\label{lem:E6_hitlist}
    A smallest non-abelian quotient of $\cale_6$ has socle isomorphic to one of the following groups
    \begin{enumerate} %[label=\roman*)]
        \item $\alt(n)$ for $n\in\{5,6,7,8\}$;
        \item $\PSL_2(q)$ for $q\in\{7,8,11,13,17,19,16,23,25,27,29,31,37,41,43\}$;
        \item $\PSL_3(3)$, $\PSL(3,4)$;
        \item $\PSU_3(3)$;
        \item $\PSp_4(3)$;
        \item $^2\mathtt B_2(8)$;
        \item $M_{11}$;
        \item $\alt(5)^2$;
        \item $\PSL_2(7)^2$.
    \end{enumerate}
\end{lemma}
\begin{proof}
    Recall that $\cale_6^{ab}$ is cyclic. The subgroup $[\cale_6 , \cale_6]$ is perfect by \cref{thm:comm-perfect}, so Lemma \ref{lem:min_quotient_En} applies.

    The associated Coxeter group $W(\cale_6)$ is isomorphic to $\SO_5(3)$ which has order $51,840$. The non-cyclic simple groups with order less than $51,840$ are listed in Table \ref{tab:E6simple}.
\end{proof}

\subsection{Excluding alternating groups}

\begin{lemma}\label{lem:E6_aut_alt}
    Let $k\leq 2$ and $n\leq8$.  If $\alpha\colon\cale_6\to\Aut(\alt(n)^k)\wr2$ is a homomorphism, then $\alpha$ has cyclic image.
\end{lemma}
\begin{proof}
The group $\Out(\alt(n))$ is a $2$-group for each $n$; it is isomorphic to $2$ for $n\neq6$ and $2^2$ for $n=6$.  Since $\alt(n)$ for $n\leq 8$ does not contain an element of order $9$, neither does $\Aut(\alt(n))$ nor $\Aut(\alt(n)^2)\wr2\cong (\Aut(\alt(n))\wr2)\wr2$.  We conclude by \cref{lem:E6_elemOrders}. %Thus, the element $\epsilon_9$ in \cref{lem:E6_elemOrders} must be mapped by $\alpha$ to an element of order $3$ or a trivial element, but again by \cref{lem:E6_elemOrders}, this implies that the image of $\alpha$ is cyclic as required.
\end{proof}

\subsection{Excluding groups in the family \texorpdfstring{$\mathtt A_n$}{An}}
Here we give a uniform way of inductively excluding quotients isomorphic to subgroups of finite linear groups. 

\begin{thm}~\cite[Theorem~2.2.8]{DeFranceschiLiebeckOBrien2025}\label{thm:GLnq_centraliser}
    Let $x\in\GL_n(q)$ have minimal polynomial $f_1(t)^{e_1}\dots f_h(t)^{e_h}$ with $f_i\in\FF_q[t]$ distinct and irreducible of degree $d_i$.  Let the Jordan form of $x$ be
    \[\bigoplus_{i=1}^h \left( \bigoplus_{l_{i,1}}B_{\lambda_{i,1}}\oplus\dots\oplus \bigoplus_{l_{i,k_i}}B_{\lambda_{i,k_i}} \right)\]
    where $\lambda_{i,1}<\dots < \lambda_{i,k_i}=e_i$ for each $i$.  Then $C_G(x)\cong U\rtimes R$ where
    \[R\cong \prod_{i=1}^h\prod_{j=1}^{k_i}\GL_{l_{i,j}}(q^{d_i})\]
    and $U=q^m$ where
    \[m=\sum_{i=1}^hd_i\left(2\sum_{a<b}\lambda_{ia}\right).\]
\end{thm}

\begin{corollary}\label{lem:induct_on_Anq}
   Let $(G_{m-2},G_{m})\in\{(\cala_4,\cale_6),(\cala_5,\cale_7),(\cale_6,\cale_8),(\cala_n,\cala_{n+2})\ n\geq 4\}$.  If $G_{m}$ admits a non-cyclic homomorphism to $\GL_n(q)$, then $G_{m-2}$ admits a non-cyclic homomorphism to $\GL_{n-1}(q^d)$ for some $d\leq n$.
\end{corollary}
\begin{proof}
    Note that the centraliser of $s_{m}$ contains a group isomorphic to $G_{m-2}$. Suppose $\alpha\colon G_m\to \GL_n(q)$ has non-cyclic image.  Let the Jordan form of $\alpha(s_{m})$ be
    \[\bigoplus_{i=1}^h \left( \bigoplus_{l_{i,1}}B_{\lambda_{i,1}}\oplus\dots\oplus \bigoplus_{l_{i,k_i}}B_{\lambda_{i,k_i}} \right)\]
    where $\lambda_{i,1}<\dots < \lambda_{i,k_i}=e_i$ for each $i$.
    By \Cref{thm:GLnq_centraliser} we have that $\alpha(G_{m-2})$ is contained in a subgroup of shape $U\rtimes R$, where $R$ is a product of groups isomorphic to $\prod_{i=1}^h\prod_{j=1}^{k_i}\GL_{l_{i,j}}(q^{d_i})$.  Moreover, $\alpha(G_{m-2})$ cannot be cyclic by \cref{lem:E8_non_cyclic_E6}.  In particular, there must be some $\pi_{i,j}\colon U\rtimes R \twoheadrightarrow \GL_{l_{i,j}}(q^{d_i})$ such that $\pi(\alpha(G_{m-2}))$ is non-cyclic.  Now, if some $l_{i,j}=n$, then $h=1$, $i=j=1$, and $\lambda_{1,1}=e_1=d_1=n$.  From this we conclude that $\alpha(s_m)$ must be central.  Since $s_m$ satisfies an odd length Artin relation with $s_{m-1}$, this implies the image of $\alpha$ is cyclic by \cref{lem:artin-gen-collapsing}.  We conclude that $l_{i,j}<n$ as required.
\end{proof}

\begin{corollary}\label{lem:induct_on_PGLnq}
   Let $(G_m,G_{m+2})\in\{(\cala_4,\cala_6),(\cala_5,\cala_7),(\cale_6,\cale_8),(\cala_n,\cala_{n+2})\text{ for } n\geq 4\}$. If $G_{m+2}$ admits a non-cyclic homomorphism to $\PGL_n(q)$, then $G_{m}$ admits a non-cyclic homomorphism to $\PGL_{n-1}(q^d)$ for some $d\leq n$.
\end{corollary}
\begin{proof}
    This is an easy consequence of \Cref{lem:induct_on_Anq}.
\end{proof}

\begin{remark}
    One could in principle formulate similar results with the target being other finite classical groups using the results of \cite{DeFranceschiLiebeckOBrien2025}.
\end{remark}

% \katie{commented out a "to possibly add later" corollary noted by Sam, feel free to add it back if there's time}
%\begin{corollary}
%    Maybe a result about gln quotients of braid groups, tbd later
%\end{corollary}

\begin{lemma}\label{lem:E6_PSL2q_murder}
    If $\alpha\colon \cale_6 \to \PGL_2(q)$ is a homomorphism, then $\alpha$ has cyclic image.
\end{lemma}
\begin{proof}
By \Cref{lem:induct_on_PGLnq}, it suffices to check that every homomorphism $\cala_4\to\PGL_1(q^d)$ for every $d\geq 1$ has cyclic image. The result then follows from the fact that $\PGL_1(q^d)$ is solvable and \cref{lem:a4-solvable-cyclic}.
\end{proof}

\begin{remark}
Note that $\Aut(\PSL_2(q))\cong\PGL_2(q):\Aut(\FF_q)$.  It thus follows from the previous lemma that the only groups in the family $\mathtt A_1$ we need to rule out are $\Aut(\PSL_2(q))$ when $\FF_q$ admits a non-trivial automorphism. This occurs exactly when $q$ is one of $8,16,27$, or $32$.
\end{remark}

\begin{lemma}\label{lem:E6_aut_various_PSL2q_murder}
	If $\alpha\colon \cale_6 \to \Aut(\PSL_2(q))$ with $q=8,16,27,32$ is a homomorphism, then $\alpha$ has cyclic image.
\end{lemma}
\begin{proof}
Let $G=\Aut(\PSL_2(q))$ and $Q=\PSL_2(q)$. We can resolve three cases with element order considerations.
If $q=8$, then $G$ does not contain an element of order 12 \cite[pg.~6]{ATLAS}.
If $q=16$, then $G$ does not admit an element of order 9 \cite[pg.~12]{ATLAS}.
If $q=32$, then $G$ does not admit an element of order 9 or 12 \cite[pg.~29]{ATLAS}.
In any of the these cases, we conclude by \cref{lem:E6_elemOrders}.

The case $q=27$ is more involved.  In this case, $G=\PSL_2(27).6$. We claim that every homomorphism $\alpha\colon \cale_6\to Q.k$ for $k=2,3$ is cyclic.  Indeed, the image of $\cala_4$ must be a proper subgroup of $Q.k$ by \cref{lem:centraliser-proper-subgp}. By \cite[Page~18]{ATLAS}, the maximal subgroups of $Q.k$ are soluble groups of shapes $3^3:26$, $D_{56}$, $D_{52}$, $\sym(4)$, $3^3:13:3$, $28:6$, $26:6$, $\sym(4)\times 3$, and the non-soluble group $Q$.  

The image of $\cala_4$ cannot be one of the soluble groups because this would force its image to be cyclic by \cref{lem:a4-solvable-cyclic}, and the image of $\alpha$ would then be cyclic by \cref{lem:E8_non_cyclic_E6}. Thus, the image must be $Q$. Observe, however, that $\alpha(\cala_4)$ centralises $\alpha(s_1)$, and $Q$ does not centralise any non-trivial elements of $Q.k$. It follows that $Q$ cannot equal $\alpha(\cala_4)$. Since every maximal subgroup of $Q$ is solvable, we conclude that $\alpha$ has cyclic image.  

We now repeat the same argument for $G$. The maximal subgroups of $G$ are the soluble groups of shape $3^3:26:3$, $28:6$, $26:6$, $\sym(4)\times3$ as well as $Q.2$ and $Q.3$.  By the same reasoning as before, the image of $\cala_4$ cannot be $Q.2$ or $Q.3$ or $Q$. Since every other possible image is soluble, we conclude that $\alpha$ has cyclic image.
\end{proof}

\begin{lemma}\label{lem:E6_PSL33_murder}
    Let $Q=\PSL_3(3)$.  If $\alpha\colon\cale_8\to\Aut(Q)$ is a homomorphism, then $\alpha$ has cyclic image.
\end{lemma}
\begin{proof}
    We first note that $\alpha$ cannot have image contained in $Q$ because $Q$ does not contain an element of order $12$  \cite[Page~13]{ATLAS}.  We now consider the image of $\cala_4$ which centralises $s_1$.  Since $G=\Aut(\PSL_3(3))$ has trivial centre and $\alpha(s_1)$ is non-trivial, we must have that $\alpha(\cala_4)$ is a proper subgroup of $G$. Thus, $\alpha(\cala_4)$ is contained in a maximal subgroup. This maximal subgroup cannot be $Q$ because $Q$ does not centralise any non-trivial elements. The remaining maximal subgroups of $G$ are soluble, so we conclude by \cref{lem:a4-solvable-cyclic}.
\end{proof}

\subsection{Excluding other groups}

\begin{lemma}\label{lem:E6_order_murder}
    Let $Q =\PSL_3(4)$, $\SU_3(3)$, $^2\mathtt B_2(8)$, or $M_{11}$, and let $\alpha\colon\cale_6 \to \Aut(Q)$ be a homomorphism. The image $\alpha(\cale_6)$ is cyclic.
\end{lemma}
\begin{proof}
    Using the Atlas, one can verify that the groups $\Aut(\SU_3(3)) \cong G_2(2)$ and $\Aut({}^2\mathtt B_2(8)) \cong {}^2\mathtt B_2(8).3$ have no elements of order 9 \cite[~pg. 14, ~pg. 28]{ATLAS}, and $\Aut(M_{11}) \cong M_{11}$ has no elements of order 12 \cite[~pg. 18]{ATLAS}. The statement then follows from Lemma \ref{lem:E6_elemOrders}. 
\end{proof}

\begin{lemma}\label{lem:E6_PSp43_murder}
    If $\alpha\colon \cale_6 \to \PSp_4(3):2 $ is a homomorphism, then it has cyclic image.
\end{lemma}
\begin{proof}
    Recall that we must have $\PSp_4(3) \leq \alpha(\cale_6)$ by Lemma \ref{lem:min_quotient_En}. Thus it suffices to show that there is no non-cyclic homomorphism $\alpha\colon\cale_6 \to \PSp_4(3)$. The group $\PSp_4(3)$ does not have any elements of order 8 \cite[~pg. 27]{ATLAS}. Thus we conclude by \cref{lem:E6_elemOrders}.
\end{proof}

\begin{lemma}\label{lem:E6_PSL27_2_murder}
    If $\alpha\colon \cale_6 \to \Aut(\PSL_2(7)^2)\wr2 $ is a homomorphism, then is has cyclic image.
\end{lemma}
\begin{proof}
    By \cite[pg.3]{ATLAS}, the group $\Aut(\PSL_2(7))\cong \PSL_2(7).2$ has no elements of order $9$.  It follows that $\Aut(\PSL_2(7)^2)\cong (\PSL_2(7).2)\wr 2$ has no elements of order $9$.  Therefore, $\Aut(\PSL_2(7)^2)\wr2$ also has no elements of order 9.  Thus we conclude by \cref{lem:E6_elemOrders}.
\end{proof}

\subsection{Completing the proof}

\begin{thm}\label{thm:E6}
    The smallest non-abelian quotient of $\cale_6$ is isomorphic to $\PSU_4(2):2\cong W({\cale_6})$.  This quotient is unique up to an automorphism of the image.
\end{thm}

\begin{remark}
    Note that $\PSU_4(2):2\cong \PGO^-_6(2)\cong\PSp_4(3):2$.
\end{remark}

\begin{proof}[Proof of \Cref{thm:E6}]
    Let $P$ be the smallest non-abelian quotient of $\cale_6$. By \cref{lem:E6_hitlist}, it suffices to consider subgroups of $\Aut(Q)$ where $Q$ is one of the $9$ classes of groups listed in the statement of the lemma. \cref{lem:E6_aut_alt} excludes the possibility that $G < \Aut(\alt(n)^2)\wr2$, which encompasses classes $(1)$ and $(5)$ in the list. Lemmas \ref{lem:E6_PSL2q_murder} and \ref{lem:E6_aut_various_PSL2q_murder} shows that $G$ cannot be a subgroup of $\Aut(\PSL_2(q))$, eliminating class $(2)$. Lemmas \ref{lem:E6_order_murder}, \ref{lem:E6_PSL33_murder}, and \ref{lem:E6_PSL27_2_murder} eliminate classes $(2)$, $(3)$, $(4)$ $(6)$, $(7)$, and (9). Thus $P$ must be a subgroup of $\Aut(\PSU_4(2)) \cong \PSU_4(2): 2$. By \cref{lem:E6_PSp43_murder}, the image must be all of $\PSU_4(2): 2$.

    We now prove uniqueness. Let $G=\PSU_4(2):2$ and $Q=\PSU_4(2)$.  It suffices to show that in any homomorphism $\alpha\colon\cale_6\to G$ with non-cyclic image, the image of $s_1$ has order $2$. If $\alpha(s_1)$ has order 2, then $\alpha$ must factor through the Coxeter quotient map.
    
    Note that such a quotient induces a non-trivial homomorphism of $\cala_4$. Because $\cala_4$ centralises $s_1$, the image of $\cala_4$ must be a proper subgroup of $G$ by \cref{lem:centraliser-proper-subgp}. This subgroup cannot be $Q$ because $Q$ does not centralise any non-trivial elements of $G$, so it must be contained in a maximal proper subgroup of $G$ or $Q$. 
    
    By \cite[Page 26]{ATLAS}, the possible candidates are $2^4:\sym(5)$, $\sym(6)\times2$, $3^{1+2}_+:2\sym(4)$, $3^3:(\sym(4)\times2)$, and $(2\cdot \alt(4)^2.2).2$.  Lemmas \ref{lem:a4-solvable-cyclic} and \ref {lem:E8_non_cyclic_E6} allow us to eliminate any solvable groups from this list, leaving only $2^4:\sym(5)$ and $\sym(6)\times2$. 
    
    By \cref{kolay}, every homomorphism $\cale_4\to \sym(5)$ is the composite of the Coxeter quotient map and an an automorphism of $\sym(5)$.  In particular, if the maximal proper subgroup is $2^4:\sym(5)$, then the image of $s_1$ must have order $2$.  If $\cala_4$ instead had image contained in $\sym(6)\times 2$, then by \Cref{lem:A4toSym6}, either the image of $\cala_4$ is cyclic or the image of $\cala_4$ is $\sym(5)$ and the homomorpism factors through $W(\cala_4)$.  It follows that either the entire quotient is cyclic by \cref{lem:E8_non_cyclic_E6} or $\alpha(s_1)$ has order $2$ as required.
\end{proof}

\section{The Artin group \texorpdfstring{$\cale_7$}{E7}}\label{sec:e7}
In this section we will compute the smallest non-abelian quotient of $\cale_7$.  Throughout we will use \Cref{StandardFiniteGroupsNotation}.

\subsection{The groups}

\begin{lemma}\label{lem:E7_hitlist}
    A smallest non-abelian quotient of $\cale_7$ has socle isomorphic to one of the following groups:
    \begin{enumerate}
        \item $Q$ in \Cref{tab:E7simple};
        \item $\PSL_2(7)^2$, $\PSL_2(8)^2$, $\PSL_2(11)^2$, and $\PSL_2(13)^2$;
        \item $\alt(5)^2$, $\alt(5)^3$; and
        \item $\alt(6)^2$.
    \end{enumerate}
\end{lemma}
\begin{proof}
   By \Cref{lem:min_quotient_En}, a smallest non-abelian quotient has the shape $S^n.k$ for some simple group $S$ and where  $|S^n| \leq |W(\cale_7) / Z(W(\cale_7))|$ and $k$ divides $|\Out(S^n)|$. All simple groups $S$ with $|S| \leq |W(\cale_7) / Z(W(\cale_7))|$ are listed in Table \ref{tab:E7simple}.
\end{proof}

\begin{table}[h]
    \centering
    \begin{tabular}{|c|p{4.5cm}|r|p{2cm}|}
       \hline
       Family & Group & Order & $\Out(G)$ \\
       \hline
       \hline
       \multirow{5}{*}{Alternating} & $\alt(5)$  & 60 & $2$ \\
                                    & $\alt(6)$  & 360 & $2^2$ \\
                                    & $\alt(7)$  & 2,520 & $2$ \\
                                    & $\alt(8)\cong \PSL_4(2)$  & 20,160 & $2$ \\
                                    & $\alt(9)$  & 181,440 & $2$ \\
       \hline
      {$\texttt A_1$} &  {$\PSL_2(q)$  $q\in\{$7, 8, 11, 13, 17, 19, 16, 23, 25, 27, 29, 31, 32, 37, 41, 43, 47, 53, 59, 61, 67, 71, 73, 79, 64, 81, 83, 89, 97, 101, 103, 107, 109, 113, 121, 125, 127, 131, 137, 139$\}$} & {$\frac{1}{\gcd(2,q-1)}q(q^2-1)$} & $d\times k$ where $d=\gcd(2,q-1)$ and $q=p^k$ for $p$ a prime  \\
       \hline
       \multirow{3}{*}{$\texttt A_2$} & $\PSL_3(3)$ & 5,616   & $2$ \\
                                      & $\PSL_3(4)$ & 20,160  &  $2\times\sym(3)$ \\
                                      & $\PSL_3(5)$ & 372,000 & $2$ \\
       \hline
       \multirow{3}{*}{$^2\texttt A_2$} & $\SU_3(3)$  & 6,048   & $2$ \\
                                        & $\SU_3(4)$  & 62,400  & $4$ \\ 
                                        & $\PSU_3(5)$ & 126,000 & $\sym(3)$ \\
       \hline
       \multirow{2}*{$\texttt C_2$} & $\PSp_4(3) \cong \SU_4(2)$ & 25,920 & $2$  \\
                                    & $\Sp_4(4)$  & 979,200 & 4  \\
       \hline
       Exceptional & $^2\texttt B_2(8)$ & 29,120 & $3$ \\
       \hline
       \multirow{5}{*}{Sporadic} & $M_{11}$ & 7,920 & 1 \\
                                 & $M_{12}$ & 95,040 & $2$ \\ 
                                 & $M_{22}$ & 443,520 & $2$ \\
                                 & $J_1$    & 175,560 & 1 \\
                                 & $J_2$    & 604,800 & 2 \\
    \hline
    \end{tabular}
    \caption{Simple groups with order less than 1,451,520.}
    \label{tab:E7simple}
\end{table}

\subsection{Excluding the family \texorpdfstring{$\mathtt{A}_1$}{A1}}
To exclude groups in the family $\mathtt A_1$, we will work with the commutator subgroup $\cale_7'$ of $\cale_7$.  As $\cale_7'$ is a perfect group, it will suffice to rule out the simple groups $\PSL_2(q)$ rather than the groups of shape $\PSL_2(q)^n.k$.

\begin{lemma}\label{lem:E7_comm_PSL2q}
    Let $Q\cong \PSL_2(q)$ where $q=p^k$ with $p$ a prime.  If $\alpha\colon \cale_7'\to Q$ is a homomorphism, then $\alpha$ has trivial image.
\end{lemma}
\begin{proof}
    Let $X$ be the commutative totally symmetric set in \Cref{lem:E7_comm_TSS}.  
    We aim to understand $H=\Stab_Q(\alpha(X))$. Since $H$ is the stabiliser of a commutative totally symmetric set of size 3, we see that $H$ admits a surjection to $\sym(3)$.  We now proceed as in \cite[Lemma~4]{Caplinger-Kordek20}.  By \cite[\S2.1]{King2005} \footnote{This reference is secondary; the canonical reference is \cite{Dickson1958}, but the first proofs may be found in \cite{Wiman1900,Moore1903})}, subgroups of $\PSL_2(q)$ are as follows:
    \begin{enumerate}
        \item elementary abelian $p$-groups;
        \item cyclic groups of order $\ell$ where $\ell|(q\pm 1)/d$ where $d=\gcd(q-1,2)$;
        \item dihedral groups of order $2\ell$ with $\ell$ as in (2);
        \item $\alt(4)$ if $p>2$ or if $p=2$ and $k\equiv0\pmod{2}$;
        \item $\sym(4)$ if $q^2-1\equiv0\pmod{16}$;
        \item $\alt(5)$ if $p=5$ or $q^2-1\pmod{5}$;
        \item $p^m:t$ where $t$ divides $(p^m-1)$ and $(q-1)/d$, note that in this case either the subgroup is dihedral as in case (3), or it is contained in the stabiliser of a point $\mathbf{P}^1_q$; 
        \item $\PSL_2(p^{k_0})$ where $k_0|k$;
        \item $\PGL_2(p^{k_0})$ where $2k_0|k$.
    \end{enumerate}

    Our goal now is to rule out the above cases for $H$.  Since cases 1, 2, 4, 6, 8, and 9 do not admit $\sym(3)$ as an epimorphic image, we can already rule them out. By \cite[Theorem 3.3]{KordekLiPartin2021}, Dihedral groups do not contain a totally symmetric set of size $3$, so we may rule out case $3$.

    We now rule out case 5 where $H=\sym(4)$.  In this case the totally symmetric of size three is exactly the kernel of the projection $\sym(4)\twoheadrightarrow \sym(3)$ and so is contained in a Klein $4$ group $K$.  Write $c_i$ for the image of $s_1s_i^{-1}$  so $K=\{\id, c_2, c_5, c_7\}$.  We have that $c_ic_j=c_\ell$ for any distinct $i,j,\ell\in \{2,5,7\}$.  Now,
    \[c_4c_7 c_4c_5=c_4(c_7c_5)c_4 = c_4 c_2 c_4= c_2 c_4 c_2=(c_7c_5)c_4(c_7c_5)=c_7 c_4  c_7 c_5^2= c_4 c_7 c_4 c_5^2.\]
    Hence $c_5=1$, which would imply the homomorphism $\alpha$ is trivial.
    
    The remaining case is $H\cong p^m:t$ with $k\leq m$ and $t|q-1$, which was not treated by Caplinger--Kordek.  In this case, $H$ is a subgroup of $M$, the stabiliser of a point in $\mathbf P^1_q$.  In particular, $H$ is a subgroup of a maximal subgroup $M$ isomorphic to the image of the upper triangular matrices of $\SL_2(q)$ projected to $\PSL_2(q)$.  In particular, $H$ is isomorphic to a subgroup of the affine general linear group $\AGL_1(q)\cong \FF_q\rtimes(\FF_q^{\times,2})$ where $\FF_q^{\times,2}$ is $\FF^\times_q/\gcd(q-1,2)$.  Note that $H$ must project non-trivially to $\FF_q^{\times,2}$, since otherwise $H$ would be abelian and cannot surject $\sym(3)$. 
    
    %Up to conjugacy we may assume that $M$ is the image of the standard Borel subgroup consisting of upper triangular matrices in $\SL_2(q)$.
    
    The subgroup $H$ normalises $\langle X\rangle$.  Since $H$ projects non-trivially to $\FF_q^{\times,2}$, we see that any subgroup that $H$ normalises is contained in $M$.   We claim we may assume $p>2$.  Indeed, $H$ must surject $\sym(3)$ and the conjugation action of $H$ on $X$ must be non-trivial, in particular, a preimage of an involution in $\sym(3)$ in $H$ must project non-trivially to $\FF^{\times,2}$.% As $t|(p^m-1)$ we have $2|(p^m-1)$ if and only if $p$ is odd.  We claim we may assume $p=3$.  Indeed, otherwise $H\cong p^m:t$ has no homomorphism onto $\sym(3)$.  Hence, $p=3$.

    %Now, $3$ does not divide $3^m-1$.  So any epimorphism $H\to \sym(3)$ must be non-trivial on the $3^m$ subgroup.  
    
    Now, elements of $X$ commute so if $X$ projected non-trivially to $\FF_q^{\times,2}$ in $M$, then $X$ must be contained in a cyclic subgroup not intersecting $p^m$.  But then $\langle X\rangle$ is a cyclic group and its automorphism group does not admit $\sym(3)$ as a subquotient.  So $X$ must be contained in the maximal non-trivial elementary abelian normal subgroup of $M$ isomorphic to $\FF_q$.  This implies that the homomorphism of $H\to\sym(X)$ induced by conjugation on $X$ must factor through a cyclic group of order $q-1$, contradicting the assumption that $X$ is totally symmetric of size $3$.

    It follows from the above analysis that $\alpha$ cannot be injective on $X$.  
    By \Cref{lem:TSS_collapse}, %\Cref{lem:E7_TSS_murder}, 
    the image of $\alpha$ must be cyclic. Since $\cale_7'$ is perfect, the image of $\alpha$ is in fact trivial.
\end{proof}

\begin{corollary}\label{cor:aut-psl2-e7-murder}
    Let $Q\cong \PSL_2(q)^\ell.k$ for some $k,\ell\geq 1$.  If $\alpha\colon \cale_7\to Q$ is a homomorphism, then $\alpha$ has cyclic image.
\end{corollary}
\begin{proof}
    The homomorphism $\alpha$ induces a map $\cale_7'\to \PSL_2(q)$.  \Cref{lem:E7_comm_PSL2q} then implies that the image $\alpha(\cale_7')$ is trivial.  In particular, $\alpha$ factors through $\cale_7^\ab \cong \mathbb{Z}$, concluding the proof.
\end{proof}

\subsection{\texorpdfstring{$3^2$}{3\^2}s and 7s}

\begin{lemma}\label{lem:E7-non-A1-murder}
     Let be $Q$ be a group in \Cref{lem:E7_hitlist}(1) such that $Q\not\cong\alt(9)$ or a group in the family $\mathtt{A}_1$.  If $\alpha\colon \cale_7\to \Aut(Q)$ is a homomorphism, then $\alpha$ has cyclic image. 
\end{lemma}
\begin{proof}
    We first deal with groups of type (1) in \Cref{lem:E7_hitlist}.  Let $Q$ be one of $\alt(5)$, $\alt(6)$, $\PSL_3(3)$, $\PSL_3(5)$, $\SU(3,4)$, $\PSp_4(3)$, $\Sp_4(4)$, $M_{11}$, or $M_{12}$. Then $|Q|$ is not divisible by $7$. Since $|\Out(Q)|=2,3,2,2,4,2,4,1,2$ respectively, we see that $|\Aut(Q)|$ is not divisible by $7$. Consequently, $\Aut(Q)$ does not contain any elements of order $7$. \Cref{lem:E7_elemOrders} then implies that the image of $\alpha$ is cyclic.  %It follows that $\alpha(\cale_7)=1$ by .  Hence $\alpha$ is trivial.

  The group $^2\mathtt B_2(8)$ does not have order divisible by $3$, so it cannot have elements of order 9. In the remaining cases, we will examine the Sylow 3-subgroups, which can be found in \cite{ATLAS}. Since $\Out(^2\mathtt B_2(8))\cong 3$, $\Aut(^2\mathtt B_2(8))$ has its $3$-Sylow subgroups isomorphic to $3$. A 3-Sylow subgroup of $\Aut(J_1)$ is isomorphic to $3$.  A 3-Sylow subgroup of $\Aut(Q)$ for $Q$ one of $\alt(7)$ or $\alt(8)$ is isomorphic to $3^2$. A 3-Sylow subgroup of $\Aut(Q)$ for $Q$ one of $\PSL_3(4)$, $\SU_3(3)$, $M_{22}$, $J_2$ is isomorphic to $\He_3(3)$.  A 3-Sylow subgroup of $\PSU_3(5)$ is isomorphic to $3^2$.  The automorphism group of $\PSU_3(5)$ contains $\PU_3(5)$ as an index $2$ subgroup and a 3-Sylow subgroup of $\PSU_3(5)$ is isomorphic to $\He_3(3)$.  Hence, a 3-Sylow subgroup of $\Aut(\PSU_3(5))$ is isomorphic to $\He_3(3)$.  
    
    Since neither $3$, $3^2$, nor $\He_3(3)$ has elements of order $9$, \Cref{lem:E7_elemOrders} implies that $\alpha$ has cyclic image. This proves the result for all groups in the hypothesis of type (1).
\end{proof}

\begin{lemma}\label{lem:E7-non-A1-murder-prods}
    Let $S^n\in \{\alt(5)^2,\alt(5)^3,\alt(6)^2\}$ and let $k$ be an order of an element of $\Out(S^n)$.  Let $Q$ be a group of shape $S^n.k$.  If $\alpha\colon\cale_7\to Q$ is a homomorphism, then $\alpha$ has cyclic image.
\end{lemma}
\begin{proof}
        We have $\Out(\alt(5)^2)\cong 2\wr2$, $\Out(\alt(5)^2)\cong 2\wr\sym(3)$, and $\Out(\alt(6)^2)\cong 2^2\wr 2$.  By the universal embedding theorem \cite{KrasnerKaloujnine1951}, any group of shape $S^n.k$ as in the statement of the lemma embeds into one of the groups  $\alt(5)^2\wr4$, $\alt(5)^3\wr12$ or $\alt(6)^2\wr8$.  None of these groups has an element of order $7$, so we conclude by \cref{lem:E7_elemOrders}.   
\end{proof}

\subsection{The alternating group of degree 9}

% The centraliser of $\alpha(s_1)$ contains $\alpha(W_{\cala_5})$ which by \Cref{kolay} must have order at least $6!=720$.  The only conjugacy class of $\alt(9)$ not containing elements of order $2$ with centraliser of size bigger than $720$ is the conjugacy class of $3$-cycles of shape $(1\ 2 \ 3)$.  This has centraliser isomorphic to $3\times \alt(6)$.  But $\alt(6)$ is not a quotient of $\cala_5$.  Thus $\alpha(\cala_5)=1$ and hence $\alpha(\cale_6)=1$.

\begin{lemma}\label{lem:E7_alt9-murder}
    If $\alpha\colon \cale_7 \to \sym(9) =\Aut(\alt(9))$ is a homomorphism, then $\alpha$ has cyclic image.
\end{lemma}
\begin{proof}
    First, observe that since $\cala_5$ is contained in the centraliser of $s_1$, the image of $s_1$ must have a centraliser with size at least $720$. Otherwise, the image of $\cala_5$ is cyclic by \cref{kolay}, and the image of $\alpha$ is cyclic by \cref{lem:E8_non_cyclic_E6}. Furthermore, the image of $s_1$ must have order larger than 2. If its image has order 2, then the homomorphism would factor through $W(\cale_7)$, and this is impossible because no non-abelian subgroup of $\sym(9)$ is a quotient of $W(\cale_7)$. 
    
    The only conjugacy class of $\alt(9)$ of elements of order greater than $2$ with centraliser of size bigger than $720$ is the conjugacy class of $3$-cycles of shape $(1\ 2 \ 3)$. This has centraliser isomorphic to $3\times \sym(6)$.  Suppose for contradiction that $\alpha(s_1)= (a\ b\ c)$ and $\alpha$ has non-cyclic image. Then $\alpha(\cala_{\{2,4,5,6,7\}})$ is exactly equal to $\sym([9]\setminus\{a,b,c\})$. By \Cref{kolay}, this quotient is unique up to an automorphism of $\sym(6)$ and is induced by imposing the relation $s_2^2$.  Since $s_1$ is conjugate to $s_2$, the element $s_1$ must satisfy $s_1^3=1$ and $s_1^2=1$, implying it is trivial.  Hence, $\alpha$ must have cyclic image.
\end{proof}

\subsection{Completing the proof}

\begin{thm}\label{thm:E7}
    The smallest non-cyclic quotient of $\cale_7$ is isomorphic to $\Sp_6(2)\cong W({\cale_7})/Z(W(\cale_7))$.  Any such quotient factors though the quotient to $W(\cale_7)$.  In particular, the quotient is unique up to an automorphism of the image.
\end{thm}

\begin{remark}
	Note that $\Sp_6(2)\cong\PSp_6(2)\cong\GO_7(2)\cong\PSO_7(2)$.
\end{remark}

\begin{proof}
    Let $Q$ be a finite non-cyclic quotient of $\cale_7$, and suppose $|Q| \leq |\Sp_6(2)| = 1451520$. By Lemma \ref{lem:E7_hitlist}, $Q$ has socle $S$ where $S$ is one of the groups in Table \ref{tab:E7simple}, or $\PSL_2(7)^2$, $\PSL_2(8)^2$, $\PSL_2(11)^2$, $\PSL_2(13)^2$, $\alt(5)^2$, $\alt(5))^3$, or $\alt(6)^2$.

    By Corollary \ref{cor:aut-psl2-e7-murder}, the image of $\cale_7$ under any homomorphism to $\mathrm{Aut}(\mathrm{PSL}_2(q)^k)$ for $k \geq 1$ is cyclic. Lemmas \ref{lem:E7-non-A1-murder} and \ref{lem:E7-non-A1-murder-prods} rule out all the remaining possibilities except $G \cong \mathrm{Alt}(9)$, and the case $G \cong \mathrm{Alt}(9)$ is ruled out by Lemma \ref{lem:E7_alt9-murder}. Noting that $W({\cale_7})$ is a quotient of $\cale_7$ and that $\Sp_6(2)$ is non-abelian completes the proof that $\Sp_6(2)$ is the smallest possible non-cyclic quotient.

    We now prove that the homomorphism is unique up to automorphisms of the image.  As in the case of $\cale_6$, we show that any homomorphism $\alpha\colon\cale_7\to\Sp_6(2)$ sending a generator to an element of order other than $2$ must be cyclic.  We first claim that $\alpha(s_1)$ must be contained in the conjugacy class $3_A$ in the ATLAS \cite[Page~47]{ATLAS}.  Indeed, this follows from the Orbit-Stabiliser Theorem: the centraliser of $\alpha(s_1)$ contains $\alpha(\cala_5)$. This must be noncylic by \cref{lem:E8_non_cyclic_E6}, so it must have order at least $720$ by \cref{kolay}, and $3_A$ is the only class with a centraliser of order at least $720$.
    
    By the information on \cite[Page~46]{ATLAS}, the normaliser of the subgroup generated by an element in the class $3_A$ is isomorphic to $\sym(3)\times\sym(6)$.  It follows that $\alpha(\cala_5)$ must surject onto $\sym(6)$.  But by \cref{kolay}, any such homomorphism factors through $W(\cala_5)$.  In particular, if $s_1$ has order other than $2$, then $\alpha$ has cyclic image.
\end{proof}

\section{The Artin group \texorpdfstring{$\cale_8$}{E8}}\label{sec:e8}
In this section we will compute the smallest non-abelian quotient of $\cale_8$.  Throughout we will use \Cref{StandardFiniteGroupsNotation}.

\subsection{The groups}
\Cref{tab:E8simple} recalls the finite simple groups of order at least $1,451,520$ but below $348,364,800$ which are not in the family $\mathtt A_1$.

\begin{table}[h]
    \centering
    \begin{tabular}{|c|l|r|r|c|c|}
       \hline
       Family & Group & Order & $\Out(G)$ & 7s & 9s \\
       \hline
       \hline
       \multirow{3}{*}{Alternating} & $\alt(10)$  & 1,814,400 & $2$  & y & y\\
                                    & $\alt(11)$  & 19,958,400 & $2$ & y & y \\
                                    & $\alt(12)$  & 239,500,800 & 2 & y & y\\
       \hline
       \multirow{5}{*}{$\mathtt A_2$} & $\PSL_3(7)$ & 1,876,896 & $\sym(3)$ & y & n\\
                                      &  $\PSL_3(8)$ & 16,482,816 & 6 & y & y\\
                                      &  $\PSL_3(9)$ & 42,456,960 & $2^2$& y & y\\
                                      &  $\PSL_3(11)$ & 212,427,600 & 2 & y & n\\
                                      & $\PSL_3(13)$ & 270,178,272 & $\sym(3)$ & y & n\\
        \hline
        \multirow{4}{*}{$^2\mathtt A_2$} & $\PSU_3(8)$  & 5,515,776 &$3\times\sym(3)$ & y & y\\
                                         &  $\SU_3(7)$  & 5,663,616 & 2 & y & n\\
                                        &   $\SU_3(9)$  & 42,573,600 & 4 & n & y\\
                                        &   $\PSU_3(11)$ & 70,915,680 & $\sym(3)$ & n & n \\
        \hline
        \multirow{2}{*}{$\mathtt C_2$} &  $\PSp_4(5)$ & 4,680,000 & 2 & n & n\\
                                      &  $\PSp_4(7)\cong \POmega_{5}(7)$ & 138,297,600 & 2 & y & n\\
        \hline
        $\mathtt C_3$ & $\PSp_6(2)$ & 1451520 & 1 & y & y \\
        \hline
        $\mathtt A_3$           & $\PSL_4(3)$ & 6,065,280 & $2^2$ & n & n \\
        \hline
        $^2\mathtt A_3$           & $\PSU_4(3)\cong O_6^-(3)$ & 3,265,920 & $D_8$ & y & y\\
        \hline
        $\mathtt A_4$           & $\PSL_5(2)$ & 9,999,360 & 2 & y & n\\
        \hline
        $^2\mathtt A_4$           & $\PSU_5(2)$ & 13,685,760 & 2 & n & y \\
        \hline
        $\mathtt D_4$           & $\POmega_4^+(2)$ & 174,182,400 & 6 & n & n \\
        \hline
        $^2\mathtt D_4$           & $\POmega_4^-(2)$ & 197,406,720 & 2 & y & y \\
        \hline	
        \multirow{5}{*}{Exceptional} & ${}^2\mathtt B_2(32)$  & 32,537,600 & 5 & n & n\\
                                     & ${}^2\mathtt F_4'(2)$   & 17,971,200 & 2 & n & n\\
                                     & $\mathtt G_2(3)$       & 4,245,696 & 2 & y & y\\
                                     & $^3\mathtt D_4(2)$     & 211,341,312 & 3 & y & y\\
                                     & $\mathtt G_2(4)$       & 251,596,800 & 2 & n & y\\
        \hline
        \multirow{4}{*}{Sporadic}  & $M_{23}$ & 10,200,960 & 1 & y & n\\
                                   & $HS$     & 44,352,000 & 2 & y & n\\
                                   & $J_3$    & 50,232,960 & 2 & n & y\\
                                   & $M_{24}$ & 244,823,040 & 2 & y & y\\
        \hline
    \end{tabular}
    \caption{Simple groups of order at least 1,451,520 and less than 348,364,800 not in the family $\mathtt A_1$.}
    \label{tab:E8simple}
\end{table}

% \begin{lemma}\label{lem:E8_hitlist}
%     It suffices to exclude subgroups of the following groups:
%     \begin{enumerate}
%         \item $\Aut(Q)$ for any group $Q$ appearing in \Cref{tab:E7simple};
%         \item $\Aut(Q)$ for any group $Q$ appearing in \Cref{tab:E8simple};
%         \item $\Aut(\alt(5)^5)\wr k$ for $k$ the order of an element of $\sym(5)\wr\sym(5)$;
%         \item $\Aut(\alt(6)^3)\wr k$ for $k$ the order of an element of $\sym(6).2\wr\sym(3)$;
%         \item $\Aut(Q^2)\wr |\Out(Q)|$ for $Q$ one of $M_11$, $\PSL_3(3)^2$, $\PSU_3(3)^2$;
%         \item $\PSL_2(q)^n.k$ for any $n,k\leq 12$ and any $q$.
%     \end{enumerate}
% \end{lemma}
\begin{lemma}\label{lem:E8_hitlist}
    A smallest non-abelian quotient of $\cale_8$ has socle isomorphic to one of the following groups
    \begin{enumerate}
        \item any group $S$ appearing in \Cref{tab:E7simple};
        \item any group $S$ appearing in \Cref{tab:E8simple};
        \item $\alt(5)^n$ for $n\leq 5$;
        \item $\alt(6)^2$, $\alt(6)^3$;
        \item $S^2$ for $S$ one of $M_11$, $\PSL_3(3)$, $\PSU_3(3)$;
        \item $\PSL_2(q)^n$ for any $n\leq 12$ and any $q$.
    \end{enumerate}
\end{lemma}
\begin{proof}
    We stress that these conditions are sufficient but in the case where the socle is $\PSL_2(q)^n$, they are stronger than necessary (since we are allowing arbitrarily large $q$).  By \Cref{lem:min_quotient_En}, it suffices to consider groups of the shape $S^n.k$ for some simple group $S$ and where  $|S^n| \leq |W(\cale_8) / Z(W(\cale_8))|$ and $k$ is the order of an element of $\Out(S^n)$. All simple groups $S$ with $|S| \leq |W(\cale_8) / Z(W(\cale_8))|$ are listed in \Cref{tab:E7simple} and \Cref{tab:E8simple}. The tables also lists their orders.
\end{proof}

\begin{remark}\label{rm:e8-coxeter-quotient}
    The Coxeter group $W(\cale_8)$ modulo its centre has a simple subgroup of index $2$. As such, to show that a homomorphism $\alpha: \cale_8 \rightarrow Q$ to a particular group $Q$ is cyclic, it will often suffice to show that it must factor through the Coxeter group.
\end{remark}

\subsection{\texorpdfstring{$3^2$}{3\^2}s and 7s}

\begin{lemma}\label{lem:E8_7s_9s}
    A non-cyclic quotient of $\cale_8$ contains elements of order $7$ and $9$.
\end{lemma}
\begin{proof}
    Let $\alpha\colon\cale_8\to Q$ be a finite quotient with $Q$ non-cyclic.  The homomorphism $\alpha$ induces a homomorphism $\beta\colon \cale_7\to Q$.  The image $P=\im(\beta)$ is a finite quotient of $\cale_7$, and it is non-cyclic by \cref{lem:E8_non_cyclic_E6}. By \Cref{lem:E7_elemOrders} we see that $P$ contains elements of order 7 and 9.  Whence, the lemma.
\end{proof}

\begin{lemma}\label{lem:E8-7s-9s-murder}
    Let $Q$ be one of $\PSL_3(7)$, $\PSL_3(11)$, $\PSL_3(13)$, $\SU_3(7)$, $\SU_3(9)$, $\PSU_3(11)$, $\PSp_4(5)$, $\PSp_4(7)$, $\PSL_4(3)$, $\GL_5(2)$, $\PSU_5(2)$, $^2\mathtt B_2(32)$, $^2\mathtt F_4(2)$, $\mathtt G_2(4)$, $M_{23}$, $HS$, or $J_3$.  If $\alpha\colon\cale_8\to\Aut(Q)$ is a homomorphism, then $\alpha$ has cyclic image.
\end{lemma}
\begin{proof}
    Table \ref{tab:E8simple} lists whether each group contains elements of order 7 and 9. With one exception, this information for $Q$ and $\Aut(Q)$ can be found in \cite{ATLAS}. If the group does not contain an element of order 7 and an element of order 9, we conclude by \cref{lem:E8_7s_9s}. %If the order of a group is not divisible by both $7$ and $9$, then it cannot have elements of orders $7$ and $9$ by Lagrange's theorem. For elements of order $7$, the converse is also true: if $7$ divides the order of the group, then there is an element of order $7$ by the first Sylow theorem. For elements of order $9$, the converse is not necessarily true, but the existence of elements of order $9$ in the remaining cases is classified by the ATLAS.  
    The element order information for $Q\cong \PSL_3(13)$ is not in the ATLAS, so we verify there are no elements of order 9 in $\Aut(Q)$ via MAGMA; see \Cref{code:E8_7s_9s}.
\end{proof}

\subsection{Alternating groups}

\begin{lemma}\label{lem:E8-sym12-murder}
    If $\alpha\colon\cale_8\to \sym(12)$ is a homomorphism, then $\alpha$ has cyclic image.
\end{lemma}
\begin{proof}
        Let $\alpha\colon\cale_8\to \sym(12)$ be a homomorphism  We aim to constrain the possible images of the generator $s_8$.  The element $\alpha(s_8)$ cannot have order $2$ because this would imply that $\alpha$ factors through the quotient to $W(\cale_8)$, and $\sym(12)$ is not a quotient of $W(\cale_8)$; see \cref{rm:e8-coxeter-quotient}.

        The group $\sym(12)$ has $77$ conjugacy classes. Of these, $70$ contain non-trivial elements with order distinct from $2$.   
        The third smallest conjugacy class of these $70$ is the class of $5$-cycles $(a\ b\ c\ d\ e)$ consisting of $\binom{12}{5}\times 4!=19008$ elements. By the Orbit-Stabiliser Theorem, the third largest centraliser of a conjugacy class of non-trivial elements of $\sym(12)$ with order distinct from $2$ is $25,200$.  
        This is smaller than the smallest non-abelian quotient of $\cale_6$.  It follows that if $\alpha$ had non-cyclic image, then $\alpha(s_8)$ must be contained in the one of the two smallest conjugacy classes of non-trivial elements of order not equal to $2$.  
        
        The second smallest conjugacy class of the $70$ classes is the class of $4$-cycles $(a\ b\ c\ d)$ consisting of ${{12}\choose{4}}\times 3!=2970$ elements.  The centraliser of such an element has order $161,280$ and is isomorphic to $4\times \sym(8)$.  In particular, we obtain a homomorphism $\cale_6\to \sym(8)$.  Since $|\sym(8)|=8!=40,320$ is smaller than $51,840$, $\alpha(\cale_6)$ must have cyclic image by \cref{thm:E6}. \cref{lem:E8_non_cyclic_E6} then implies that $\alpha(\cale_8)$ itself is cyclic. %$\cale_6$ must have cyclic image in $\sym(8)$ and so $\alpha(\cale_8)$ must have cyclic image in $\sym(11)$.
        
        The smallest conjugacy class of the 70 classes is the class of $3$-cycles $(a\ b\ c)$ consisting of ${{12}\choose{4}}\times 2!=440$ elements.  The centraliser of such an element has order 1,088,640 and is isomorphic to $3\times\sym(9)$.  In particular, we obtain a homomorphism $\beta\colon\cale_6\to \sym(9)$ such that every Artin generator is mapped to a cycle of shape $3^1$. 
        
        The centraliser of $\beta(s_1)$ must contain $\beta(\cala_4)$ where $\cala_4$ is generated by $\{s_2,s_4,s_5,s_6\}$.  The centraliser of a cycle of shape $3^1$ on $\sym(9)$ is isomorphic to $3\times\sym(6)$, so we obtain an induced homomorphism $\theta\colon \cala_4\to \sym(6)$.  By \Cref{lem:A4toSym6}, this either factors through Coxeter group $W(\cala_4)$, implying that $\beta(s_2)$ has order $2$ and contradicting the fact it has order $3$, or the homomorphism $\beta$ has cyclic image.  We conclude that $\alpha$ has cyclic image as required.
\end{proof}

\subsection{Classical groups}

We begin by observing the following special case of \cref{lem:centraliser-proper-subgp}.
\begin{lemma}\label{lem:E8_proper_E6_subgroup}
    Let $Q$ be a non-trivial group with trivial centre, and let $\alpha\colon \cale_8 \to Q$ be a homomorphism. The image $\alpha(\cale_6)$ of the standard parabolic subgroup of type $\cale_6$ must be a proper subgroup of $Q$.
\end{lemma}
\begin{proof}
    The standard parabolic subgroup of type $\cale_6$ is contained in $\Gamma - \text{St}(s_8)$. The result then follows from \cref{lem:centraliser-proper-subgp}.
\end{proof}
%\begin{proof}
%    Recall that the subgroup $\cale_6 < \cale_8$ centralizes the Artin generator $s_8$. Since all Artin generators are conjugate, if $\alpha(s_8)$ is trivial, then the entire subgroup $\alpha(\cale_8)$ is trivial. If it is non-trivial, then we cannot have $\alpha(\cale_6) \cong Q$ because this would require $\alpha(s_8) \in Z(\alpha(\cale_6))$, and $Q$ has trivial centre. 
%\end{proof}
% \katie{I left this in to avoid having to reiterate that $\cale_6$ commutes with $s_8$ every time it arises, but we can change that. I also don't think we need the trivial centre assumption.}

\begin{lemma}\label{lem:E8_PSL38-murder}
    If $\alpha\colon\cale_8\to \Aut(\PSL_3(8))$ is a homomorphism, then $\alpha$ has cyclic image.
\end{lemma}
\begin{proof}
    Let $Q=\PSL_3(8)$ and recall that $\Out(Q)=6$ from \Cref{tab:E8simple}.  In particular, $G=\Aut(\PSL_3(8))$ is a group of shape $\PSL_3(8).6$.  Let $\alpha\colon \cale_8\to G$ be a homomorphism. The image of $\cale_6$ under $\alpha$ must be a proper subgroup of $G$ by \cref{lem:E8_proper_E6_subgroup}.

   By \Cref{lem:induct_on_PGLnq} if the image of $\alpha$ is non-cyclic and contained in $Q$, then $\cale_6$ admits a non-cyclic homomorphism to $\PGL_2(8^n)$ for some $n\in\{1,2,3\}$.  We claim we may take $n=1$.  One could prove this by examining the minimal polynomial of an element of $\GL_3(q)$; instead, we observe from \cite[p.~73]{ATLAS} that $Q$ has no subgroups isomorphic to $\PSL_2(64)$ or $\PSL_2(512)$.  By \cref{lem:E6_PSL2q_murder}, every homomorphism $\cale_6\to\PSL_2(8)$ has cyclic image.  Thus if $\alpha$ is non-cyclic, $\alpha(\cale_6)$ must not be contained in $Q$.

    Consulting \cite[p.~73]{ATLAS}, we see that a proper maximal subgroup of $G$ is isomorphic to one of the following
    \begin{enumerate}
        \item $2^{3+6}:7^2:6$, $73:18$, $7^2:(6\times \sym(3))$,
        \item $(D_{14}\times\PSL_2(8)):3$, 
        \item $\PSL_2(7):2\times 3$,
        \item $Q.2$ and $Q.3$.
    \end{enumerate}
    The maximal subgroups of $Q.2$ and $Q.3$ are isomorphic to subgroups of the maximal subgroups of $Q.6$ with extra addition of the subgroup $Q$.  The subgroups $Q$, $Q.2$, and $Q.3$ do not centralise any non-trivial elements, so the image of $\cale_6$ cannot be equal to one of these.  Thus, $\alpha(\cale_6)$ must be contained in a group of type (1), (2), or (3).  The groups of type (1) are solvable, so any homomorphism $\cale_6\to M$ with $M$ of type (1) has cyclic image by \cref{lem:a4-solvable-cyclic}.  A group of type (2) has order at most 42336, and a group of type (3) has order 1008.  Both of these are smaller than the smallest non-cyclic quotient of $\cale_6$, so we conclude by \cref{thm:E6} and \cref{lem:E8_non_cyclic_E6}. 
\end{proof}

\begin{lemma}\label{lem:E8_PSL39_murder}
    If $\alpha\colon\cale_8\to \Aut(\PSL_3(9))$ is a homomorphism, then $\alpha$ has cyclic image.
\end{lemma}
\begin{proof}
    Let $Q=\PSL_3(9)$ and recall that $\Out(Q)=2^2$.  Let $G=\Aut(\PSL_3(9))$, and note that because $\cale_8^{ab}$ is cyclic, any homomorphism $\cale_8\to G$ must be contained in an index $2$ subgroup of the shape $Q.2$.  By \cite[p.~78]{ATLAS}, any non-solvable maximal subgroup of an extension $Q.2$ admits an epimorphism with solvable kernel onto one of the following groups:
    \begin{enumerate}
        \item $Q$;
        \item $\Aut(\PSU_3(3))$;
        \item $\Aut(\PSL_3(3))$; or
        \item $\Aut(\alt(6))$.  
    \end{enumerate}
    Note that $\alt(6)\cong\PSL_2(9)$.  Lemmas \ref{lem:induct_on_PGLnq} and \ref{lem:E6_aut_alt} imply that that $\alpha$ cannot have a non-cyclic image contained in a group of type (1). If $\alpha$ has non-cyclic image, then $\alpha(\cale_6)$, which is a proper subgroup of $\alpha(\cale_8)$ by \Cref{lem:E8_proper_E6_subgroup}, must be contained in a group of type (1)-(4).  Now, $\alpha(\cale_6)$ cannot be equal to a group of type (1) since it does not centralise any non-trivial elements of $G$.  If $\alpha(\cale_6)$ was a proper subgroup of a group of type (1), then $\alpha(\cale_6)$ would admit a quotient into a subgroup of a group of type (2), (3), or (4).  Thus, we need to rule out the existence of homomorphisms $\cale_6\to M$ with non-cyclic image where $M$ is one types (2), (3), or (4).  This is handled by \Cref{lem:E6_order_murder}, \Cref{lem:E6_PSL33_murder}, and \Cref{lem:E6_aut_alt} respectively.
\end{proof}

\begin{lemma}\label{lem:E8_PSU38_murder}
    If $\alpha\colon\cale_8\to \Aut(\PSU_3(8))$ is a homomorphism, then $\alpha$ has cyclic image.
\end{lemma}
\begin{proof}
    Let $Q=\PSU_3(8)$, and recall that $\Out(Q)=3\times\sym(3)$.  Note that any homomorphism $\cale_8\to \Aut(\PSU_3(8))$ must have image contained in a subgroup of shape Q.k where $k=2,3,6$.
    
   We first treat the case that the image of $\alpha$ is contained in $Q$. Observe that $\alpha(\cale_6)$ must be a proper subgroup of $Q$ by \cref{lem:E8_proper_E6_subgroup}. %The element $\alpha(s_8)$ must be non-trivial in any non-trivial quotient because the Artin generators are all conjugate, and $Q$ has trivial centre. 
   Thus $\alpha(\cale_6)$ must be contained in some proper maximal subgroup. The only non-soluble maximal proper subgroup of $Q$ is $3\times \PSL_2(8)$, and the order of $3\times \PSL_2(8)$ is $1,512$. This is smaller than the smallest non-cyclic quotient of $\cale_6$. \cref{thm:E6} and \cref{lem:E8_non_cyclic_E6} then imply $\alpha(\cale_8)$ is cyclic.
   
   Now, suppose $\alpha$ has image contained in a group of shape $Q.2$ or $Q.3$.  A non-soluble maximal subgroup of such a group admits a homomorphism to either $Q$ or to $\Aut(\PSL_2(8))$.  The subgroup $\alpha(\cale_6)$ cannot be $Q$ by the same argument as in the previous paragraph, so $\alpha(\cale_6)$ must admit a homomorphism to $\Aut(\PSL_2(8))$. The order of $\Aut(\PSL_2(8))$ is $1512$, so any homomorphism from $\cale_6$ to $\Aut(\PSL_2(8))$ is cyclic by \cref{thm:E6}. \cref{lem:E8_non_cyclic_E6} then implies that $\alpha(\cale_8)$ is cyclic.
   
   Finally, suppose $\alpha$ has image contained in a group of shape $Q.6$. The only possible images for $\alpha(\cale_6)$ which have not already been addressed are a group of shape $Q.2$ or $Q.3$. Subgroups of this form do not centralize any non-trivial elements, concluding the proof.
\end{proof}

\begin{lemma}\label{lem:E8_PSU43_murder}
    If $\alpha\colon\cale_8\to \Aut(\PSU_4(3))$ is a homomorphism, then $\alpha$ has cyclic image.
\end{lemma}
\begin{proof}
    Let $Q=\PSU_4(3)$, and recall $\Out(Q)=Dh_4$ which has order $8$.  If $\alpha$ factors through the canonical Coxeter quotient, it must be cyclic; see \cref{rm:e8-coxeter-quotient}. Otherwise, the image of an Artin generator has order larger than 2. By \cite[Page~54]{ATLAS}, the largest centraliser of a class of elements of order greater than $2$ in $Q$ has size 5832. This is smaller than the smallest non-cyclic quotient of $\cale_6$, so \Cref{lem:E8_non_cyclic_E6} implies that $\alpha(\cale_8)$ is cyclic if it lies in $Q$. 
    
    We now show that any homomorphism $\alpha\colon \cale_8\to Q.2$ has cyclic image.   According to the ATLAS \cite[Page~52]{ATLAS}, there are exactly three non-isomorphic almost simple groups of shape Q.2. We study the possible images of $\cale_6$ in a group of shape $Q.2$. By \cref{lem:E8_proper_E6_subgroup}, $\alpha(\cale_6)$ must be contained in a proper maximal subgroup of $Q.2$.  Furthermore, this subgroup cannot be $Q$ because $Q$ does not centralize any non-trivial element. %CANNOT BE $Q$ NOT CENTRAL.
    
    The remaining possibilities are a soluble group or one of $3^4:\alt(6)$, $\PSU_4(2)$, $\PSL_3(4)$, $\PSU_3(3)$, $2^4:\alt(6)$, $\alt(7)$, $\alt(6).2$, $\PSU_4(2):2$, $\PSL_3(4):2$, $\PSU_3(3)\times 2$, $\alt(6).2^2$, $3^4:\sym(6)$, $\PSU_4(2)\times2$, $2^4:\sym(6)$, $2^5:\alt(6)$, $\sym(7)$, $3^4:(\alt(6).2)$, $\PSU_3(3)$, and $(2^4\colon\alt(6)).2$). Note that by \cite[Pg. 300]{BrauerATLAS}, the group $2^4:\sym(6)$ is incorrectly listed in \cite{ATLAS} and should be replaced by a group of shape $(2^4\colon\alt(6)).2$ as we have done.  It follows that $\cale_6$ admits a homomorphism to $\Aut(L)$ for $L$ one of $\alt(6)$, $\alt(7)$, $\PSL_3(4)$, $\PSU_3(3)$, or $\PSU_4(2)$.  The only case where this is not immediate is $M=(2^4\colon\alt(6)).2$, where one observes first that the abelian factor of $\soc(M)$ is exactly $2^4$ and so one may pass to the quotient $\alt(6).2$.  The first two possibilities for $\Aut(L)$ are ruled out by Lemma \ref{lem:E6_aut_alt}, and the second two are ruled out by \Cref{lem:E6_order_murder}.
    
    The group $\PSU_4(2):2$ is a quotient of $\cale_6$, but by \cref{thm:E6}, every quotient map $\cale_6 \to \PSU_4(2):2$ differs from the standard one by an automorphism of the image.  Since all Artin generators are conjugate and thus have the same order in any quotient, this would imply the image of $\cale_8$ in $Q.2$ factors through $W(\cale_8)$.  This is only possible if $\alpha(\cale_8)$ is cyclic.
    
    A nearly identical argument handles the case of $Q.4$.  Since any other possible extension is by a soluble but not cyclic group, we conclude that every homomorphism $\alpha\colon \cale_8\to\Aut(\PSU_4(3))$ has cyclic image.
\end{proof}

\begin{lemma}\label{lem:E8_PSp62_murder}
    If $\alpha\colon\cale_8\to \Aut(\PSp_6(2))$ is a homomorphism, then $\alpha$ has cyclic image.
\end{lemma}
\begin{proof}
    Let $Q=\PSp_6(2)$ and recall $\Out(Q)=1$. By \cite[pg.~47]{ATLAS}, the largest centraliser of a conjugacy class of non-trivial elements of order larger than $2$ in $\PSp_6(2)$ has order 648. This is smaller than the smallest non-cyclic quotient of $\cale_6$, so $\alpha(\cale_6)$ is cyclic by \cref{thm:E6}, and $\alpha(\cale_8)$ is consequently cyclic by \cref{lem:E8_non_cyclic_E6}. If the image of $s_8$ has order 2, we instead conclude by \cref{rm:e8-coxeter-quotient}.
\end{proof}

\begin{lemma}
    If $\alpha\colon\cale_8\to \POmega_4^+(2)$ is a homomorphism, then $\alpha$ has cyclic image.
\end{lemma}
\begin{proof}
	Let $Q=\POmega_4^+(2)$.  By \cite[Page~86]{ATLAS}, every conjugacy class of non-trivial elements of order greater than 2 in $Q$ has centraliser smaller than the smallest non-cyclic quotient of $\cale_6$ except for three classes of order 3 elements: classes $3_A$, $3_B$, and $3_C$ in the ATLAS. Each of these three classes of order 3 elements has centraliser of order $77,760$, and the ATLAS prints their normalisers as $3\times\PSU_4(2)):2$ of order $155,520$.  It follows that the centraliser must be index $2$ and isomorphic to $3\times\PSU_4(2)$.  The map $\alpha$ cannot restrict to a non-cyclic map $\cale_6 \rightarrow \PSU_4(2)$ by \cref{lem:E6_order_murder}, and any homomorphism from $\cale_6$ to an abelian group is cyclic. The entire homomorphism $\alpha$ then has cyclic image by either \cref{lem:E8_non_cyclic_E6} or \cref{rm:e8-coxeter-quotient}. %This implies that $\cale_6$ admits a homomorphism to $\PSU_4(2)$, but by \Cref{lem:E6_order_murder}, any such homomorphism has cyclic image. 
\end{proof}

\begin{lemma}
    If $\alpha\colon\cale_8\to \POmega_4^-(2)$ is a homomorphism, then $\alpha$ has cyclic image.
\end{lemma}
\begin{proof}
    Let $Q=\POmega_4^-(2)$ and let $\alpha\colon\cale_8\to Q$ be a homomorphism.  By \cite[Page~88]{ATLAS} and the Orbit-Stabiliser Theorem, the only conjugacy class of elements with order larger than 2 and with centraliser larger than the smallest non-cyclic quotient of $\cale_6$ is the class $3_A$, whose centraliser has order 60480.  The ATLAS lists the normaliser of a cyclic subgroup generated by an element in the class $3_A$ as $(3\times\alt(8)):2$ of order 120,960.  It follows the centraliser of such an element has index 2 and is isomorphic to $3\times\alt(8)$.  By \cref{lem:E6_aut_alt} every homomorphism $\cale_6\to\alt(8)$ has cyclic image, so $\alpha$ must have cyclic image by \cref{lem:E8_non_cyclic_E6} or \cref{rm:e8-coxeter-quotient}.
\end{proof}
%CENTRALISER E6

\subsection{Exceptional and sporadic groups}

\begin{lemma}\label{lem:E8_G23_murder}
    If $\alpha \colon \cale_8 \to \Aut(\mathtt G_2(3))$ is a homomorphism, then $\alpha$ has cyclic image.
\end{lemma}
\begin{proof}
    Let $Q = G_2(3)$ and recall $G=\Aut(Q) = \mathtt G_2(3).2$. If the order of $\alpha(s_8)$ is 2, then we conclude by \cref{rm:e8-coxeter-quotient}. Otherwise, by \cite[Page~61]{ATLAS}, the smallest conjugacy class of elements of order larger than 2 in $Q$ is of size $728$. By the Orbit-Stabiliser theorem, the centraliser of an element in this conjugacy class must have order $5832$ in $Q$. This is smaller than the smallest non-cyclic quotient of $\cale_6$, so if $\alpha(\cale_8)$ lies in $Q$, it is cyclic by \cref{thm:E6} and \cref{lem:E8_non_cyclic_E6}.   
    
    Now, consider a homomorphism $\alpha\colon \cale_8\to G$. The subgroup $\alpha(\cale_6)$ is proper by \cref{lem:E8_proper_E6_subgroup}, so the image of $\cale_6$ must be contained in either $Q$ or a subgroup a group of shape $\PSU_3(3):2$, $\PSL_3(3):2$, $\PSL_2(8):3\times 2$,  $2^3\cdot\PSL_3(2):2$, $\PSL_2(13):2$, or a soluble group.  In each case, we obtain a homomorphism from $\cale_6$ to one of $\Aut(L)$ where $L\in\{\PSU_3(3),\PSL_3(2),\PSL_2(8),\PSL_2(13)\}$. Homomorphisms to the latter three automorphism groups have cyclic image by Lemmas \ref{lem:E6_PSL2q_murder},  \ref{lem:E6_aut_various_PSL2q_murder}, and \ref{lem:E6_order_murder}. The entire homomorphism $\alpha$ must then have cyclic image by \cref{lem:E8_non_cyclic_E6}.
\end{proof}

\begin{lemma}
    If $\alpha\colon\cale_8\to {}^3\mathtt D_4(2)$ is a homomorphism, then $\alpha$ has cyclic image.
\end{lemma}
\begin{proof}
  Let $G={}^3\mathtt D_4(2)$ and $\alpha\colon \cale_8\to G$ be a homomorphism.  By \cite[Page~90]{ATLAS}, the largest centraliser of an element of order larger than $2$ has order $3,584$.  This is smaller than the smallest non-cyclic quotient of $\cale_6$, so we conclude that $\alpha$ must have cyclic image by \cref{thm:E6} and \cref{lem:E8_non_cyclic_E6} or \cref{rm:e8-coxeter-quotient}. 
\end{proof}

\begin{lemma}
    If $\alpha\colon\cale_8\to M_{24}$ is a homomorphism, then $\alpha$ has cyclic image.
\end{lemma}
\begin{proof}
  Let $G=M_{24}$ and $\alpha\colon \cale_8\to G$ be a homomorphism.  By \cite[Page~96]{ATLAS}, the largest centraliser of an element of order larger than $2$ has order $1,080$.  This is smaller than the smallest non-cyclic quotient of $\cale_6$, so $\alpha$ must have cyclic image by \cref{thm:E6} and \cref{lem:E8_non_cyclic_E6} or \cref{rm:e8-coxeter-quotient}.
\end{proof}

\subsection{Products}

% \begin{lemma}\label{lem:E8_products_murder}
%     Let $Q$ be one of 
%     \begin{enumerate}
%         \item $\Aut(\alt(5)^5)\wr k$ for $k$ the order of an element of $\sym(5)\wr\sym(5)$;
%         \item $\Aut(\alt(6)^3)\wr k$ for $k$ the order of an element of $\sym(6).2\wr\sym(3)$; or
%         \item $\Aut(Q^2)\wr |\Out(Q)|$ for $Q$ one of $M_{11}^2$, $\PSL_3(3)^2$, or $\PSU_3(3)^2$;
%     \end{enumerate}
%     If $\alpha\colon \cale_8\to \Aut(Q)$ is a homomorphism, then $\alpha$ has cyclic image.
% \end{lemma}
\begin{lemma}\label{lem:E8_products_murder}
    Let $S^n$ be one of 
    \begin{enumerate}
        \item $\alt(5)^n$ with $n\leq 5$;
        \item $\alt(6)^n$ with $n\leq 3$;
        \item $M_{11}^2$;
        \item $\PSL_3(3)^2$;
        \item $\alt(7)^2$;
        \item $\PSU_3(3)^2$.
    \end{enumerate}
    Let $k$ be the order of an element of $\Out(S^n)$.  Let $Q$ be a group of shape $S^n.k$.
    If $\alpha\colon \cale_8\to Q$ is a homomorphism, then $\alpha$ has cyclic image.
\end{lemma}
\begin{proof}
We first observe the structure of $\Aut(S^n)$ in the first four cases: $\Aut(\alt(5)^n)\cong \sym(5)\wr\sym(n)$, $\Aut(\alt(6)^n)\cong (\sym(6).2)\wr \sym(n)$, $\Aut(M_{11}^2)\cong M_{11}\wr2$, and $\Aut(\PSL_3(3)^2)\cong \PGL_3(3)\wr 2$.  None of these groups have order divisible by $7$, so they do not have any elements of order 7.  Any positive integer $k$ whose order divides $|\Out(S^n)|$ is not divisible by $7$.  Thus, the wreath product $\Aut(S^n)\wr k$ does not contain any elements of order $7$.  The Universal Embedding Theorem \cite{KrasnerKaloujnine1951} implies every possible group $Q$ of shape $S^n.k$ is a subgroup of $\Aut(S^n)\wr k$ and so cannot have an element of order 7.  Hence by \Cref{lem:E8_7s_9s}, every homomorphism $\alpha\colon \cale_8\to Q$ has cyclic image.
    
   Now, $|\Out(\alt(7)^2)|=8$.  Thus, any possible group of shape $\alt(7)^2.k$ must have $k$ a power of $2$ dividing 8.  The group $\Aut(\alt(7)^2)\wr8\cong(\sym(7)\wr2)\wr8$ does not have an element of order 9. Indeed, a 3-Sylow subgroup is isomorphic to the direct product of sixteen copies $3^2$, i.e., it is isomorphic to $3^{32}$.  This is easily seen from the fact that a Sylow 3-subgroup of $\sym(7)$ is isomorphic to $3^2$.    In particular, $\Aut(\alt(7))\wr8$ has no elements of order $9$. By \Cref{lem:E8_7s_9s} every homomorphism from $\alpha\colon \cale_8\to \Aut(\alt(7)^2)\wr8$ must be cyclic.  The Universal Embedding Theorem \cite{KrasnerKaloujnine1951} implies every possible group $Q$ of shape $\alt(7)^2.k$ with $k|8$ is a subgroup of $\Aut(\alt(7)^2)\wr k$.  Hence every homomorphism $\alpha\colon \cale_8\to Q$ has cyclic image.
    
    The group $\Aut(\PSU_3(3)^2\cong \PGU_3(3)\wr 2$ does not have an element of order $9$.  Indeed, a 3-Sylow subgroup is isomorphic to the direct product of 2 copies of $\He_3(3)$, the isomorphism class of the 3-Sylow subgroup of $\PSU_3(3)$, since $|\Out(\PSU_3(3)^2)|=8$.  We conclude as in the last paragraph that every homomorphism $\cale_8$ to $\Aut(\PSU_3(3)^2)\wr8$ must be cyclic.  The Universal Embedding Theorem \cite{KrasnerKaloujnine1951} implies every possible group $Q$ of shape $\PSU_3(2)^2.k$ with $k|8$ is a subgroup of $\Aut(\PSU_3(2)^2)\wr k$.  Hence every homomorphism $\alpha\colon \cale_8\to Q$ has cyclic image.
\end{proof}

\begin{lemma}\label{lem:E8_PSL2qn_murder}
    Let $Q=\PSL_2(q)^n$.  If $n\leq 12$, then any homomorphism $\alpha\colon \cale_8\to \Aut(Q)$ has cyclic image.  If $n$ is arbitrary, then any homomorphism $\alpha\colon \cale_8\to Q.n$ has cyclic image.
\end{lemma}
\begin{proof}
    First, we note that any homomorphism $\alpha\colon\cale_8 \to \PSL_2(q)$ induces a homomorphism $\cale_7'\to\PSL_2(q)$. By \cref{lem:E7_comm_PSL2q}, the image of $\cale_7'$ is trivial. This implies that $\alpha(\cale_7)$ factors through the abelianization of $\cale_7$, and is thus cyclic. By Lemma \ref{lem:E8_non_cyclic_E6}, this implies that $\alpha(\cale_8)$ is also cyclic.
    
    Next, we deal with the group $\Aut(\PSL_2(q)^n))\cong\Aut(\PSL_2(q))\wr\sym(n)$ with $n\leq 12$.  By \cref{lem:E8-sym12-murder}, the projection of the image of such a homomorphism to $\sym(n)$ must be cyclic. Let $C$ denote this cyclic group. The image $\alpha(\cale'_7) < \alpha(\cale_8')$ must be contained in the commutator subgroup of $\Aut(\PSL_2(q))\wr C$, which is a product of copies of $\PGL_2(q)$. Since $\cale_7'$ is perfect, it in fact must be contained in a product of copies of $\PSL_2(q)$. Projection onto any factor induces a homomorphism  $\cale_7'\to\PSL_2(q)$, and every such homomorphism has trivial image by \cref{lem:E7_comm_PSL2q}. Then $\alpha(\cale_7)$ must be cyclic, so \cref{lem:E8_non_cyclic_E6} implies $\alpha(\cale_8)$ is cyclic.   
    %Induces a homomorphism $\cale_7'\to\PSL_2(q)$, LEMMA says this is cyclic.  Now, we deal with a group $\Aut(\PSL_2(q)^n))\cong\PGL_2(q)\wr\sym(n)$.  By LEMMA, the projection of the image of such a homomorphism to $\sym(n)$ must be cyclic.  Hence we obtain an induced homomorphism $\cale_7'\to\PSL_2(q)$, LEMMA says this is cyclic.
\end{proof}

\subsection{Completing the proof}

\begin{thm}\label{thm:E8}
    The smallest non-abelian quotient of $\cale_8$ is isomorphic to $W(\cale_8) / Z(W(\cale_8))$. It is unique up to automorphisms of the image.
\end{thm}
\begin{proof}
    %Let $Q$ denote a smallest non-abelian quotient of $\cale_8$. Since the abelianization of $\cale_8$ is cyclic and $\cale_8'$ is perfect, by Lemma \ref{lem:min_quotient_En}, it suffices to rule out the case where $Q$ is subgroup of $\Aut(S^n)$. We need only consider the cases where $|S^n| \leq |W(\cale_8) / Z(W(\cale_8))|$ because Lemma \ref{lem:min_quotient_En} also specifies that $S^n \cong \soc(Q) \leq Q$. The orders of all finite simple groups $S$ with order at most $|W(\cale_8)|$ are provided in tables \ref{tab:E7simple} and \ref{tab:E8simple}.
    We now exclude all groups listed in \Cref{lem:E8_hitlist}.  Notice that any homomorphism $\alpha: \cale_8 \rightarrow \Aut(S)$ induces a homomorphism $\cale_7 \rightarrow \Aut(S^n)$ by restricting to the standard parabolic subgroup of type $\cale_7$. Corollary \ref{cor:aut-psl2-e7-murder} and Lemma \ref{lem:E7-non-A1-murder} showed that any homomorphism from $\cale_7$ to the automorphism group of any of the groups in Lemma \ref{lem:E7_hitlist} has cyclic image, so the image of $\cale_8$ under any such homomorphism must also be cyclic by \cref{lem:E8_non_cyclic_E6}.  This handles case (1).

    Note that all of the groups in Table \ref{tab:E8simple} are large enough that $|S^2| > |W(\cale_8)|$, so in these cases, it suffices to consider only $\Aut(S)$. Lemma \ref{lem:E8-7s-9s-murder} rules out the possibility that the smallest non-abelian quotient $Q$ of $\cale_8$ is a subgroup of the automorphism group of any of the groups in Table \ref{tab:E8simple} except possibly $\alt(10)$, $\alt(11)$, $\PSL_3(8)$, $\PSL_3(9)$, $\PSU_3(8)$, $\PSp_6(2)$, $\PSU_4(3) \cong O_6^{-}(3)$, and $\mathtt G_2(3))$. 
    Lemma \ref{lem:E8-sym12-murder} eliminates the possibilities that $Q$ is a subgroup of either $\sym(12)=\Aut(\alt(12))$, $\Aut(\alt(11))$ or $\Aut(\alt(10))$, the latter two of which are subgroups of the former. Lemmas \ref{lem:E8_PSL38-murder}, \ref{lem:E8_PSL39_murder}, \ref{lem:E8_PSU38_murder}, \ref{lem:E8_PSp62_murder}, \ref{lem:E8_PSU43_murder}, and \ref{lem:E8_G23_murder} eliminate the possibilities that $Q$ is a subgroup of $\Aut(\PSL_3(8))$, $\Aut(\PSL_3(9))$, $\Aut(\PSU_3(8))$, $\Aut(\PSp_6(2))$, $\Aut(\PSU_4(3))$, or $\Aut(\mathtt G_2(3))$ respectively.   This handles case (2).

    The remaining cases correspond to the socle of the smallest non-abelian quotient $Q$ of $\cale_8$ being a direct product $S^n$. These are listed in \Cref{lem:E8_products_murder} except when $S\cong\PSL_2(q)$ which is given in Lemma \ref{lem:E8_PSL2qn_murder}.  The lemmas show any the image of $\cale_8$ under any homomorphism to one of these cases is cyclic.  This completes the proof that $\bar W=W(\cale_8) / Z(W(\cale_8))$ is the smallest non-abelian quotient of $\cale_8$.

    We now prove that every epimorphism $\cale_8\to \bar W$ factors through $W=W(\cale_8)$.  In particular, the quotient is unique up to automorphisms of the image.  We have that $\bar W\cong \POmega^+_4(2).2$ and as usual we will consider possible element orders of the image $s_8$ and the images of $\cale_6$ in it centraliser.  The only conjugacy classes with centraliser at least as large as $W(\cale_6)$ are those listed as $3_A$, $3_B$, and $3_C$ in the ATLAS \cite[pg~86]{ATLAS}.  In particular, if the epimorphism does not factor through the canonical Coxeter quotient (up to automorphisms), then the images $\bar s_1,\dots \bar s_8$ of $s_1,\dots,s_8$ must have order $3$ and be contained in one of $3_A$, $3_B$, or $3_C$.  The ATLAS lists these conjugacy classes as having normaliser isomorphic to $\sym(3)\times \SU_4(2):2$.  One deduces that the centraliser $C_{\bar W}(\bar s_8)$ must be of the shape $3\times \SU(4):2$ with $C_{\bar W}(\bar s_8)/\langle \bar s_8\rangle \cong \SU_4(2)\cong W(\cale_6)$.  But by \Cref{thm:E6}, every homomorphism with non-cyclic image $\beta\colon \cale_6 \to W(\cale_6)$ is equal to the canonical Coxeter homomorphism up to an automorphism.  But this forces $\overline s_1$ to have order $2$, contradicting that it had order $3$.  
\end{proof}

\section{Spherical Artin groups}\label{sec:spherical}

By combining previous theorems of this paper together with the work of Kolay, we have proven the first of our main theorems.

\setcounter{thmx}{1}
\begin{thmx}\label{thmx:spherical}
    Let $G$ be an irreducible spherical Artin group.  The following conclusions hold:
    \begin{enumerate}
        \item If $G$ is isomorphic to $\cala_{n-1}$, $\calb_n$, or $\cald_n$, for $n\geq 5$, then the smallest non-abelian quotient of $G$ is $\sym(n)$ and it is unique up to automorphisms of the image.
        \item If $G$ is isomorphic to one of $\cala_2$ $\cala_3$, $\calb_3$, $\calb_4$, $\cald_4$, $\calf_4$, or $\cali_2(m)$ with $4|m$, then the smallest non-abelian quotient of $G$ is $\sym(3)$.
        \item If $G$ is isomorphic to $\cali_2(m)$ with $4\not\mid m$ then the smallest non-abelian quotient of $G$ is $D_p$, where $p$ is the smallest odd prime divisor of $m$, and it is unique up to automorphisms of the image.
        \item If $G$ is isomorphic to $\calh_3$, then the smallest non-abelian quotient of $G$ is $\alt(5)$.  Up to automorphisms of the image there are two such homomorphisms.
        \item If $G$ is isomorphic to $\calh_4$, $\cale_6$, $\cale_7$, or $\cale_8$, then the smallest non-abelian quotient of $G$ is isomorphic to $W(G)/Z(W(G))$ and it is unique up to automorphisms of the image.
    \end{enumerate}
\end{thmx}
\begin{proof}
Combine Theorems~\ref{kolay}, \ref{thm:I2}, \ref{thm:Cn}, \ref{thm:Dn}, \ref{thm:F4quotient}, \ref{thm:H3}, \ref{thm:H4}, \ref{thm:E6}, \ref{thm:E7}, and \ref{thm:E8}.
\end{proof}

\section{Affine Artin groups}\label{sec:affine}

In this section, we complete the proof of \Cref{thmx:affine}.

\begin{thmx}\label{thmx:affine}
        Let $G$ be an irreducible affine Artin group.  The following conclusions hold:
    \begin{enumerate}
        \item If $G$ is isomorphic to $\widetilde\cala_{n-1}$, $\widetilde\calb_n$, $\widetilde\calc_n$ or $\widetilde\cald_n$, for $n\geq 5$, then the smallest non-abelian quotient of $G$ is $\sym(n)$.
        \item If $G$ is isomorphic to one of $\widetilde\cala_1$, $\widetilde\cala_2$ $\widetilde\cala_3$, $\widetilde\calb_2$, $\widetilde\calb_3$, $\widetilde\calb_4$, $\widetilde\cald_4$, $\widetilde\calf_4$, or $\widetilde\calg_2$, then the smallest non-abelian quotient of $G$ is $\sym(3)$.
        \item If $G$ is isomorphic to $\widetilde\cale_n$ for $n=6,7,8$ then the smallest non-abelian quotient of $G$ is isomorphic to $W(\cale_n)/Z(W(\cale_n))$.
    \end{enumerate}
\end{thmx}

\begin{proof}
    The group $\widetilde\cala_1$ is a rank $2$ free group, so its smallest non-abelian quotient is $\sym(3)$. Let $G=\widetilde\calx_n$ be an irreducible affine Artin group such that $G \neq \widetilde \cala_1$.  Now, $W(\widetilde\calx_n)\cong B_\calx\rtimes W(\calx_n)$ where $B_\calx$ is free abelian and acts by translations on $\EE^n$, see \cite[\S4.2 Proposition]{Humphreys1992}.  Thus, we have an epimorphism 
    $$\widetilde \calx_n\to W(\widetilde\calx_n) \to W(\calx_n).$$  
    This proves (2) noting $\calg_2\cong \cali_2(6)$.

    We now prove (1).  We first prove the result for $G=\widetilde \cala_{n-1}$ with $n\geq 5$.  In particular, by \Cref{thmx:spherical} the image of any $\cala_{n-1}$ in a finite non-abelian quotient of $G$ must either be cyclic or have size at least $n!$ with equality if the image is isomorphic to $\sym(n)$.  By \Cref{lem:artin-gen-collapsing}, the Artin generators of $G$ form a collapsing set. If the image of any such $\cala_{n-1}$ is cyclic, then the image of $G$ must be cyclic as well by Lemma \ref{lem:E8_non_cyclic_E6}. Hence the smallest non-abelian quotient is $\sym(n)$ as required. 
    % \thomas{This shows there's a lower bound of $n!$.  I guess it can be computed by hand or computer?}
    
    We next prove the result for $G=\widetilde\calb_n$ with $n\geq 5$.  Note that there are two parabolic $\cala_{n-1}$ subgroups of $G$ that together generate a parabolic $\cald_{n}$ subgroup. By \Cref{thmx:spherical}, in any quotient, these either have cyclic image or size at least $n!$, with equality if the image is isomorphic to $\sym(n)$.  We now proceed as in  \Cref{thm:Cn}.  If the image of either parabolic $\cala_{n-1}$ is cyclic, then the quotient collapses the entire $\cald_n$ to a cyclic group and the quotient must factor through $\langle \overline s_1, \overline s_2\ |\ [\overline s_1, \overline s_2]_4,\ [\overline s_1, \overline s_2] \rangle\cong \Z^2$ where $\bar{s_1}$ generates the image of the parabolic $\cald_n$ and $\bar{s_2}$ generates the image of the one remaining generator. In particular, the quotient of $G$ must be abelian.  Hence the smallest non-abelian quotient is $\sym(n)$ as required.  
    
    The proofs of $\widetilde\calc_n$ and $\widetilde \cald_n$ with $n\geq 5$ follow similar arguments.

    To prove (3) note that $\cale_n$ is naturally a subgroup of $\widetilde\cale_n$ for $n=6,7,8$.  By \Cref{lem:artin-gen-collapsing}, the Artin generators of $G=\widetilde\cale_n$ form a collapsing set.  In particular, by \Cref{thmx:spherical}, the image of $\cale_n$ in a finite non-abelian quotient of $G$ must have size at least $|W(\cale_n)/Z(W(\cale_n))|$ with equality if the image is isomorphic to $W(\cale_n)/Z(W(\cale_n))$.  Hence the theorem.
\end{proof}

\section{The isomorphism problem for affine Artin groups}\label{sec:iso}

In this section we prove \Cref{thmx:iso}.

\begin{lemma}\label{lem:affine_cd}
    Let $A_\Gamma$ be an irreducible affine Artin group other than $\widetilde A_1$.  If $|V(\Gamma)|=n$, then $\cd(A_\Gamma)=n$.
\end{lemma}
\begin{proof}
    The $K(\pi,1)$ conjecture holds for affine Artin groups \cite{PaoliniSalvetti21}.  This immediately yields the upper bound.  Recall that the \emph{nerve} $N(\Gamma)$ of a Coxeter diagram $\Gamma$ is the simplicial complex with $1$-skeleton equal to $\Gamma$ and vertices $v_0,\dots v_k$ span a $k$-simplex if and only if the Artin group of the subdiagram spanned by $v_0,\dots,v_k$ is spherical type.  Now, in \cite[Theorem 28]{DavisLeary2003} Davis--Leary prove that if $\caln(A_\Lambda)$ denotes the von Neumann Algebra of any Artin group $A_\Lambda$ satisfying the $K(\pi,1)$ conjecture, then
    \[H^k(A_\Lambda;\caln(A_\Lambda))\cong \overline H^{k-1}(N(\Lambda);\QQ)\otimes_\QQ \caln(A_\Lambda),\]
    where $\overline H^{k-1}(N(\Lambda);\QQ)$ is the ordinary reduced rational homology of $N(\Lambda)$.
    Note that Davis--Leary state their result for $L^2$-homology, but for groups with finite classifying spaces the $L^2$-homology and cohomology groups are isomorphic, see \cite[Remark 1.31]{Lueck2002book}.  
    
    For an irreducible affine Artin group $A_\Gamma$ with $|V(\Gamma)|=n$, the nerve $N(\Gamma)$ is the $(n-2)$-skeleton of the $(n-1)$-simplex.  That is $N(\Gamma)\cong S^{n-1}$.  It follows that we have
    $$H^k(A_\Gamma;\caln(A_\Gamma)) \cong \overline H^{k-1}(S^{n-1};\QQ)\otimes_\QQ \caln(A_\Gamma)$$ implying
    \[H^k(\Gamma;\caln(A_\Gamma))\cong \begin{cases}
        \caln(A_\Gamma) & k=n;\\
        0   & \text{otherwise.}
    \end{cases}\]
    Whence the lemma.
\end{proof}

\begin{lemma}\label{lem:B4tilde_F4tilde}
    The Coxeter group $W(\calf_4)$ is not a quotient of the Artin group $\widetilde \calb_4$.
\end{lemma}
\begin{proof}
    We verify this using MAGMA, the code is provided in \Cref{App:lem_B4tild_F4tilde}.
\end{proof}

\begin{center}
\renewcommand{\arraystretch}{1.4}
\begin{table}[ht]
    \centering
    \begin{tabular}{|c|c|c|c|c|c|c|c|c|}
    \hline
        Diagram & $\widetilde A_1$ & $\widetilde A_n$, $n\geq 2$ & $\widetilde B_n$, $n\geq 3$ & $\widetilde C_n$ $n\geq 2$ & $\widetilde D_n$, $n\geq 4$ & $\widetilde E_n$ $n=6,7,8$ & $\widetilde F_4$ &  $\widetilde G_2$  \\
        \hline
        $b_1$ &2 & 1 & 2 &3 &1 &1 &2 &2\\
        \hline
    \end{tabular}
    \caption{The first Betti number of an irreducible affine Artin group.}
    \label{tab:b1_Affine}
\end{table}
\end{center}

\begin{thm}\label{thm:affine_iso}
    The isomorphism problem is solvable amongst affine Artin groups: if $\Gamma$ and $\Lambda$ denote affine Coxeter diagrams, then $A_\Gamma\cong A_\Lambda$ if and only if $\Gamma\cong \Lambda.$
\end{thm}

\begin{proof}
    We first reduce to the irreducible case. 
    
    The $K(\pi,1)$ conjecture holds for affine Artin groups \cite{PaoliniSalvetti21}.  Thus affine Artin groups are torsion-free and have trivial centre by \cite{JankiewiczSchreve2023}.  
    Since irreducible affine Artin groups are acylindrically hyperbolic \cite{Calvez2022}, they cannot be decomposed as a direct product of two infinite groups \cite[Corollary 7.2]{Osin2016}. 
    Moreover, since affine Artin groups are torsion free the are indecomposable as a direct product of any two nontrivial groups.
    % Since irreducible affine Artin groups are torsion-free and acylindrically hyperbolic \cite{Calvez2022}, they are indecomposable as a direct product of two non-trivial groups \cite[Corollary 7.2]{Osin2016}. 
    Henceforth, we will refer to groups which are indecomposable as a direct product of two non-trivial groups as simply \emph{indecomposable}.
    
   % By \cite[Theorem 4.2]{Senden2021}, for a direct product of groups $G=\prod_{i=1}^rG_i$ with each $G_i$ indecomposable and centre-free, we have $\Aut(G)\cong \left(\prod_{i=1}^r\Aut(G_i)\right)\rtimes \mathrm{Perm}$ where $\mathrm{Perm}$ is the subgroup of $\sym(r)$ which permutes isomorphic copies of the $G_i$.  In particular if $\prod_{i=1}^r G_i \cong \prod_{j=1}^s H_j$ with each $H_j$ indecomposable and centre-free, then $r=s$ and up to a permutation of $\{1,\dots,r\}$ we have $G_i\cong H_i$.  It follows that if $A_\Gamma=\prod_{i=1}^r A_{\Gamma_i}$ and $A_\Lambda=\prod_{i=1}^sA_{\Lambda_i}$ are affine Artin groups with $A_\Gamma\cong A_\Lambda$, then $r=s$ and there exists a permutation of $\{1,\dots, r\}$ such that $A_{\Gamma_i}\cong  A_{\Lambda_i}$.  This completes the reduction.
   
   By \cite[Page 264]{Korinek1937} or \cite[Page 116, Theorem 12]{Zassenhaus1999}, for a direct product of groups $G=\prod_{i=1}^rG_i\cong \prod_{j=1}^s H_j$ with each $G_i$ and $H_j$ indecomposable and centre-free, there exists a permutation of $\{1,\dots,r\}$ such that we have $G_i\cong H_i$.  It follows that if $A_\Gamma=\prod_{i=1}^r A_{\Gamma_i}$ and $A_\Lambda=\prod_{i=1}^sA_{\Lambda_i}$ are affine Artin groups with $A_\Gamma\cong A_\Lambda$, then $r=s$ and there exists a permutation of $\{1,\dots, r\}$ such that $A_{\Gamma_i}\cong  A_{\Lambda_i}$.  This completes the reduction.

    We now assume that $A_\Gamma$ and $A_\Lambda$ are irreducible affine Artin groups with $A_\Gamma\cong A_\Lambda$.  By \Cref{lem:affine_cd} we assume for the remainder of the proof that $\cd(A_\Gamma)=\cd(A_\Lambda)=n$ is the number of standard generators in each Artin group.
    
    Since $\widetilde \cala_1$ is a rank 2 free group it has $\cd(\widetilde \cala_1) = 1$.  This distinguishes it from the others.  
    Moving on to the non-exceptional families observe that types $\widetilde \cala_n$, $\widetilde \calb_n$, $\widetilde \calc_n$ and types $\widetilde \cald_n$, $\widetilde \calb_n$, $\widetilde \calc_n$ are pairwise distinguished by the first Betti number, see \Cref{tab:b1_Affine}.  Thus for infinite families it remains to distinguish type $\widetilde \cala_n$ from $\widetilde \cald_n$.  By \Cref{thmx:affine}, for $n\geq 5$ we have that the smallest non-cyclic quotient of $\widetilde \cala_n$ is $\sym(n+1)$ but the smallest non-cyclic quotient of $\widetilde \cald_n$ is $\sym(n)$.  For the remaining case $n=4$ we have that the smallest non-cyclic quotient of $\widetilde \cala_4$ is $\sym(5)$ but the smallest non-cyclic quotient of $\widetilde \cald_4$ is $\sym(3)$.   This distinguishes the infinite families up to isomorphism.

    We now treat the exceptional irreducible affine Artin groups.  To distinguish $\widetilde \calg_2$ from $\widetilde \cala_2$ and $\widetilde \calc_2$ we appeal to the data in \Cref{tab:b1_Affine}.  To distinguish $\widetilde \cale_n$ for $n=6,7,8$ from $\widetilde \cala_n$, $\widetilde \calb_n$, $\widetilde \calc_n$, and $\widetilde \cald_n$ we again apply \Cref{thmx:affine}.  Finally, we must distinguish $\widetilde \calf_4$ from $\widetilde \cala_4$, $\widetilde \calb_4$, $\widetilde \calc_4$, and $\widetilde \cald_4$.  Appealing to \Cref{tab:b1_Affine} we only need to distinguish $\widetilde \calb_4$ from $\widetilde \calf_4$.  By \Cref{lem:B4tilde_F4tilde} we see that $\widetilde \calb_4$ does not admit $W(\calf_4)$ as a finite quotient, but $\widetilde \calf_4$ does.  This completes the proof of the theorem.
\end{proof}

\setcounter{thmx}{0}
\begin{thmx}\label{thmx:iso}
    The isomorphism problem is solvable amongst all spherical and affine Artin groups: if $\Gamma$ and $\Lambda$ denote spherical or affine Coxeter diagrams, then $A_\Gamma\cong A_\Lambda$ if and only if $\Gamma\cong \Lambda$.
\end{thmx}
\begin{proof}
    Spherical Artin groups have non-trivial centre, thus distinguishing spherical from affine Artin groups.
    The isomorphism problem for spherical Artin groups was proved by Paris \cite{Paris2004}, so the result follows from \Cref{thm:affine_iso}.
\end{proof}

\section{Profinite rigidity}\label{sec:profinite}
In this section we will prove \Cref{thmx:profinite}.  The arguments are of a distinctly different flavour to much of the rest of the paper and combine \Cref{thmx:spherical} with some standard techniques from profinite rigidity.  We the refer the reader to Reid's excellent survey \cite{Reid2015} which introduces all of the notions we will use below.  
 We denote by $\widehat G$ the profinite completion of a group $G$.  Note that by \cite{DixonFormanekPolandRibes1982}, if $G$ is finitely generated this contains the same information as the set of isomorphism classes of finite quotient groups of $G$.

\begin{lemma}\label{lem:goodness}
    The following hold:
    \begin{enumerate}
        \item \cite[Proposition 1]{Marin2012} Let $G$ be one of $\cala_n$, $\calb_n$, $\cald_n$, and $\cali_2(m)$.  Then $G$ is good in the sense of Serre.  Namely, for every finite discrete $\widehat G$-module $M$, the natural map $H^k(\widehat G;M)\to H^k(G;M)$ is an isomorphism for all $k\geq 0$.
        \item Let $G$ and $H$ be irreducible spherical Artin groups of type $A$, $D$, or $I_2(2m+1)$.  If $\widehat G\cong \widehat H$, then $G\cong H$.
        \item
            \begin{enumerate}
                \item For $n\geq 1$ we have $\cd(\widehat \cala_n)=n$, 
                \item for $n\geq 3$ we have $\cd(\widehat\calb_n)=n$, 
                \item for $n\geq 4$ we have $\cd(\widehat\cald_n)=n$, 
                \item $\cd(\cali_2(m))=2$,
                \item $\cd(\widehat \calf_4)\geq 3$.
            \end{enumerate}
        \item Neither $\calb_3$ nor $\calb_4$ admits $W(\calf_4)$ as a quotient.
    \end{enumerate}
\end{lemma}
\begin{proof}
    (1) is as cited. (2) follows from applying (1) to the discussion after Question 3 in \cite{Marin2012} which in turn depends on \cite{CallegaroMarin2014} (note this reference is quite secondary for cohomology of the Artin groups we are interested in, so we also refer the reader to the classics \cite{Arnold1969,Deligne1972,Brieskorn1973,Goryunov1978,Vajushtejn1978}).

    We now prove (3).  Begin by letting $G$ be one of the groups in cases (a)-(d). In case (d), we let $n=2$.  
    Note that the goodness proved in (1) implies that $\cd(\widehat G)\leq n$, where $\cd(\widehat G)$ refers to the cohomological dimension of $\widehat G$; see \cite[Proposition~7.4]{Reid2015}.  
    Thus we need to exhibit some non-trivial cohomology in degree $n$ for some finite discrete $\widehat G$-module.  
    
    The kernel of the map $G\to W(G)$, denoted $PG$, is the fundamental group of an aspherical complexified hyperplane arrangement $X$ which has non-vanishing cohomology in degree $n$ \cite{Deligne1972,Brieskorn1973}. 
    Moreover, the $n$th integral cohomology of $X$ is non-zero \cite{Deligne1972,Brieskorn1973}.
    The Universal Coefficient Theorem implies $H^n(X;\FF_p)\neq 0$ for some prime $p$.   
    By asphericity of $X$, goodness of $G$ proved in (1), and several applications of Shapiro's Lemma \cite{Brown1982} we obtain isomorphisms
    $$H^n(X;\FF_p)\cong H^n(PG;\FF_p) \cong H^n(G;\FF_p[W(G)]) \cong H^n(\widehat G;\FF_p[W(G)]).$$
    Hence, $\cd(\widehat G)=n$ as required.

    We now treat case (e).  Since cohomological dimension is monotonic on subgroups, it suffices to exhibit a subgroup of $\widehat \calf_4$ which has cohomological dimension at least 3.  There is a retract 
    $$\alpha\colon \calf_4\twoheadrightarrow A=\langle s_1,s_2,s_4\rangle \cong \cala_2\times\Z$$ given by $s_3\mapsto s_4$ and fixing the other generators.  
    It follows from \cite[Lemma 2.2]{Minasyan2021} that $A$ has the full induced profinite topology, that is every finite index subgroup of $A$ is separable in $\calf_4$.  In particular, by \cite[Lemma~4.6]{Reid2015} we have $\widehat A\rightarrowtail \widehat \calf_4$.  That is, $\widehat A$ is a subgroup of $\widehat G$.  Moreover by (3)(a) and the K\"unneth formula we have $\cd(\widehat A)=3$, as required.

    The proof of (4) is a simple computation in MAGMA, see \Cref{Code:11.1}.
\end{proof}

\setcounter{thmx}{3}
\begin{thmx}\label{thmx:profinite}
    Let $G$ and $H$ be irreducible spherical Artin groups.  If $\widehat G\cong \widehat H$, then $G\cong H$.
\end{thmx}
\begin{proof}
    First suppose that $b_1(G)=1$, where $b_1(G)=\dim_\QQ H^1(G;\QQ)$ refers to the first Betti number of $G$.
    By \cite[Proposition~3.2]{Reid2015}, $b_1(H)=1$.  
    In particular, $H$ is not $\cali_{2}(2m)$, $\calf_4$, or $\calb_n$ since these groups satisfy $b_1>1$.  
    The result follows for the exceptional Artin groups $\calh_3$, $\calh_4$, $\cale_6$, $\cale_7$, and $\cale_8$ by \Cref{thmx:spherical}.  The remaining cases with $b_1=1$ are all of type $A$, $D$, and $I_2$ and so the result follows from \Cref{lem:goodness}.

    Next, suppose $b_1(G)=2$.  By by \cite[Proposition~3.2]{Reid2015}, $b_1(H)=2$.  
    Thus, $H$ is either $\cali_{2}(2m)$, $\calf_4$, or $\calb_n$.  For $n\geq 5$ the group $\calb_n$ is distinguished by \Cref{thmx:spherical}.  The computations of cohomological dimension in \Cref{lem:goodness}(3) distinguish types $\calb_3$, $\calb_4$, and $\calf_4$ from $\cali_2(2m)$.   We distinguish $\calb_3$, and $\calb_4$ from $\calf_4$ by \Cref{lem:goodness}(4).  
    
    It remains to distinguish $\cali_{2}(2m)$ from $\cali_2(n)$ when $m\neq n$.  The idea to use cup products in the sequel is inspired by \cite[Proof of Theorem~4.12]{HLIPSV2025}.  By work of Landi \cite[Table II]{Landi2000}, we have 
    $$H^\ast(\cali_2(2m);\Z)=\Z[x,y,z]/(x^2,y^2,xy=mz)$$
    where $x,y$ have degree 1 and z has degree 2, so 
    $$H^n(\cali_2(2m));\Z\cong\begin{cases}
        \Z & n=0,2; \\
        \Z & n=1; \\
        0 & \text{otherwise.}
    \end{cases}$$
    By the Universal Coefficient Theorem we obtain 
    $$H^n(\cali_2(2m);\Z/k\Z)\cong\begin{cases}
        \Z/k\Z & n=0,2; \\
        (\Z/k\Z)^2 & n=1; \\
        0 & \text{otherwise.}
    \end{cases}$$
    A careful reading of \S5.1, Table I and Table II of \cite{Landi2000} and substituting the $\Z$ coefficients for $\Z/k$ coefficients, yields that 
    $$H^\ast(\cali_2(2m);\Z/k\Z))=(\Z/k\Z)[x,y,z]/(x^2,y^2,xy=mz).$$
    Note that \cite[Example~1]{Landi2000} is particularly instructive.  It follows that $m$ equals the smallest integer $k\geq 2$ such that $H^\ast(\cali_2(2m);\Z/k\Z))$ has no non-trivial cup products.  Since cup products are preserved by goodness see \cite[\S2.C.1]{HughesNg2026} (or \cite[Proposition~6]{KrophollerWilkes2016}), the same is true for $H^\ast(\widehat \cali_2(2m);\Z/k\Z))$.  Whence, the theorem.
\end{proof}

\begin{appendix}

\section{Quotient computations}\label{apx:computations}
 This section contains the computations of Section \ref{subsec:element-orders}. Our computations are always accomplished by applying one of a small handful of possible rearrangements of a positive word $w$ in the Artin group and by cancelling adjacent pairs of letters which consist of an Artin generator and its inverse. While none of these rearrangements are themselves difficult, it is somewhat tedious to work out which rearrangements should be applied in which order to obtain the desired equality. Rather than present lengthy line-by-line computations, we develop the following notation to record which sequence of rearrangements we used. We recommend that a reader interested in following through these steps does write them out line-by-line. We note that the sequences are not necessarily unique; we simply provide one effective option.

 \subsection{Notation}
 Throughout, let $S$ denote the collection of Artin generators, and let $w$ denote a positive word in this generating set. We will often equate the positive word $w = s_1 s_3 s_5 s_2 s_1$, for example, with the string of integers $13521$. When necessary, we will equate the inverse of the positive word $13521$ with the string $1^{-1} 2^{-1} 5^{-1} 3^{-1} 1^{-1}$. We will primarily use this to improve readability when a word is used as a subscript.

 We will be almost exclusively concerned with simplifying equations of the form $w s_i w^{-1} = v s_j v^{-1}$ or $w s_i w^{-1} = v^{-1} s_j v$. Because $w$ is a positive word, it will always be clear which generator should be the distinguished $s_j$; it is the unique generator which is adjacent to both a positive letter and an inverse. There are a few methods of rearranging one or both sides of the equation which will be used repeatedly; we classify them below.

 Throughout, we will adopt the convention that any subword of the form $s_i s_i^{-1}$ should be immediately deleted. We do not mark rearrangements of this kind. We are also using the Artin generator labellings of Figure \ref{fig:FiniteCoxeterGroups} throughout.

 The majority of our methods are \emph{sided operations}. These are rearrangements that are applied to a particular side of the equation, so either to $w s_i w^{-1}$ or to $v s_j v^{-1}$, but not both. The correct side of the equation will be denoted via an $RHS$ or $LHS$ preceding the relevant step except in the case where we have defined a variable $\epsilon$ and are attempting to reduce an expression of the form $\epsilon s_i \epsilon^{-1} := w s_i w^{-1}$. In this setting, it is clear that sided operations should be applied to the right hand side.
 \begin{itemize}
     \item \textbf{$M_{ij, \square}$:} Suppose that $w$ in the above equation contains a unique subword which is an alternating product $s_i s_j ...$ of length $m_{ij}$. To denote that one should apply the Artin relation of length $m_{ij}$ to this subword and to the corresponding subword of $w^{-1}$, we write $M_{ij}$. 
     
     If there are multiple such subwords, we write $M_{ij, n}$ to indicate that one should apply the Artin relation to the alternating product subword beginning at the $n^{th}$ appearance, when read from left to right, of the letter $s_i$ in $w$, but this is rarely necessary. To avoid confusion between the index of the Artin generator and the position, we write the position $n$ in Roman numerals when it is needed, e.g., $M_{12, IV}$ means that one should apply the Artin relation $m_{12}$ beginning at the 4th appearance of the letter $s_1$ in $w$ and working left to right, and at the 4th appearance of $s_1^{-1}$ in $w^{-1}$ working right to left.

     \begin{example} 
     Define $\epsilon = s_1 s_2 s_1$ in the Artin group $\cala_3$. To simplify $\epsilon s_2 \epsilon^{-1} := (s_1 s_2 s_1) s_2 (s_1^{-1} s_2^{-1} s_1^{-1})$, we can apply $M_{12}$. We do not need a position indicator, because there is only one alternating product of $s_1$ and $s_2$ of the appropriate length in $w$ and $w^{-1}$. Applying $M_{12}$ and cancelling the resulting adjacent $s_2$ and $s_2^{-1}$ yields the shorter expression $\epsilon s_2 \epsilon^{-1} = (s_2 s_1) s_2 (s_1^{-1} s_2^{-1})$.
    \end{example}

     \item \textbf{$\Tilde{M_{ij}}$:} Recall that Artin relations yield a conjugacy relationship in an obvious way: if two generators $s_i$ and $s_j$ satisfy a braid relation, then $(s_i s_j) s_i (s_j^{-1} s_i^{-1}) = s_j s_i s_j s_j^{-1} s_i^{-1} = s_j$. If they satisfy an Artin relation of a different length, then the same holds for $s_i s_j$ replaced with the alternating product $s_i s_j ...$ of length $m_{ij} -1$. This follows from applying the Artin relation to the initial subword and then cancelling inverses. We denote the replacement $s_i s_j ... s_i s_j^{-1} s_i^{-1}... = s_j$ with $\Tilde{M}_{ij}$. Notice that this is specifically a conjugacy relation, so it always occurs on the rightmost subword of $w$ and the leftmost subword of $w^{-1}$, and no position indicator is needed.

     \begin{example} 
     Define $\epsilon = s_3 s_1 s_2$ in the Artin group $\cala_4$. To reduce the expression $\epsilon s_1 \epsilon^{-1} := s_3 s_1 s_2 s_1 s_2^{-1} s_1^{-1} s_3^{-1}$, we apply $\Tilde{M}_{12}$. This replaces the interior subword $s_1 s_2 s_1 s_2^{-1} s_1^{-1}$ with $s_2$, yielding $\epsilon s_1 \epsilon^{-1} = s_3 s_2 s_3^{-1}$.
     \end{example}

    \smallskip

    \item \textbf{$S_{i,R/L, \square}$:} Often, we will want to move an Artin generator in $w$ as far to the left or right as possible via commuting relations. While this could be written as several $M$ operations, it is shorter to define a third operation: $S_{i, n, R/L}$ denotes that one should shift the $n$th appearance, from left to right, of Artin generator $s_i$ in $w$ as far to the right or left as possible in $w s_i w^{-1}$ via commuting relations with the adjacent letters. The precise opposite should occur with $w^{-1}$ so that the expressions for $w$ and $w^{-1}$ are still mirrored. If there is a unique instance of $s_i$ in $w$, we omit the position indicator $n$ and simply write $S_{i,R}$ or $S_{i,L}$. As before, the position indicator $n$, when it is needed, is given in Roman numerals to avoid confusion with the index $i$.

    If we are considering an expression of the form $w^{-1} s_i w$, we will write $S_{i^{-1}, R}$ to denote that one should shift the unique instance of $s_i^{-1}$ in $w^{-1}$ as far to the right as possible and the unique instance of $s_i$ in $w$ as far to the left as possible.
    
    \begin{example} Define $\epsilon = s_4 s_1^7$ in $\cala_4$. To simplify the expression $\epsilon s_2 \epsilon^{-1}$, we apply $S_{4 , R}$ to the equation $\epsilon s_2 \epsilon^{-1} := s_4 s_1^7 s_2 s_1^{-7} s_4^{-1}$ to obtain $\epsilon s_2 \epsilon^{-1} := s_1^7 s_4 s_2 s_4^{-1} s_1^{-7}$. If we were instead interested in $\epsilon^{-1} s_2 \epsilon$, we could instead apply $S_{4^{-1}, L}$ to $\epsilon^{-1} s_2 \epsilon := s_1^{-7} s_4^{-1} s_2 s_4 s_1^7$ to obtain $\epsilon s_2 \epsilon^{-1} := s_4^{-1} s_1^{-7} s_2 s_1^7 s_4$.
    \end{example}

    \smallskip

    \item \textbf{$X_{v}$:} Let $w = uv$ where $v$ is the longest suffix, i.e., subword read from right to left, of $w$ which commutes with $s_j$. We write $X_{v}$ to denote the replacement $w s_j w^{-1} = u s_j u^{-1}$. As with the shift operation $S$ and with $M$, the operation $X_{v}$ could be accomplished via successive applications of $\Tilde{M}_{ij}$ for each letter in $v$, but $X$ is significantly shorter. As with $\Tilde{M}_{ij}$, no position indicator is needed, since this always occurs in the middle of the expression.
    %If we are working with an expression $w s_j w^{-1}$ and $w = u v$ for some subwords $u$ and $v$ such that $v$ commutes with $s_j$,

    \begin{example} 
    Define $\epsilon = s_2 s_5 s_6$ in $\cala_6$. To simplify $\epsilon s_1 \epsilon^{-1} := s_2 s_5 s_6 s_1 s_6^{-1} s_5^{-1} s_2^{-1}$, we apply $X_{56}$ to denote the replacement $s_5 s_6 s_1 s_6^{-1} s_5^{-1} = s_1$, yielding $\epsilon s_1 \epsilon^{-1} := s_2 s_1 s_2^{-1}$.
    \end{example}
    %\yv{Would the notation $X_{j,v}$ be a bit clearer so that people know which generator $v$ commutes with?}
    %\katie{I'm worried it's going to be confusing if you have a long word $v$, since the comma in the subscript is quite small}
    %\yv{Is it right to interpret this as: $w = uv$ where $v$ is the longest subword of $w$ when read from right to left which commutes with $s_j$? That way there is no room for interpretation.}
    %\katie{I think that's true in every case that appears, I'll change it to that}
 \end{itemize}

We will also occasionally modify both sides of an equation of the form $w s_i w^{-1} = v s_j v^{-1}$ at the same time. These are \emph{unsided} operations.

\begin{itemize}
    \item $C_v$ denotes conjugating both sides of the equation by $v$. Here we say that conjugating $a$ by $v$ is replacing $a$ with $v a v^{-1}$. 

    \begin{example} 
    Given the equation $s_1 = s_4 s_1 s_4^{-1}$ in $\cala_5$, the operation $C_{4^{-1}}$ denotes replacing $s_1 = s_4 s_1 s_4^{-1}$ with $s_4^{-1} s_1 s_4 = s_1$.
    \end{example}
    
    \item $R_v$ denotes right-multiplying both sides of the equation by $v$.
    $L_v$ denotes left-multiplying both sides of the equation by $v$.
\end{itemize}
In all of these, $v$ is typically represented by its associated string of integers.

We give an example which combines several steps, and then the remainder of our computations are provided only via the shorthand notation. 

\begin{example} In the Artin group $\cale_6$, one can show that $(s_4 s_2 s_3 s_1 s_5 s_6) s_5  (s_6^{-1} s_5^{-1} s_1^{-1} s_3^{-1} s_2^{-1} s_4^{-1}) = s_6$ via the following sequence. 
\begin{center}
    $^{LHS} \Tilde{M}_{56} \rightarrow X_{4231}$
\end{center}
The first step, $\Tilde{M}_{56}$, is applied to the left-hand side, and replaces $(s_4 s_2 s_3 s_1 s_5 s_6) s_5  (s_6^{-1} s_5^{-1} s_1^{-1} s_3^{-1} s_2^{-1} s_4^{-1})$ with $(s_4 s_2 s_3 s_1) s_6  (s_1^{-1} s_3^{-1} s_2^{-1} s_4^{-1})$. 

The Artin generators $s_1$, $s_2$, $s_3$, and $s_4$ all commute with $s_6$, so the entire subword $(s_4 s_2 s_3 s_1)$ commutes with $s_6$. The second step, $X_{4231}$, denotes replacing $(s_4 s_2 s_3 s_1) s_6  (s_1^{-1} s_3^{-1} s_2^{-1} s_4^{-1})$ with $s_6$.
\end{example}

%\yv{Since we're defining $v$ as $w = uv$, shouldn't it really be $X_{\bar{4231}}$ where I'm using the bar to denote the inverse? Since you didn't define either $X$, $L$, or $C$ using numbers, I really don't know how to parse this.}
%\katie{It could be written that way to denote that we're replacing $w$ with $wv^{-1}$, but the point is just that you have an interior subword of the form $v s_i v^{-1}$ where $v$ commutes with $s_i$, and the $v$, $v^{-1}$ pair cancels itself out. I'm hoping this is clearer now with an example given earlier on.}

\begin{example} 
In the Artin group $\cale_6$, one reduces the expression 
 \begin{align*}
     \epsilon s_1 \epsilon^{-1} &:= (s_4 s_2 s_3 s_1 s_5 s_6) (s_4 s_2 s_3) s_1 (s_3^{-1} s_2^{-1} s_4^{-1})(s_6^{-1} s_5^{-1} s_1^{-1} s_3^{-1} s_2^{-1} s_4^{-1})
 \end{align*}
 as follows:
 %\yv{How did $s_3$ become $s_1$? I can't even get the equality in the first line here by just applying relations. I found the $\epsilon$'s below, and I'm assuming that $\epsilon^3$ means that your cubing epsilon, but maybe that's wrong? I also really can't parse the $X$ notation consistently, which means I can't actually check the computations. I think one place to start is to add in the extra subscript for $X$ so that I know where to look.} 
%\katie{I chose a bad example here; I had copied it from somewhere else in the code without realizing that I was implicitly using the fact that I had already found a reduced expression for $\epsilon_{12}^2 s_3 \epsilon_{12}^{-1}$. I think I've fixed it to be self-contained now.}

\begin{align*}
\epsilon  s_3 \epsilon^{-1} &= (s_4 s_2 s_3 s_1 s_5 s_6 s_4 s_2 s_3) s_1 (s_3^{-1} s_2^{-1} s_4^{-1}s_6^{-1} s_5^{-1} s_1^{-1} s_3^{-1} s_2^{-1} s_4^{-1})\\
        &= (s_4 s_2 s_3 s_1 s_5 s_4 s_2 s_3 s_6) s_1 (s_6^{-1} s_3^{-1} s_2^{-1} s_4^{-1} s_5^{-1} s_1^{-1} s_3^{-1} s_2^{-1} s_4^{-1})\\
        &= (s_4 s_2 s_3 s_1 s_5 s_4 s_2 s_3) s_1 (s_3^{-1} s_2^{-1} s_4^{-1} s_5^{-1} s_1^{-1} s_3^{-1} s_2^{-1} s_4^{-1})\\
        &= (s_4 s_2 s_3 s_5 s_4 s_2 s_1 s_3) s_1 (s_3^{-1} s_1^{-1}s_2^{-1} s_4^{-1} s_5^{-1} s_3^{-1} s_2^{-1} s_4^{-1})\\
        %&= (s_4 s_2 s_3 s_5 s_4 s_2 s_1 s_3) s_1 (s_3^{-1} s_1^{-1}s_2^{-1} s_4^{-1} s_5^{-1} s_3^{-1} s_2^{-1} s_4^{-1})\\
        &= (s_4 s_2 s_3 s_5 s_4 s_2) s_3 (s_2^{-1} s_4^{-1} s_5^{-1} s_3^{-1} s_2^{-1} s_4^{-1})\\
        &= (s_4 s_2 s_3 s_5 s_4) s_3 (s_4^{-1} s_5^{-1} s_3^{-1} s_2^{-1} s_4^{-1})\\
        &= (s_4 s_2 s_5 s_3 s_4) s_3 (s_4^{-1} s_3^{-1} s_5^{-1} s_2^{-1} s_4^{-1})\\
        &= (s_4 s_2 s_5) s_4 (s_5^{-1} s_2^{-1} s_4^{-1})\\
\end{align*}
In the shorthand, this is
$\epsilon s_3 \epsilon^{-1} := (s_4 s_2 s_3 s_1 s_5 s_6 s_4 s_2 s_3) s_1 (s_3^{-1} s_2^{-1} s_4^{-1} s_6^{-1} s_5^{-1} s_1^{-1} s_3^{-1} s_2^{-1} s_4^{-1})$
\begin{center} 
$S_{6,R} \rightarrow X_{6} \rightarrow S_{1, R} \rightarrow \Tilde{M}_{13} \rightarrow X_{2} 
\rightarrow S_{3,R} \rightarrow \Tilde{M}_{34}$ 
\end{center}
$\epsilon s_3\epsilon^{-1} = (s_4 s_2 s_5) s_4 (s_5^{-1} s_2^{-1} s_4^{-1})$.
\end{example}

We now proceed to the computations used in the paper.

\subsection{The group \texorpdfstring{$\cale_6$}{E6}}

Recall the following definitions of elements and their orders in $\cale_6 / Z(\cale_6)$, as given in \cite{Soroko2021}. 

\begin{itemize}
    \item $\epsilon_{12} = s_4 s_2 s_3 s_1 s_5 s_6$ has order 12.
    \item $\epsilon_9 = s_4 s_2 s_5 s_4 s_3 s_1 s_5 s_6$ has order 9.
    \item $\epsilon_8 = s_4 s_3 s_1 s_5 s_4 s_2 s_3 s_6 s_5$ has order 8.
\end{itemize}

\subsubsection{\texorpdfstring{The element $\epsilon_{12}$}{Order 12}}\label{apx:E6-order12}

We begin with $\epsilon_{12}$. In each row, it should be understood that when $|i| \neq 1$, the beginning expression for $\epsilon_{12}^i \cdot s_5 \cdot \epsilon_{12}^{-i}$ is the conjugate of the reduced expression for $\epsilon_{12}^{i-1} \cdot s_5 \cdot \epsilon_{12}^{-i+1}$ rather than the conjugate of $s_5$ by $i$ copies of the expression for $\epsilon_{12}$ given above.

To see that $\cale_6 \langle \langle \epsilon_{12}^6 \rangle \rangle$ is cyclic, we compute the following.

\begin{center}
\renewcommand{\arraystretch}{1.4}
    \begin{tabular}{| c | c | c |}
    \hline
    Starting expression & Ending expression & Path \\ 
    \hline
    $\epsilon_{12} \cdot s_3 \cdot \epsilon_{12}^{-1}$ &  $s_1$ & $X_{56} \rightarrow \Tilde{M}_{13} \rightarrow X_{42}$ \\  [2pt]
    \hline
    $\epsilon_{12}^2 \cdot s_3 \cdot \epsilon_{12}^{-2}$ & $(s_4 s_2 s_3) s_1 (s_3^{-1} s_2^{-1} s_4^{-1})$  & $X_{156}$ \\
    \hline
    $\epsilon_{12}^{3} \cdot s_3 \cdot \epsilon_{12}^{-3}$ & $(s_4 s_2 s_5) s_4 (s_5^{-1} s_2^{-1} s_4^{-1})$ & $S_{1,R} \rightarrow \Tilde{M}_{13} \rightarrow X_{2} \rightarrow S_{3,R} \rightarrow \Tilde{M}_{34}$\\
    \hline
    $\epsilon_{12}^{-1} \cdot s_3 \cdot \epsilon_{12}$ & $(s_6^{-1} s_5^{-1} s_2^{-1}) s_4 (s_2 s_5 s_6)$ & $S_{3^{-1}, R} \rightarrow \Tilde{M}_{34} \rightarrow S_{1^{-1}, R} \rightarrow X_{1^{-1}}$\\
    \hline
    $\epsilon_{12}^{-2} \cdot s_3 \cdot \epsilon_{12}^2$ & $(s_5^{-1} s_1^{-1} s_3^{-1}) s_4 (s_3 s_1 s_5)$ & $\begin{aligned}
        S_{6,II,R} \rightarrow M_{65, I} &\rightarrow S_{5^{-1}, II, R} \rightarrow \Tilde{M}_{24} \\\rightarrow X_{5^{-1}} \rightarrow &\Tilde{M}_{24} \rightarrow S_{6^{-1},R} \rightarrow X_{6^{-1}}
    \end{aligned}$\\
    \hline
    $\epsilon_{12}^{-3} \cdot s_3 \cdot \epsilon_{12}^3$ & $( s_1^{-1} s_3^{-1} s_2^{-1} s_4^{-1} s_1^{-1}) s_3 (s_1 s_4 s_2 s_3 s_1)$ & $\begin{aligned}
        S_{5^{-1},I,R} \rightarrow M_{54, I} \rightarrow &S_{4^{-1},II,R} \rightarrow \Tilde{M}_{34} \\\rightarrow S_{5^{-1},R} \rightarrow X_{5^{-1}} \rightarrow &S_{6^{-1}, R} \rightarrow X_{6^{-1}}
    \end{aligned}$\\
    \hline
\end{tabular}
\end{center}

 In any quotient where $\epsilon_{12}^6$ is trivial, we must have
    \begin{align*}
        (s_4 s_2 s_5) s_4 (s_5^{-1} s_2^{-1} s_4^{-1}) = ( s_1^{-1} s_3^{-1} s_2^{-1} s_4^{-1} s_1^{-1}) s_3 (s_1 s_4 s_2 s_3 s_1)
    \end{align*}
    
    This imposes the relation $s_5 = s_3$ via the following sequence of reductions.
    \begin{align*}
        C_{2^{-1}4^{-1}} \rightarrow^{RHS} S_{1^{-1}, I, L} \rightarrow M_{24} \rightarrow M_{34} &\rightarrow \Tilde{M}_{13} \rightarrow X_{4^{-1}} \rightarrow S_{1^{-1}, R} \rightarrow \Tilde{M}_{13} \rightarrow X_{2^{-1}}\\
        \rightarrow C_{4} &\rightarrow^{LHS} \Tilde{M}_{45}
    \end{align*}

    The quotient $\cale_6 / \langle \langle \epsilon_{12}^6 \rangle \rangle$ must then be cyclic by \cref{lem:artin-gen-collapsing}. %Furthermore, the quotients $\cale_6 / \langle \langle \epsilon_{12} \rangle \rangle$, $\cale_6 / \langle \langle \epsilon_{12}^2 \rangle \rangle$, and $\cale_6 / \langle \langle \epsilon_{12}^3 \rangle \rangle$ must also be cyclic because they are further quotients. 

%    In any quotient where $\epsilon_{12}^4$ is trivial, we must have 
%    \begin{align*}
%        (s_4 s_2 s_3) s_1 (s_3^{-1} s_2^{-1} s_4^{-1}) = s_5^{-1} s_1^{-1} s_3^{-1} s_4 (s_3 s_1 s_5)\text{.}
%    \end{align*}
%    This imposes the relation $s_5 s_4 = s_4 s_5$ via the following sequence of reductions.
%    \begin{align*}
%        C_{1} \rightarrow^{LHS} S_{1, R} \rightarrow \Tilde{M}_{13} \rightarrow X_2 \rightarrow^{RHS} S_{1, R} \rightarrow C_3 \rightarrow^{LHS} \Tilde{M}_{34} \rightarrow^{RHS} S_{3,R} \rightarrow L_5
%    \end{align*}
%    This implies that the images of $s_4$ and $s_5$ satisfy coprime length Artin relations in the quotient. \cref{L:gcdRelation} implies that they must have the same image in the quotient, and we conclude by \cref{lem:artin-gen-collapsing}.
    \hfill $\blackdiamond$

\subsubsection{\texorpdfstring{The element $\epsilon_{9}$}{Order 9}}\label{ap:E6-order9}

    Using the commuting relations between Artin generators, $\epsilon_9$ can be rearranged to the form $\epsilon_9 = s_4 s_2 s_5 s_4 s_5 s_6 s_3 s_1$. We compute the following using this form of $\epsilon_9$. In each row, it should be understood that when $|i| \neq 1$, the beginning expression for $\epsilon_{9}^i \cdot s_2 \cdot \epsilon_{9}^{-i}$ is the conjugate of the reduced expression for $\epsilon_{9}^{i-1} \cdot s_2 \cdot \epsilon_{9}^{-i+1}$ rather than the conjugate of $s_2$ by $i$ copies of the expression for $\epsilon_{9}$ given above.

    \begin{center}
    \renewcommand{\arraystretch}{1.4}
    \begin{tabular}{| c | c | c |}
    \hline
    Starting expression & Ending expression & Path \\ 
    \hline
    $\epsilon_9 s_2 \epsilon_9^{-1}$ & $s_5$  & $X_{5631} \rightarrow S_{2,R} \rightarrow \Tilde{M}_{24} \rightarrow \Tilde{M}_{45}$ \\  
    \hline
    $\epsilon_{9}^2 \cdot s_2 \cdot \epsilon_{9}^{-2}$ & $(s_4 s_5) s_6 (s_5^{-1} s_4^{-1})$ & $X_{31} \rightarrow \Tilde{M}_{56} \rightarrow S_{2,R} \rightarrow X_{24}$ \\
    \hline
    $\epsilon_{9}^{-1} \cdot s_2 \cdot \epsilon_{9}$ & $s_6^{-1} s_5 s_6$ &  $\Tilde{M}_{24} \rightarrow \Tilde{M}_{45} \rightarrow X_{5^{-1}} \rightarrow S_{6^{-1},L} \rightarrow X_{1^{-1}3^{-1}}$\\
    \hline
\end{tabular}
\end{center}

     In a quotient where $\epsilon_9^{3}$ is trivial, we must have $(s_4 s_5) s_6 (s_5^{-1} s_4^{-1}) = s_6^{-1} s_5 s_6$. This imposes the relation $s_4 s_5 = s_5 s_4$ via the following sequence of reductions.
     \begin{align*}
         C_{6} \rightarrow^{LHS} S_{4, L} \rightarrow \Tilde{M}_{56} \rightarrow R_{4}
     \end{align*}

     The relation $s_4 s_5 = s_5 s_4$ implies that $\cale_6 / \langle \langle \epsilon_{9}^3 \rangle \rangle$ is cyclic by \cref{L:gcdRelation} and \cref{lem:artin-gen-collapsing}. %This also shows that $\cale_6 / \langle \langle \epsilon_{9} \rangle \rangle$ is cyclic, since it is a further quotient.
\hfill $\blackdiamond$

\subsubsection{\texorpdfstring{The element $\epsilon_{8}$}{Order 8}}\label{ap:E6-order8}
    For $\epsilon_8$, we compute the following. In each row, it should be understood that when $|i| \neq 1$, the beginning expression for $\epsilon_{8}^i \cdot s_6 \cdot \epsilon_{8}^{-i}$ is the conjugate of the reduced expression for $\epsilon_{8}^{i-1} \cdot s_6 \cdot \epsilon_{8}^{-i+1}$ rather than the conjugate of $s_6$ by $i$ copies of the expression for $\epsilon_{8}$ given at the beginning of the section.

    \begin{center}
    \renewcommand{\arraystretch}{1.4}
    \begin{tabular}{| c | c | c |}
    \hline
    Starting expression & Ending expression & Path \\ 
    \hline
    $\epsilon_8 s_6 \epsilon_8^{-1}$ & $s_3$  & $\Tilde{M}_{56} \rightarrow X_{23} \rightarrow \Tilde{M}_{45} \rightarrow X_{1} \rightarrow \Tilde{M}_{34}$ \\  
    \hline
    $\epsilon_{8}^2 \cdot s_6 \cdot \epsilon_{8}^{-2}$ & $(s_4 s_3 s_1 s_5 s_4 ) s_3 (s_4^{-1} s_5^{-1} s_1^{-1} s_3^{-1} s_4^{-1})$  & $X_{2365}$ \\
    \hline
    $\epsilon_{8}^{-1} \cdot s_6 \cdot \epsilon_{8}$ & $(s_3^{-1} s_2^{-1}) s_4 ( s_2 s_3 )$ & $\begin{aligned}
        X_{1^{-1}3^{-1}4^{-1}} \rightarrow &S_{6^{-1},R} \rightarrow \Tilde{M}_{56} \rightarrow S_{5^{-1},R} \\
        &\rightarrow \Tilde{M}_{45}
    \end{aligned}$ \\
    \hline
    $\epsilon_8^{-2} \cdot s_6 \cdot \epsilon_8^{2}$ & $(s_5^{-1} s_6^{-1} s_3^{-1} s_2^{-1} s_4^{-1} s_5^{-1} s_4^{-1}) s_2 (s_4 s_5 s_4 s_2 s_3 s_6 s_5)$ & $M_{34} \rightarrow \Tilde{M}_{24} \rightarrow S_{1^{-1},R} \rightarrow X_{1^{-1}3^{-1}} $\\
    \hline
\end{tabular}
\end{center}

    In a quotient where $\epsilon_8^4$ is trivial, we must have the following relation.
    \begin{align*}
        (s_4 s_3 s_1 s_5 s_4 ) s_3 (s_4^{-1} s_5^{-1} s_1^{-1} s_3^{-1} s_4^{-1}) = (s_5^{-1} s_6^{-1} s_3^{-1} s_2^{-1} s_4^{-1} s_5^{-1} s_4^{-1}) s_2 (s_4 s_5 s_4 s_2 s_3 s_6 s_5)\text{.}
    \end{align*}
    This implies $s_6 = s_2$ via the following sequence of reductions. 
    \begin{align*}
        \text{ }^{RHS} S_{3^{-1},L} \rightarrow C_{3} \rightarrow^{LHS} M_{34} \rightarrow S_{4,II,R} \rightarrow M_{45} &\rightarrow X_5 \rightarrow S_{3,R} \rightarrow \Tilde{M}_{34} \rightarrow S_{4,R} \rightarrow \Tilde{M}_{45} \rightarrow X_{1} \\
        \rightarrow^{RHS} C_5 \rightarrow S_{2^{-1},L} \rightarrow C_{62} \rightarrow^{LHS} &X_2 \rightarrow^{RHS} M_{45} \rightarrow C_4 \rightarrow^{LHS} S_{4,R} \rightarrow \\X_4 \rightarrow C_5 \rightarrow^{LHS} &\Tilde{M}_{56} \rightarrow C_4 \rightarrow^{LHS} X_4 
    \end{align*}

    We conclude that $\cale_6 / \langle \langle \epsilon_{8}^4 \rangle \rangle$ is cyclic by \cref{lem:artin-gen-collapsing}. Consequently, $\cale_6 / \langle \langle \epsilon_{8}^2 \rangle \rangle$ and $\cale_6 / \langle \langle \epsilon_{8} \rangle \rangle$ are also cyclic.
    \hfill $\blackdiamond$
    
\subsection{The group \texorpdfstring{$\cale_7$}{E7}}
    Recall the following definitions of elements and their orders in $\cale_7 / Z(\cale_7)$, as given in \cite{Soroko2021}.

    \begin{itemize}
        \item $\epsilon_7 = s_4 s_2 s_7s_6 s_5 s_4 s_2 s_3 s_1$ has order $7$.
        \item $\epsilon_9= s_4 s_2 s_3 s_1 s_5 s_6 s_7$ has order $9$.
    \end{itemize}

    \subsubsection{\texorpdfstring{The element $\epsilon_7$}{Order 7}}\label{apx:E7-order7}

    The expression $\epsilon_7 s_3 \epsilon_7^{-1}$ can be reduced to $s_1$ via the following sequence of operations.
    \begin{align*}
        \Tilde{M}_{13} \rightarrow X_{4276542}
    \end{align*}

    We conclude that the quotient $\cale_7/\langle\langle\epsilon_7\rangle\rangle$ is cyclic by \cref{lem:artin-gen-collapsing}.
    \hfill $\blackdiamond$

    \subsubsection{\texorpdfstring{The element $\epsilon_9$}{Order 9}}\label{apx:E7-order9}

    Applying a rearrangement of commuting terms, $\epsilon_9 = s_4 s_5 s_6 s_7 s_2 s_3 s_1$. Using this expression, we compute the following. In each row, it should be understood that when $|i| \neq 1$, the beginning expression for $\epsilon_{9}^i \cdot s_6 \cdot \epsilon_{9}^{-i}$ is the conjugate of the reduced expression for $\epsilon_{9}^{i-1} \cdot s_6 \cdot \epsilon_{9}^{-i+1}$ rather than the conjugate of $s_6$ by $i$ copies of the expression for $\epsilon_{9}$ given above.

        \begin{center}
    {
    \renewcommand{\arraystretch}{1.4}
    \begin{tabular}{| c | c | c |}
    \hline
    Starting expression & Ending expression & Path \\ 
    \hline 
    $\epsilon_9 \cdot s_6 \cdot \epsilon_9^{-1}$ &  $s_7$ & $X_{231} \rightarrow \Tilde{M}_{67} \rightarrow X_{45}$ \\  [2pt]
    \hline
    $\epsilon_{9}^2 \cdot s_6 \cdot \epsilon_{9}^{-2}$ & $(s_4 s_5 s_6) s_7 (s_6^{-1} s_5^{-1} s_4^{-1})$  & $X_{7231}$ \\
    \hline
    $\epsilon_{9}^{-1} \cdot s_6 \cdot \epsilon_{9}$ & $s_5$ & $X_4 \rightarrow \Tilde{M}_{56} \rightarrow X_{1327}$ \\
    \hline
\end{tabular}}
\end{center}

    In any quotient where $\epsilon_9^3$ is trivial, we have the following relation.
    \begin{align*}
        s_4 s_5 s_6 s_7 s_6^{-1} s_5^{-1} s_4^{-1} = s_5
    \end{align*}
    This imposes the relation $s_7 = s_4$ via the following sequence of reductions.
    \begin{align*}
        C_{5^{-1}4^{-1}} \rightarrow^{RHS} \Tilde{M}_{45} \rightarrow C_{6^{-1}} \rightarrow^{RHS} X_{6^{-1}}
    \end{align*}
    As in the previous cases, we can conclude that $\cale_7 / \langle \langle \epsilon_9^3 \rangle \rangle$ is cyclic by \cref{lem:artin-gen-collapsing}
    \hfill $\blackdiamond$

\subsection{The group \texorpdfstring{$\calh_4$}{H4}}
Recall the following definitions of elements and their orders in $\calh_4 / Z(\calh_4)$, as given in \cite{Soroko2021}.
\begin{itemize}
    \item $\epsilon_{10} = s_1 s_2 s_1 s_2 s_3 s_4 $ has order $10$.
    \item $\epsilon_{15} = s_1 s_2 s_3 s_4$ has order $15$.
\end{itemize}

Before we proceed to the computations of the quotients, we highlight that because $s_1$ and $s_2$ satisfy an Artin relation of length 5 rather than length 3, it is no longer true that $(s_1 s_2) s_1 (s_2^{-1} s_1^{-1}) = s_1$. In $\calh_4$, we instead have the relation $(s_1 s_2 s_1 s_2) s_1 (s_2^{-1} s_1^{-1} s_2^{-1} s_1^{-1}) = s_1$. We also highlight that, because we are working in a graph with no valence 3 vertices, the set of indices which commute with one another is different than in the $\cale_n$ cases; see Figure \ref{fig:FiniteCoxeterGroups}.

\subsubsection{\texorpdfstring{The element $\epsilon_{10}$}{Order 10}}\label{apx:H4-order10}
We begin with $\epsilon_{10}$. We compute the following. In each row, it should be understood that when $|i| \neq 1$, the beginning expression for $\epsilon_{10}^i \cdot s_1 \cdot \epsilon_{10}^{-i}$ is the conjugate of the reduced expression for $\epsilon_{10}^{i-1} \cdot s_1 \cdot \epsilon_{10}^{-i+1}$ rather than the conjugate of $s_1$ by $i$ copies of the expression for $\epsilon_{10}$ given above.

\begin{center}
    {
    \renewcommand{\arraystretch}{1.4}
    \begin{tabular}{| c | c | c |}
    \hline
    Starting expression & Ending expression & Path \\ 
    \hline 
    $\epsilon_{10} s_1 \epsilon_{10}^{-1}$ & $s_2$  & $X_{34} \rightarrow \Tilde{M}_{12}$ \\  [2pt]
    \hline
     $\epsilon_{10}^{2} s_1 \epsilon_{10}^{-2}$ &  $(s_1 s_2) s_3 (s_2^{-1} s_1^{-1})$ & $X_4 \rightarrow \Tilde{M}_{23} \rightarrow X_1 $ \\
    \hline
     $\epsilon_{10}^{3} s_1 \epsilon_{10}^{-3}$ & $(s_2 s_1 s_2 s_3) s_4 (s_3^{-1} s_2^{-1} s_1^{-1} s_2^{-1})$ & $\begin{aligned}
         S_{1,III, L} \rightarrow &M_{12} \rightarrow M_{23} \rightarrow \\ \Tilde{M}_{34} \rightarrow &S_{3,L} \rightarrow X_{12} 
     \end{aligned}$ \\
    \hline
    $\epsilon_{10}^{4} s_1 \epsilon_{10}^{-4}$ &  $(s_1 s_2 s_1 s_3 s_2 s_1) s_2 (s_1^{-1} s_2^{-1} s_3^{-1} s_1^{-1} s_2^{-1} s_1^{-1})$  & $S_{4,R} \rightarrow \Tilde{M}_{34} \rightarrow M_{23} \rightarrow S_{3,II,R} \rightarrow \Tilde{M}_{23}$\\
    \hline
    $\epsilon_{10}^{-1} s_1 \epsilon_{10}$ & $(s_4^{-1} s_3^{-1} s_2^{-1} s_1^{-1} s_2^{-1}) s_1 (s_2 s_1 s_2 s_3 s_4)$ & $X_{1}$\\
    \hline
\end{tabular}}
\end{center}
 
In any quotient where $\epsilon_{10}^2$ is trivial, the following relation is satisfied.
\begin{align*}
    s_1 = (s_1 s_2) s_3 (s_2^{-1} s_1^{-1})
\end{align*}

Since $s_1$ commutes with $s_4$, the image of the right hand side must also commute with the image of $s_4$ in the quotient, yielding the following.
\begin{align*}
    s_4 = (s_1 s_2) s_3 (s_2^{-1} s_1^{-1}) s_4 (s_1 s_2) s_3^{-1} (s_2^{-1} s_1^{-1})
\end{align*}
This imposes the relation $s_4 s_3 = s_3 s_4$ via the following sequence of reductions.
\begin{align*}
    C_{2^{-1}{1^{-1}}} \rightarrow^{LHS} X_{2^{-1}1^{-1}} \rightarrow^{RHS} X_{2^{-1}1^{-1}} \rightarrow R_{3} 
\end{align*}

As in the previous cases, we conclude that the quotients $\calh_4 / \langle \langle \epsilon_{10} \rangle \rangle$ and $\calh_4 / \langle \langle \epsilon_{10}^2 \rangle \rangle$  are cyclic by Lemmas \ref{L:gcdRelation} and \ref{lem:artin-gen-collapsing}.
\hfill $\blackdiamond$

In any quotient where $\epsilon_{10}^5$ is trivial, the following relation is satisfied.
\begin{align*}
    (s_1 s_2 s_1 s_3 s_2 s_1) s_2 (s_1^{-1} s_2^{-1} s_3^{-1} s_1^{-1} s_2^{-1} s_1^{-1}) = (s_4^{-1} s_3^{-1} s_2^{-1} s_1^{-1} s_2^{-1}) s_1 (s_2 s_1 s_2 s_3 s_4)
\end{align*}
This imposes the relation $s_2 = s_4$ via the following sequence of reductions.
\begin{align*}
    C_{1}^{-1} \rightarrow^{RHS} S_{1^{-1},I,R} \rightarrow \Tilde{M}_{12} \rightarrow C_{2^{-1}} &\rightarrow^{RHS} S_{2,R} \rightarrow \Tilde{M}_{23} \rightarrow C_{3^{-1} 1^{-1}} \rightarrow^{RHS} S_{1^{-1},R} \rightarrow X_{1^{-1}} \\\rightarrow \Tilde{M}_{34}
    \rightarrow &C_{1^{-1}2^{-1}} \rightarrow^{RHS} X_{1^{-1}2^{-1}}
\end{align*}

We conclude that $\calh_4 / \langle \langle \epsilon_{10}^5 \rangle \rangle$ is cyclic by \cref{lem:artin-gen-collapsing}.
\hfill $\blackdiamond$

\subsubsection{\texorpdfstring{The element $\epsilon_{15}$}{Order 15}}\label{apx:H4-order15}

    For $\epsilon_{15}$, we compute the following. In each row, it should be understood that when $|i| \neq 1$, the beginning expression for $\epsilon_{15}^i \cdot s_2 \cdot \epsilon_{15}^{-i}$ is the conjugate of the reduced expression for $\epsilon_{15}^{i-1} \cdot s_2 \cdot \epsilon_{15}^{-i+1}$ rather than the conjugate of $s_2$ by $i$ copies of the expression for $\epsilon_{15}$ given above.
    
    \begin{center}
    {
    \renewcommand{\arraystretch}{1.4}
    \begin{tabular}{| c | c | c |}
    \hline
    Starting expression & Ending expression & Path \\ 
    \hline 
    $\epsilon_{15} \cdot s_2 \cdot \epsilon_{15}^{-1}$ & $s_3$  & $X_4 \rightarrow \Tilde{M}_{23} \rightarrow X_1$ \\  [2pt]
    \hline
    $\epsilon_{15}^{2} \cdot s_2 \cdot \epsilon_{15}^{-2}$ & $s_4$  & $\Tilde{M}_{34} \rightarrow X_{12}$ \\
    \hline
    $\epsilon_{15}^{3} \cdot s_2 \cdot \epsilon_{15}^{-3}$ & $(s_1 s_2 s_3) s_4 (s_3^{-1} s_2^{-1} s_1^{-1})$ & $X_4$ \\
    \hline
\end{tabular}}
\end{center}

In any quotient where $\epsilon_{15}^3$, the following relation is satisfied.
\begin{align*}
    s_2 = (s_1 s_2 s_3) s_4 (s_3^{-1} s_2^{-1} s_1^{-1})
\end{align*}
This imposes the relation $s_3 = ( s_2 s_1) s_2 ( s_1^{-1} s_2^{-1})$ via the following sequence of reductions.
\begin{align*}
    C_4 \rightarrow^{LHS} X_4 \rightarrow^{RHS} S_{4,R} \rightarrow \Tilde{M}_{34} \rightarrow C_{23} \rightarrow^{LHS} \Tilde{M}_{23} \rightarrow^{RHS} S_{3,R} \rightarrow \Tilde{M}_{23}
\end{align*}
The right hand side commutes with the image of $s_4$, so the image of $s_3$ commutes with the image of $s_4$. Since $s_3$ and $s_4$ also satisfy a braid relation, they are identified in the quotient by \ref{L:gcdRelation}. The quotient is then cyclic by \cref{lem:artin-gen-collapsing}.

In any quotient where $\epsilon_{15}^5$ is trivial, we must have $\epsilon_{15}^3 s_3 \epsilon_{15}^{-3} = \epsilon_{15}^{-2} s_3 \epsilon_{15}^2$. Applying our previous calculation of $\epsilon_{15}^3 s_3 \epsilon_{15}^{-3}$ and the unreduced form of $\epsilon_{15}^{-2} s_3 \epsilon_{15}^2$, the following relation is satisfied in the quotient.
\begin{align*}
    (s_1 s_2 s_3) s_4 (s_3^{-1} s_2^{-1} s_1^{-1}) = (s_4^{-1} s_3^{-1} s_2^{-1} s_1^{-1} s_4^{-1} s_3^{-1} s_2^{-1} s_1^{-1}) s_3 (s_1 s_2 s_3 s_4 s_1 s_2 s_3 s_4 )
\end{align*}
This imposes the relation $s_1 = s_2$ via the following sequence of reductions.
\begin{align*}
    C_4 \rightarrow^{LHS} S_{4,R} \Tilde{M}_{34} \rightarrow C_{3} \rightarrow^{LHS} S_{3,R} \rightarrow \Tilde{M}_{23} \rightarrow C_{412} \rightarrow^{LHS} S_{4,R} \rightarrow X_{4} \rightarrow^{RHS} X_{1} \rightarrow \Tilde{23} \rightarrow C_2 \rightarrow^{LHS} \Tilde{M}_{12}
\end{align*}

\cref{lem:artin-gen-collapsing} then shows that $\calh_4 / \langle \langle \epsilon_{15}^5 \rangle \rangle$ must be cyclic.
\hfill $\blackdiamond$

    \section{MAGMA code}

    \subsection{Code for Theorem~\ref{thm:H4}}\label{Code:6.2}
    \begin{verbatim}
H4:=Group<a,b,c,d | a*b*a*b*a=b*a*b*a*b, b*c*b=c*b*c, c*d*c=d*c*d, 
                    a*c=c*a, a*d=d*a, b*d=d*b>;
WH4:=PermutationGroup(quo<H4 | H4.1^2>);
Q:=WH4/centre(WH4);

Homomorphisms(H4, Q : Surjective:=true);
    \end{verbatim}
The code returns two homomorphisms both sending the generators to elements of order 2.  Hence the quotient factors through the Coxeter group as required.

\subsection{Code for Lemma~\ref{lem:E8_7s_9s}}\label{code:E8_7s_9s}
    \begin{verbatim}
G:=PermutationGroup(AutomorphismGroup(Group("PSL(3,13)")));
CCL:=ConjugacyClasses(G);
for x in CCL do x[1]; end for;
    \end{verbatim}
This code prints the orders of elements in $\Aut(\PSL_3(13))$.  The output shows there are no elements of order 9.

\subsection{Code for Lemma~\ref{lem:B4tilde_F4tilde}}\label{App:lem_B4tild_F4tilde}
\begin{verbatim}
F4:=Group<a,b,c,d | a*b*a=b*a*b,b*c*b*c=c*b*c*b, c*d*c=d*c*d,
                    a*c=c*a, a*d=d*a, b*d=d*b, a^2, c^2>;
WF4:=PermutationGroup(F4);
B4tilde:=Group<a,b,c,d,e | a*b*a*b=b*a*b*a, a*c=c*a, a*d=d*a, a*e=e*a,
                            b*c*b=c*b*c, b*d=d*b, b*e=e*d, c*d*c=d*c*d, 
                            c*e*c=e*c*e, d*e=e*d >;

Homomorphisms(B4tilde, WF4 : Surjective:=true, Limit:=1);
\end{verbatim}

\subsection{Code for Lemma~\ref{lem:goodness}}\label{Code:11.1}
    \begin{verbatim}
F4:=Group<a,b,c,d | a*b*a=b*a*b,b*c*b*c=c*b*c*b, c*d*c=d*c*d, 
                    a*c=c*a, a*d=d*a, b*d=d*b, a^2, c^2>;
WF4:=PermutationGroup(F4);
B3:=Group<a,b,c | a*b*a*b=b*a*b*a, a*c=c*a, b*c*b=c*b*c >;
B4:=Group<a,b,c,d | a*b*a*b=b*a*b*a, a*c=c*a, a*d=d*a, 
                    b*c*b=c*b*c, b*d=d*b, c*d*c=d*c*d >;

Homomorphisms(B3, WF4 : Surjective:=true, Limit:=1);
Homomorphisms(B4, WF4 : Surjective:=true, Limit:=1);
    \end{verbatim}
The code returns the empty set in both cases.

\end{appendix}

\bibliographystyle{halpha}
\bibliography{refs.bib}

\end{document}